\documentclass[twoside,11pt]{article}

\usepackage{myjmlr2e}

\usepackage{comment}
\usepackage[utf8]{inputenc} 
\usepackage[T1]{fontenc}    
\usepackage{hyperref}       
\usepackage{url}            
\usepackage{booktabs}       
\usepackage{amsfonts}       
\usepackage{nicefrac}       
\usepackage{microtype}      
\usepackage{amsfonts}
\usepackage{subfigure}
\usepackage{amsmath}
\usepackage{amssymb}

\usepackage{amsthm}
\usepackage[draft]{todonotes} 
\usepackage{enumerate}
\usepackage{bbm}
\usepackage{bm}
\usepackage{tabularx}
\usepackage{dsfont}
\usepackage{graphicx}
\newcolumntype{L}{>{\raggedright\arraybackslash}X}
\usepackage{centernot}
\usepackage{algorithm}
\usepackage{algpseudocode}
\usepackage{algorithmicx}
\usepackage{tikz}
\usetikzlibrary{chains,fit,shapes}
\algnewcommand\algorithmicinput{\textbf{Initialization:}}
\algnewcommand\Init{\item[\algorithmicinput]}
\algnewcommand\Adv{\item[\textbf{Adversary initialization:}]}
\algnewcommand\Given{\item[\textbf{Given:}]}
\algnewcommand\Input{\item[\textbf{Input:}]}

\newcommand{\R}{\mathbb{R}}

\newcommand{\sel}{{\texttt{SEL}}}

\newcommand{\ncbc}{\mathrm{NCBC}}
\newcommand{\st}[1]{\mathsf{#1}}
\newcommand{\ep}{\varepsilon}
\newcommand{\cl}[1]{\mathcal{#1}}

\newcommand{\dist}{\mathrm{dist}}

\newcommand{\dm}{\mathrm{diam}}
\newcommand{\ol}[1]{\overline{#1}}
\newcommand{\ul}[1]{\underline{#1}}

\newcommand{\Sn}{\mathbb{S}}

\newcommand{\thmref}[1]{Theorem \ref{#1}}

\newcommand{\lemref}[1]{Lemma \ref{#1}}

\newcommand{\corref}[1]{Corollary \ref{#1}}

\newcommand{\defref}[1]{Definition \ref{#1}}

\newcommand{\figref}[1]{Figure \ref{#1}}
\newcommand{\secref}[1]{Section \ref{#1}}

\newcommand{\aspref}[1]{Assumption \ref{#1}}

\DeclareMathOperator*{\argmin}{arg\,min}

\newcommand{\norm}[1]{\left\lVert#1\right\rVert}
\newcommand{\anon}{\:\cdot\;}
\theoremstyle{plain}
\newtheorem{thm}{Theorem}[section]
\newtheorem*{thm*}{Theorem}

\newtheorem*{thmset*}{Setup}
\newtheorem{lem}[thm]{Lemma} 
\newtheorem*{lem*}{Lemma}
\newtheorem{coro}[thm]{Corollary}
\newtheorem*{coro*}{Corollary}
\newtheorem{prop}[thm]{Proposition}
\newtheorem*{prop*}{Proposition}

\theoremstyle{definition}
\newtheorem{defn}[thm]{Definition} 
\newtheorem*{defn*}{Definition} 

\newtheorem*{notset*}{Notation}
\newtheorem{asp}[thm]{Assumption}
\newtheorem{ex}{Example}
\newtheorem*{ex*}{Example}
\newtheorem{rem}{Remark}

\theoremstyle{break}
 
\newtheorem*{breakdefn*}{Definition} 
\newcommand{\partpl}{\cl{T}}
\newcommand{\parttupl}{(\mathbb{T},\st{K},d)}

\newcommand{\tpiv}[1]{M^\pi_{#1}} 

\newcommand\Tstrut{\rule{0pt}{2.6ex}}       
\newcommand\Bstrut{\rule[-0.9ex]{0pt}{0pt}} 
\newcommand{\TBstrut}{\Tstrut\Bstrut} 

\let\cite\citep
\usepackage[inline]{enumitem}

\ShortHeadings{Online Robust Control in Large Uncertainty}{Ho, Le, Yue and Doyle}
\firstpageno{1}

\newif\ifaistats
\newif\ifarxiv
\arxivtrue

\begin{document}

\title{Online Robust Control of Nonlinear Systems\\ with Large Uncertainty\textsuperscript{1}}

\author{\name Dimitar Ho \email dho@caltech.edu \\
       \addr Caltech, Department of Control Dynamical Systems 
       \AND
       \name Hoang M. Le \email hoang.le@microsoft.com \\
       \addr Microsoft Research
       \AND
       \name John C. Doyle \email doyle@caltech.edu \\
       \addr Caltech, Department of Control Dynamical Systems
       \AND
       \name Yisong Yue \email yyue@caltech.edu \\
       \addr Caltech, Department of Computing Mathematical Sciences
}

\maketitle


\noindent 

\begin{abstract}
Robust control is a core approach for controlling systems with performance guarantees that are robust to modeling error, and is widely used in real-world systems. However, current robust control approaches can only handle small system uncertainty, and thus require significant effort in system identification prior to controller design. We present an online approach that robustly controls a nonlinear system under large model uncertainty. Our approach is based on decomposing the problem into two sub-problems, ``robust control  design''  (which assumes small model uncertainty) and ``chasing consistent models'', which can be solved using existing tools from control theory and online learning, respectively. We provide a learning convergence analysis that yields a finite mistake bound on the number of times performance requirements are not met and can provide strong safety guarantees, by bounding the worst-case state deviation. To the best of our knowledge, this is the first approach for online robust control of nonlinear systems with such learning theoretic and safety guarantees. We also show how to instantiate this framework for general robotic systems, demonstrating the practicality of our approach.

\end{abstract}

\section{Introduction}\label{sec:intro}
\footnotetext[1]{This is an extended version of a conference paper that appeared in AISTATS 2021.}
We study the problem of online control for nonlinear systems with large model uncertainty, under the requirement to provide upfront control-theoretic guarantees for the worst case online performance; by large uncertainty, we mean to say that we are given an arbitrarily large set of potential models, of which an \textit{unknown few} are exact descriptions of the true system dynamics. Algorithms with such capabilities can enable us to (at least partially) sidestep undertaking laborious system identification tasks prior to robust controller design.
Motivated by real-world control applications, we formulate a class of problems which allows to address in a unified way common control problems such as stabilization, tracking, disturbance rejection, robust set invariance, etc.. We introduce this as \textit{online control} with \textit{mistake guarantees} (OC-MG): We define a problem instance by specifying a desired system behavior and search for online control algorithms which can quantify, in terms of number of \textit{mistakes}, how often the online controlled system could deviate from this behavior in the worst-case (i.e: worst possible scenario of true system dynamics, disturbances, noise, etc.).

We propose a modular framework for OC-MG:
Use robust control to design a \textit{robust oracle } $\pi$, use online learning to design an algorithm $\sel$ which \textit{chases consistent models}, and fuse them together via a simple meta-algorithm; the end result is Algorithm \ref{alg:ApiselOCO}, which we refer to as $\cl{A}_\pi(\sel)$. 
Our approach is based on decomposing the original problem into the two independent sub-problems "robust oracle design" (ROD) and "consistent models chasing" (CMC) which for many problem instances can be readily addressed with existing tools from control theory (see \secref{sec:oracleasrc}) and online learning (see \secref{sec:selonline}). We demonstrate in \secref{sec:instantiation}, that for general robotic systems, we can solve CMC via competitive convex body chasing \citep{bubeck2020chasing,argue2019nearly,argue2020chasing,sellke2020chasing} and ROD using well-known robust control methods \citep{zhou1998essentials, spongrobustlinear87,Spongrobust92,freeman2008robust, mayne2011tube}.
Once suitable subroutines $\pi$ and $\sel$ are selected, we can provide online performance guarantees for the resulting control algorithm $\cl{A}_\pi(\sel)$ which hold in the large uncertainty setting:
\begin{itemize}[nosep]
\item \textit{Mistake guarantee:} A worst case bound on the total number of times desired system behavior is violated. See Theorem \ref{thm:red}, \ref{thm:redweak}, \ref{thm:redext}.
\item \textit{Safety guarantee:} A worst case norm bound on the state-trajectory. See Theorem \ref{thm:transient}.
\end{itemize}
To provide the above guarantees, $\pi$ and $\sel$ have to be solutions to a corresponding ROD and CMC sub-problem. In \secref{sec:oracleasrc} and \secref{sec:selonline} we discuss that in many problem settings this is not a restrictive assumption. In particular, assuming that ROD can be solved, merely ensures that the overall OC-MG problem is well-posed: The underlying control problem (i.e: assuming no uncertainty) has to be tractably solvable with robust control; It is clear that this is a bare minimum requirement to state a meaningful OC-MG problem. 
In \secref{sec:selonline} we discuss different versions of the CMC sub-problem and present a reduction to nested convex body chasing \cite{bubeck2020chasing} which is applicable for a large class of systems.

In \secref{sec:empirical_validation}, we follow our approach to design a high-performing control algorithm for a difficult nonlinear adaptive control problem: swinging-up a cartpole with large parametric uncertainty \textit{and} state constraints. We benchmark the performance of the online algorithm $\cl{A}_\pi(\sel)$ against the offline optimal algorithm over 900 problem settings (adversarially chosen system parameters, noise, disturbances) and show that $\cl{A}_\pi(\sel)$ performs only marginally worse than the optimal offline controller, which has access to the true system model.\\

\noindent Our framework reveals a promising connection between online learning and robust control theory which enables systematic and modular design of robust learning and control algorithms with provided safety and performance guarantees in the large uncertainty setting. To the best of our knowledge, this is the first time that a fundamental connection between the fields of online learning and control theory has ever been discovered in this context.

\subsection{Problem Statement}
\label{subsec:probstat}
Consider controlling a discrete-time \textit{nonlinear dynamical system} with system equations:
\begin{align}\label{eq:unknownsys}
  x_{t+1} = f^*(t,x_t,u_t),\quad f^* \in \cl{F},
\end{align}
where $x_t \in \cl{X}$ and $u_t\in\cl{U}$ denote the system state and control input at time step $t$ and $\cl{X}\times\cl{U}$ denotes the state-action space. We assume that $f^*$ is an \textit{unknown} function and that we only know of an \textit{uncertainty set} $\cl{F}$ which contains the true $f^*$.

\textbf{Large uncertainty setting. } We impose no further assumptions on $\cl{F}$ and explicitly allow $\cl{F}$ to represent arbitrarily large model uncertainties.

 \textbf{Control objective.} 
 The control objective is specified as a sequence $\bm{\cl{G}}=(\cl{G}_0,\cl{G}_1, \dots)$ of binary cost functions $\cl{G}_t:\cl{X}\times\cl{U} \mapsto \{0,1\}$, where each function $\cl{G}_t$ encodes a desired condition per time-step $t$: $\cl{G}_t(x_t,u_t)=0$ means the state $x_t$ and input $u_t$ meet the requirements at time $t$. $\cl{G}_t(x_t,u_t)=1$ means that some desired condition is violated at time $t$ and we will say that the system made a \textit{mistake} at $t$. 
 The performance metric of system trajectories $\bm{x}:=(x_0,x_1,\dots)$ and $\bm{u}:=(u_0,u_1,\dots)$ is the sum of incurred cost $\cl{G}_t(x_t,u_t)$ over the time interval $[0,\infty)$ and we denote this the \textit{total number of mistakes}:
 \begin{align} \label{eq:mistdef}
    \text{\# mistakes of }\bm{x},\bm{u} &=\sum^{\infty}_{t=0} \cl{G}_t(x_t,u_t). 
  \end{align}
For a system state-input trajectory $(\bm{x},\bm{u})$ to achieve an objective $\bm{\cl{G}}$, we want the above quantity to be finite, i.e.: eventually the system stops making mistakes and meets the requirements of the objective for all time. 

  \textbf{Algorithm design goal.} 
  The goal is to design an online decision rule $u_t=\cl{A}(t,x_t,\dots,x_0)$ such that \underline{regardless} of the unknown $f^*\in\cl{F}$, we are guaranteed to have finite or even explicit upper-bounds on the total number of mistakes \eqref{eq:mistdef} of the online trajectories. Thus, we require a strong notion of robustness: $\cl{A}$ can control any system \eqref{eq:unknownsys} with the certainty that the objective $\bm{\cl{G}}$ will be achieved after finitely many mistakes. 
  It is suitable to refer to our problem setting as \textit{online control} with \textit{mistake guarantees}.

\subsection{Motivation and related work}
\textit{How can we design control algorithms for dynamical systems with strong guarantees without requiring much a-priori information about the system?}\\

\noindent This question is of particular interest in safety-critical settings involving real physical systems, which arise in engineering domains such as aerospace, industrial robotics, automotive, energy plants \citep{vaidyanathan2016advances}. Frequently, one faces two major challenges during control design: guarantees and uncertainty.

\textbf{Guarantees.} Control policies can only be deployed if it can be certified in advance that the policies will meet desired performance requirements online. This makes the mistake guarantee w.r.t. objective $\bm{\cl{G}}$ a natural performance metric, as $\bm{\cl{G}}$ can incorporate control specifications such as tracking, safety and stability that often arise in practical nonlinear dynamical systems. The online learning for control literature mostly focused on linear systems or linear controllers \citep{dean2018regret, simchowitz2018learning,hazan2019nonstochastic, chen2020black}, with some emerging work on online control of nonlinear systems. One approach is incorporate stability into the neural network policy as part of RL algorithm \citep{donti2021enforcing}. Alternatively, the parameters of nonlinear systems can be transformed into linear space to leverage linear analysis \citep{kakade2020information, boffi2020regret}. These prior work focus on sub-linear regret bound, which is not the focus of our problem setup. We note that regret is not necessarily the only way to measure performance. For example, competitive ratio is an alternative performance metric for online learning to control \citep{goel2019online, shi2020online}. In addition, our mistake guarantee requirement is stricter than the no-regret criterion and more amenable to control-theoretic guarantees. Specifically, (fast) convergence of $\tfrac{1}{T}\sum^{T}_{t=0}\cl{G}_t(x_t,u_t)\rightarrow0$ does not imply the total number of mistakes $\sum^{\infty}_{t=0}\cl{G}_t(x_t,u_t)$ is bounded. We provide additional discussion on the large and growing literature on learning control for linear systems, as well as adaptive control techniques from control community in Appendix \ref{sec:app_related}. 

\textbf{Large uncertainty.} Almost always, the dynamics of the real system are not known exactly and one has to resort to an approximate model. The most common approach in online learning for control literature \citep{dean2017sample} is to perform system identification \citep{lennart1999system}, and then use tools from robust control theory \citep{zhou1998essentials}. Robust controller synthesis can provide policies with desired guarantees, if one can obtain an approximate model which is ``provably close enough'' to the real system dynamics. However, estimating a complex system to a desired accuracy level quickly becomes intractable in terms of  computational and/or sample complexity. In the adversarial noise setting, system identification of simple linear systems with precision guarantees can be NP-hard \citep{dahleh1993sample}. General approaches for nonlinear system identification with precision guarantees are for the most part not available (recently \citet{mania2020active} analyzed sample complexity under stochastic noise).

\subsection{Overview of our approach}
\textbf{An alternative to the pipeline SysID$\rightarrow$Control: start with rough models, learn to control online.} While accurate models of real systems are hard to obtain, it is often easy to provide more qualitative or rough models of the system dynamics \textit{without} requiring system identification. Having access to a rough system description, we aim to learn to control the real system from online data and provide control-theoretic guarantees on the online performance in advance.

\textbf{Rough models as compactly parametrizable uncertainty sets.} In practice, we never have the exact knowledge of $f^*$ in advance. However, for engineering applications involving physical systems, the functional form for $f^*$ can often be derived through first principles and application-specific domain knowledge. Conceptually, we can view the unknown parameters of the functional form as conveying both the `modeled dynamics' and `unmodeled (adversarial) disturbance' components of the ground truth $f^*$ in the system $x_{t+1} = f^*(t,x_t,u_t)$. 
It is almost always the case, that we can represent the uncertainty in $f^*$ via a collection of parameters in bounded ranges. How we choose to parametrize a given uncertainty set $\cl{F}$ is not unique and poses a design choice.

We will take this as the starting point for our approach and assume a fixed parametrization of $\cl{F}$ in the form of a tuple $(\mathbb{T},\st{K},d)$, where $(\st{K},d)$ is a compact metric space, called \textit{parameter space}, and $\mathbb{T}$ is a map $\st{K} \mapsto 2^\cl{F}$ which defines a collection of models $\{\mathbb{T}[\theta]\:|\:\theta\in\st{K}\}$ which represents a cover of the uncertainty set $\cl{F}$. We define this formally as a compact parameterization of $\cl{F}$:
 \begin{defn}\label{def:compactparameter}
  A tuple $(\mathbb{T},\st{K},d)$, where $\mathbb{T}:\st{K} \mapsto 2^\cl{F}$ is a \textit{compact parametrization} of $\cl{F}$, if $(\st{K},d)$ is a compact metric space and $\cl{F} \subset \bigcup_{\theta \in \st{K}} \mathbb{T}[\theta]$.
   \end{defn}
\noindent We will work with candidate parameters $\theta\in\st{K}$ of the system and consider a $\theta^*$ to be a \textit{true parameter} of $f^*$, if $f^*\in\mathbb{T}[\theta^*]$. Ideally, each candidate model $\mathbb{T}[\theta]$ has small uncertainty; the precise notion of "small uncertainty" however is problem specific and depends always on the objective.

For concreteness, we give several simple examples of common parameter spaces $\st{K}$:
\begin{enumerate}[nosep]
  \item \emph{Linear time-invariant system}: linear system with matrices $A$, $B$ perturbed by bounded disturbance sequence $\bm{w}\in\ell_\infty$, $\|\bm{w}\|_{\infty}\leq \eta$:
  \begin{align}\label{eq:lin}
      f^*(t,x,u) = Ax+Bu + w_t.
  \end{align}
  The parameter space $\st{K}$ contains bounded intervals describing parameters $\theta =(A,B,\eta)$.
  \item \emph{Nonlinear system, linear parametrization}: nonlinear system, where dynamics are a weighted sum of nonlinear functions $\psi_i$ perturbed by bounded disturbance sequence $\bm{w}\in\ell_\infty$, $\|\bm{w}\|_{\infty}\leq \eta$:
  \begin{align}\label{eq:nonlinlinp}
      f^*(t,x,u) = \sum^{M}_{i=1} a_i \psi_i(x,u)+ w_t.
  \end{align}
  $\st{K}$ contains bounded intervals describing $\theta =(\{a_i\},\eta)$. 
  
  \item \emph{Nonlinear system, nonlinear parametrization}: nonlinear system, with function $g$ parametrized by fixed parameter vector $p\in\R^m$ (e.g., neural networks), perturbed by bounded disturbance sequence $\bm{w}\in\ell_\infty$, $\|\bm{w}\|_{\infty}\leq \eta$:
  \begin{align}\label{eq:nonlinnlp}
      f^*(t,x,u) = g(x,u;p)+ w_t.
  \end{align}
  \noindent $\st{K}$ contains bounded intervals describing $\theta =(p,\eta)$. 
  \end{enumerate}
 \noindent In these examples, the uncertainty set $\cl{F}\subset\cup_{\theta\in\st{K}}\mathbb{T}[\theta]$ is covered by models $\mathbb{T}[\theta]$ with smaller uncertainty of the form $\mathbb{T}[\theta]=\{ t,x,u \mapsto f_\theta(x, u, w_t)\:|\: \|\bm{w}\|_{\infty} \leq \eta \}$, where $f_\theta$ denotes one of the functional forms on the right-hand side of eq. \eqref{eq:lin}, \eqref{eq:nonlinlinp} or \eqref{eq:nonlinnlp}.\\

\noindent \textbf{Online robust control algorithm.} Given a compact parametrization $(\mathbb{T},\st{K}, d)$ for the uncertainty set $\cl{F}$, we design  meta-algorithm $\cl{A}_\pi(\sel)$ (Algorithm \ref{alg:ApiselOCO}) that controls the system \eqref{eq:unknownsys} online by invoking two sub-routines $\pi$ and $\sel$ in each time step. 
\begin{itemize}[nosep]
    \item \emph{Consistent model chasing.} Procedure $\sel$ receives a finite data set $\cl{D}$, which contains state and input observations, and returns a parameter $\theta\in\st{K}$.\\
    \textit{Design goal}: For each time $t$, the procedure $\sel$ should select $\theta_t$ such that the set of models $\mathbb{T}[\theta_t]$ stays  ``consistent'' with $\cl{D}_t$, i.e., candidate models in $\mathbb{T}[\theta_t]$ can \emph{explain} the past data. Moreover posited parameters $\theta_t$ should only change when necessary: two posited parameters $\theta_t$ and $\theta_{t'}$ should not be very different from each other, if both data sets $\cl{D}_t$ and $\cl{D}_{t'}$ contain ``similar'' amount information.
    \item \emph{Robust oracle.} Procedure $\pi$ receives a posited system parameter $\theta\in\st{K}$ as input and returns a control policy $\pi[\theta]:\mathbb{N}\times \cl{X} \mapsto \cl{U}$ which can be evaluated at time $t$ to compute a control action $u_{t} = \pi[\theta](t, x_{t})$ based on the current state $x_{t}$.\\
    \textit{Design goal:}
    We require that $\pi$ represents a robust control design subroutine for the collection of models $\mathbb{T}$, in the sense that policy $\pi[\theta]$ could provide mistake guarantees for $\bm{\cl{G}}$ which are robust to bounded noise \textbf{if} the uncertainty set $\cl{F}$ \textit{were} $\mathbb{T}[\theta]$.
\end{itemize}
  
\begin{algorithm}[t]
  \caption{Meta-Implementation of $\cl{A}_{\pi}(\sel)$ for (OC-MG) }
  \label{alg:ApiselOCO}
  \begin{small}
\begin{algorithmic}[1]
  \Require{procedures $\pi$ and $\sel$}
  \Init $\cl{D}_0 \leftarrow \{\}$, $x_0$ is set to initial condition $\xi_0$
  \For{$t=0,1,\dots$ {\bfseries to} $\infty$} 
  \State{$\cl{D}_t\leftarrow $ append $(t, x_t, x_{t-1}, u_{t-1})$ to  $\cl{D}_{t-1}$ (if $t\geq 1$)} \Comment{update online history of observations}
  \State{$\theta_t \leftarrow \sel[\cl{D}_t]$} \Comment{present online data to $\sel$, get posited parameter $\theta_t$}
   \State{$u_t \leftarrow \pi[\theta_t](t,x_t)$} \Comment{ query $\pi$ for policy $\pi[\theta_t]$ and evaluate it}
  \State{$x_{t+1} \leftarrow f^*(t,x_t,u_t)$} \Comment{system transitions with unknown $f^*$ to next state}
  \EndFor 
\end{algorithmic}
\end{small}
\end{algorithm}



\paragraph{Theoretical contribution.} Our main theoretical results certify safety- and finite mistake guarantees for the online control scheme $\cl{A}_\pi(\sel)$ if the sub-routines $\pi$ and $\sel$ meet the design requirement for ``robust oracle'' and ``consistent model chasing'' for a given uncertainty set $\cl{F}$ and objective $\bm{\cl{G}}$. 
We will clarify the consistency and robustness requirements of the sub-routines $\pi$ and $\sel$ in \secref{sec:result_pi} and \secref{sec:result_sel}. 
For now, we present an informal version of the finite mistake guarantees and worst-case state deviation for the online control scheme $\cl{A}_\pi(\sel)$:
\begin{thm*}[\textit{Informal}]
  \noindent For any (adversarial) $f^*\in \cl{F}$, the online control scheme $\cl{A}_\pi(\sel)$ described in Algorithm \ref{alg:ApiselOCO} guarantees a-priori that the trajectories $\bm{x}$, $\bm{u}$ will achieve the objective $\bm{\cl{G}}$ after finitely many mistakes. 
  The total number of mistakes $\sum^{\infty}_{t=0} {\cl{G}}_t(x_t,u_t)$ is at most
  $$ \textit{oracle performance }M^\pi_\rho *\Gamma_1\left(\frac{\textit{size of uncertainty }\cl{F}}{ \textit{efficiency of }\sel * \textit{robustness margin  }\rho \text{ of }\pi }\right),$$
  and the norm of the state $\|x_t\|$ is at most 
$$\Gamma_2\left(\frac{\textit{size of uncertainty }\cl{F}}{ \textit{efficiency of }\sel * \textit{single-step robustness margin of }\pi },\|x_0\|\right),$$
for some increasing function $\Gamma_1:\R^+\mapsto \R^+$ and some function $\Gamma_2:\R^+\mapsto \R^+$ which is increasing in the first argument and is linear in the second. 
\end{thm*}
\begin{itemize}[nosep]
\item \textit{Performance of }$\pi$: Assume the worst-possible $f^*\in\cl{F}$, but also access to direct online measurements $\theta_t =\theta^*+v_t$ of the a true parameter $\theta^*$ with small noise $v_t$ of size $\rho$; 
$M^\pi_\rho$ denotes the worst-case \#mistakes if we were to apply the almost ideal control law $u_t=\pi[\theta_t](t,x_t)$ in this setting.
\item \textit{Efficiency of }$\sel$: 
We quantify efficiency of $\sel$ in the result through competitive analysis of online algorithms.
The procedure $\sel$ posits parameters efficiently, if as a function of time the parameter selection $\theta_t$ changes only when necessary, that is, it only changes when new observations are informative and keeps a constant value otherwise. We phrase this in terms of a \textit{competitive ratio} $\gamma$ (with $\gamma\geq1$) and distinguish here between $\gamma$-competitive and $(\gamma,T)$-weakly competitive algorithms. The smaller the constant $\gamma$ is, the more efficient the algorithm posits parameters: As discussed in \secref{sec:result_sel}, a smaller $\gamma$ indicates that an online algorithm performs closer to the ideal in-hindsight-optimal algorithm.
\end{itemize}
\begin{rem}
  If the same procedure $\pi$ serves as a robust oracle for a set of criteria $\bm{\cl{G}^{(1)}}, \bm{\cl{G}^{(2)}}$, $\dots$, $ \bm{\cl{G}^{(M)}}$, then correspondingly the instantiation $\cl{A}_\pi(\sel)$ provides multiple finite mistake guarantees, i.e. one for each corresponding criteria $\bm{\cl{G}^{(i)}}$, $i=1,\dots,M$.
  \end{rem}
  \noindent This approach brings several attractive qualities: 
\begin{itemize}[nosep]
    \item \emph{Generality.} The result applies to a wide range of problem settings. The objective $\bm{\cl{G}}$ and uncertainty set $\cl{F}$ serve as a flexible abstraction to represents a large class of dynamical systems and control-theoretic performance objectives.
    \item \emph{Robust guarantees in the large uncertainty setting.}
    Our result applies in settings where only rough models are available. As an example, we can use the result to provide guarantees in control settings with unstable nonlinear systems where stabilizing policies are \textbf{not} known a-priori and which are subject to online adversarial disturbances. 
    \item \emph{Decoupling algorithm design for learning and control.} The construction of the ``robust oracle'' $\pi$ and the consistent model chasing procedure $\sel$ can be addressed with existing tools from control and learning.  More generally, this perspective enables to decouple for the first time learning and control problems in the large uncertainty setting into separate robust control and online learning problems. See discussion in \secref{sec:oracleasrc} and \secref{sec:selonline}.
    \item \textit{Modular algorithm design for robust learning and control. }The above approach provides a first interface between robust control and online learning, which enables a modular design of learning and control algorithms with versatile worst-case performance guarantees against large model uncertainty. 
    \item \textit{A new tool for performance analysis of learning and control algorithms. }
    We can view the above theorem also from an analysis point-of view: Many existing certainty-equivalence based learning and control algorithms can be easily represented as an instance of the meta-algorithm $\cl{A}_{\pi}(\sel)$ and thus can be analyzed using the above theorem.
\end{itemize}

\noindent\textbf{Promising for design of efficient algorithms in practice. }
    Besides focusing on providing worst-case guarantees in a general setting, empirical results show that our framework is a promising approach to design efficient algorithms for learning and control in practice. In \secref{sec:empirical_validation}, we apply our approach to the problem of swinging-up a cartpole with large parametric uncertainty in a realistic and highly challenging setting and show that it achieves consistently (over 900 experiments with different parameter settings) good performance.
\section{Main result}\label{sec:mainresult}
As summarized in Algorithm \ref{alg:ApiselOCO}, the main ingredients of our approach are a robust control oracle $\pi$ that returns a robust controller under posited system parameters, and an online algorithm $\sel$ that chases parameters sets that are consistent with the data collected so far.  Here we expand on the desired conditions of $\pi$ and $\sel$. 

\subsection{Required conditions on procedure $\pi$} 
\label{sec:result_pi}
\textbf{Online control with oracle under idealized conditions.} Generally, procedure $\pi$ is a map $\st{K}\mapsto \cl{C}$ from parameter space $\st{K}$ to the space $\cl{C}:=\{\kappa:\mathbb{N}\times \cl{X} \mapsto \cl{U} \}$ of all (non-stationary) control policies of the form $u_t = \kappa(t,x_t)$.
A desired property of $\pi$ as an \textit{oracle} is that $\pi$ returns controllers that satisfy $\bm{\cl{G}}$ if the model uncertainty \emph{were} small.
In other words, if the uncertainty set $\cl{F}$ \textit{were} contained in the set $\mathbb{T}[\theta]$, then control policy $\pi[\theta]$ could guarantee to achieve the objective $\bm{\cl{G}}$ with finite mistake guarantees. 
 Further, in an idealized setting where the true parameter \emph{were} known exactly, the oracle should return policy such that the system performance is robust to some level of bounded noise -- which is a standard notion of \textit{robustness}. We make this design specification more precise below and discuss how to instantiate the oracle in Section \ref{sec:oracleasrc}.

\textbf{Idealized problem setting.} Let $\theta^*$ be a parameter of the true dynamics $f^*$, and assume that online we have access to noisy observations $\bm{\theta}=(\theta_0,\theta_1,\dots)$, where each measurement $\theta_t$ is $\rho$-close to $\theta^*$, under metric $d$. The online control algorithm queries $\pi$ at each time-step and applies the corresponding policy $\pi[\theta_t]$. The resulting trajectories obey the equations:
\begin{subequations}\label{eq:Srhocond}
    \begin{align}
    x_{t+1} &= f^*(t,x_t,u_t),\quad u_t = \pi[\theta_t](t,x_t)\\
    \theta_t &\text{ s.t.: }d(\theta_t,\theta^*) \leq \rho, \quad\text{ where } f^*\in\mathbb{T}[\theta^*]
\end{align}
\end{subequations}
To facilitate later discussion, define the set of all feasible trajectories of the dynamic equations \eqref{eq:Srhocond} as the \textit{nominal trajectories} $\cl{S}_{\cl{I}}[\rho;\theta]$ of the oracle: 
\begin{defn}\label{def:pi-rhorobust}
    For a time-interval $\cl{I}=[t_1,t_2]\subset \mathbb{N}$ and fixed $\theta\in\st{K}$, let $\cl{S}_{\cl{I}}[\rho;\theta]$ denote the set of all pairs of finite trajectories $x_{\cl{I}}:=(x_{t_1},\dots,x_{t_2})$, $u_{\cl{I}}:=(u_{t_1},\dots,u_{t_2})$ which for $\theta^*=\theta$, satisfy conditions \eqref{eq:Srhocond} with some feasible $f^*$ and sequence $(\theta_{t_1},\dots,\theta_{t_2})$.
\end{defn}
\textbf{Design specification for oracles.} We will say that $\pi$ is $\rho$-robust for some objective $\bm{\cl{G}}$, if all trajectories in $\cl{S}_{\cl{I}}[\rho;\theta]$ achieve $\bm{\cl{G}}$ after finitely many mistakes. We distinguish between robustness and uniform robustness, which we define precisely below:

\begin{defn}[robust oracle]\label{def:pi-rhorobust}
    Equip $\cl{X}$ with some norm $\|\:\|$. For each $\rho,\gamma\geq 0$ and $\theta\in\st{K}$, define the quantity $m^\pi_\rho(\gamma;\theta)$ as
    $$m^\pi_\rho(\gamma;\theta) := \sup\limits_{\cl{I}=[t,t']\::\:t<t'}\:\:\sup\limits_{(x_\cl{I},u_\cl{I})\in\cl{S}_{\cl{I}}[\rho;\theta], \|x_0\| \leq \gamma}\sum\limits_{t\in\cl{I}}\:\: \cl{G}_t(x_t,u_t)$$
    If $m^\pi_0(\gamma;\theta)<\infty$ for all $\gamma\geq 0$, $\theta\in\st{K}$, we call $\pi$ an \textit{oracle} for $\bm{\cl{G}}$ w.r.t. parametrization $(\mathbb{T},\st{K},d)$. In addition, we say an oracle $\pi$ is (locally) (uniformally) $\rho$-robust, if the corresponding property shown in Table \ref{tab:robustness} holds.
    \begin{table}[t]
        \centering
        \begin{tabular}{| r | c |} 
            \hline
            {$\rho$-robust} & $\forall \theta\in\st{K}:\: \sup_{\gamma \geq 0} m^\pi_\rho(\gamma;\theta)<\infty$ \TBstrut \\
            \hline
            {uniformly $\rho$-robust} & ${M}^\pi_\rho:=\sup_{\gamma \geq 0, \theta \in \st{K}} {m}^\pi_\rho(\gamma;\theta)<\infty$  \TBstrut \\
            \hline 
             {locally $\rho$-robust} & $\forall \gamma \geq 0,\: \theta \in \st{K}:\: m^\pi_\rho(\gamma;\theta)<\infty$ \TBstrut \\ 
             \hline
             {locally uniformly $\rho$-robust} & $\forall \gamma \geq 0:\:M^\pi_\rho(\gamma):= \sup_{\theta\in\st{K}} m^\pi_\rho(\gamma;\theta)<\infty$ \TBstrut \\
             \hline
        \end{tabular}
        \caption{Notions of oracle-robustness}
        \label{tab:robustness}
    \end{table}
    If it exists, ${M}^\pi_\rho$ is the \textit{mistake constant}/\textit{function} of $\pi$.
\end{defn}
\noindent The constant $\rho>0$ will be referred to as the robustness margin of $\pi$. If we use the above terms without referencing $\rho$, it should be understood that there exist some $\rho>0$ for which the corresponding property is feasible.
The mistake constant $M^\pi_\rho$ can be viewed as a robust offline benchmark: It quantifies how many mistakes we would make in the worst-case, if we could use the oracle $\pi$ under idealized conditions, i.e., described by \eqref{eq:Srhocond}.

\textbf{Invariance Property.}
On top of $\rho$-robustness, for some results, we will require the following additional condition from the oracle:

\begin{defn}\label{def:pi-invariant}
    For a fixed objective $\bm{\cl{G}}$, define the set of admissable states at time $t$ as $\st{X}_{t}= \{x\:|\: \exists u': \:\cl{G}_{t}(x,u')=0\}$, i.e.: the set of states for which it is possible to achieve zero cost at time step $t$.
    We call a $\rho$-robust oracle $\pi$ \textit{cost invariant}, if for all ${\theta}\in\st{K}$ and $t\geq 0$ the following holds 
      \begin{itemize}
          \item For all $x\in\st{X}_{t}$ holds $\cl{G}_t(x,\pi[{\theta}](t,x))=0$.   
           \item For all $x\in \st{X}_{t}$, $f\in\mathbb{T}[\theta]$ and $\theta' $ s.t. $d(\theta',\theta)\leq \rho$, holds $f(t,x,\pi[{\theta'}](t,x))\in\st{X}_{t+1}$, 
      \end{itemize}
  \end{defn}
\begin{rem}
    The above condition is related to the well-known notion of \textit{positive set/tube invariance} in control theory \cite{blanchini1999set}:
    The above condition requires that the oracle policies $\pi[\theta]$ can ensure for their nominal model $\mathbb{T}(\theta)$ the following closed loop condition: $x_t\in\st{X}_t \implies x_{t+1}\in\st{X}_{t+1}$, $\forall t$. 
\end{rem}

\textbf{Robustness of single-step closed loop transitions.}
To provide worst-case guarantees on the online state trajectory, we need to bound how system uncertainty can affect a single online time-step state transition in the worst-case. To this end, consider we equip the state space $\cl{X}$ with some norm $\|\anon\|$ and define a desired property for the oracle in terms of its performance in the idealized scenario: 
\begin{defn}\label{def:smooth}
$\pi:\st{K}\mapsto \cl{C}$ is $(\alpha,\beta)$-single step robust in the space $(\cl{X},\|\anon\|)$ if for any $2$-time-steps nominal trajectory $(x_{t+1},x_t),(u_{t+1},u_t) \in \cl{S}_{[t,t+1]}[\rho;\theta]$ holds $\|x_{t+1}\| \leq \alpha \rho \|x_t\| + \beta.$
\end{defn}
\noindent The above property requires that in the idealized setting \eqref{eq:Srhocond}, we can uniformly bound the single-step growth of the state by a scalar linear function in the noise-level $\rho$ and the previous state norm. Equivalently, we can explicitly write out condition in \defref{def:smooth} as:
$$\exists \alpha,\beta>0:\quad \forall \theta,\theta'\in\st{K}, x\in\cl{X},f \in \mathbb{T}[\theta], t\geq 0:\:\|f(t,x,\pi[\theta'](t,x))\| \leq \alpha d(\theta,\theta')\|x\| + \beta.$$
\subsection{Required conditions on procedure $\sel$}
\label{sec:result_sel}
We now describe required conditions for $\sel$, and will discuss specific strategies to design $\sel$ in Section \ref{sec:selonline}. 

Let $\mathbb{D} :=\left\{(d_1,\dots,d_N)\:|\:d_i \in \mathbb{N}\times\cl{X}\times\cl{X}\times \cl{U},\: N<\infty \right\}$ be the space of all data sets $\cl{D}=(d_1,\dots,d_N)$ of time-indexed data points $d_i=(t_i,x^+_i,x_i,u_i)$.
We will call a data sequence $\bm{\cl{D}}=(\cl{D}_1,\cl{D}_2,\dots)$ an online stream if the data sets $\cl{D}_t=(d_1,\dots,d_t)$ are formed from a sequence of observations $(d_1,d_2,\dots)$.

\textbf{Consistent models and parameters.} Intuitively, given a data set $\cl{D} = (d_1,\dots,d_N)$ of tuples $d_i=(t_i, x^+_i,x_{i},u_{i})$, \emph{any} candidate $f\in\cl{F}$ which satisfies $x^+_i=f(t_i,x_i,u_i)$ for all $1\leq i \leq N$ is \textit{consistent} with $\cl{D}$; Similarly, we will say that $f$ is \textit{consistent} with an online stream $\bm{\cl{D}}$, if at each time-step $t$, it is consistent with the data set $\cl{D}_t$. We will extend this definition to models $\mathbb{T}[\theta]$ and parameters $\theta\in\st{K}$: The model $\mathbb{T}[\theta]$ is a consistent model for a data set $\cl{D}$ or online stream $\bm{\cl{D}}$, if it contains at least one function $f$ which is consistent with $\cl{D}$ or $\bm{\cl{D}}$, respectively; $\theta$ is then called a consistent parameter. Correspondingly, we will define for some data set $\cl{D}$, the set of all consistent parameters as $\st{P}(\cl{D})$:

\begin{defn}[Consistent Sets]\label{def:Pt}
    The consistent set map $\st{P}:\mathbb{D}\mapsto 2^\st{K}$ returns for each data set $\cl{D}\in\mathbb{D}$ the corresponding set of consistent parameters $\st{P}(\cl{D})$:
    \begin{align}\label{eq:app_poly_0}
        \st{P}(\cl{D}) := \mathrm{closure}\big(\left\{\theta \in\st{K} \ \left| \exists f\in \mathbb{T}[\theta]\text{ s.t. }\begin{array}{l} \forall (t,x^+,x,u) \in \cl{D},\:\:  x^+ = f(t,x,u) \end{array}  \right. \right\}\big).
    \end{align}
\end{defn}

\textbf{Chasing conditions for selection \sel.} The first and basic design specification for subroutine {\sel} is to output a consistent parameter $\theta=\sel[\cl{D}]$ for a given data set $\cl{D}$, provided such a parameter $\theta\in\st{K}$ exists. This requirement can be precisely stated in terms of set-valued analysis: We want $\sel$ to act as a \textit{selection} map $\mathbb{D}\mapsto \st{K}$ of the consistent set map $\st{P}:\mathbb{D}\mapsto 2^\st{K}$.
\begin{defn}\cite{aubin2009set}.
    A function $f:X\mapsto Y$ is a selection of the set-valued map $F:X \mapsto 2^Y$, if $\forall x\in X:\: f(x) \in F(x)$.
\end{defn}
\noindent Since we intend to use $\sel$ in an online manner where we are given a stream of data $\bm{\cl{D}}$, in addition, we require that $\sel$ can posit consistent parameters $\theta_t=\sel[\cl{D}_t]$ in an efficient manner online. A fitting notion of "efficiency" can be defined by comparing the variation of the parameter sequence $\bm{\theta}$ and the sequence of consistent sets $\st{P}(\bm{\cl{D}}):=(\st{P}(\cl{D}_1),\st{P}(\cl{D}_2),\dots)$ over time, where we quantify the latter using the Hausdorff distance $d_H:2^\st{K}\times 2^\st{K}\mapsto \mathbb{R}^+$ defined as 
$$ d_H(\st{S},\st{S}') = \max\left\{\max\limits_{x\in\st{S}} d(x,\st{S}')\:,\: \max\limits_{y\in\st{S}'} d(y,\st{S})\right\}.$$
We phrase this in terms of desired "chasing"-properties (A)-(D) and refer to algorithms $\sel$ with such properties as consistent model chasing (CMC) algorithms:
\begin{defn}\label{def:Selprops}
    Let $\sel:\mathbb{D}\mapsto \st{K}$ be a selection of $\st{P}$. Let 
    $\bm{\cl{D}}=(\cl{D}_1,\cl{D}_2,\dots)$ be an online data stream and let $\bm{\theta}$ be a sequence defined for each time $t$ as $\theta_t=\sel[\cl{D}_t]$. Assume that there always exists an $f\in\cl{F}$ consistent with $\bm{\cl{D}}$ and consider the following statements:
    \begin{enumerate}[label=(\Alph*)]
        \item $\theta^*=\lim_{t \rightarrow \infty}\theta_t$ exists.
        \item $\lim_{t \rightarrow \infty}d(\theta_t,\theta_{t-1})=0.$
        \item $\gamma$-\textit{competitive}:
        $ \sum^{t_2}_{t=t_1+1} d(\theta_t,\theta_{t-1}) \leq \gamma \:d_H(\st{P}(\cl{D}_{t_2}),\st{P}(\cl{D}_{t_1}))$ holds for any $t_1<t_2$.
        \item $(\gamma,T)$-\textit{weakly competitive}: $\sum^{t_2}_{t=t_1+1} d(\theta_t,\theta_{t-1}) \leq \gamma \:d_H(\st{P}(\cl{D}_{t_2}),\st{P}(\cl{D}_{t_1}))$ holds for any $t_2-t_1 \leq T$.
    \end{enumerate}
    We will say that $\sel$ is a \textit{consistent model chasing} (CMC) algorithm, if one of the above \textit{chasing} properties can be guaranteed for any pair of $\bm{\cl{D}}$ and corresponding $\bm{\theta}$. 
\end{defn}
The desired properties describe a natural notion of efficiency for this problem: The posited consistent parameter $\theta_t$ should only change online if new data is also informative. The competitiveness properties (C) and (D) naturally restricts changing $\theta_t$ when little new information is available and permits bigger changes in $\theta_t$ only when new data is informative.
Per time interval $\cl{I}=[t_1,t_2]$, the inequalities in (C) and (D) enforce the following: If the consistent sets $\st{P}(\cl{D}_{t_2})$ and $\st{P}(\cl{D}_{t_1})$ are the same, i.e.: ($d_H(\st{P}(\cl{D}_{t_2}),\st{P}(\cl{D}_{t_1}))=0$), the posited parameters $\theta_{t_1},\dots,\theta_{t_2}$ should all have the same value; On the other hand, if $\st{P}(\cl{D}_{t_2})$ is much smaller than $\st{P}(\cl{D}_{t_1})$, i.e.: ($d_H(\st{P}(\cl{D}_{t_2}),\st{P}(\cl{D}_{t_1}))$ large ) the total variation $\sum^{t_2}_{t=t_1+1} d(\theta_t,\theta_{t-1})$ in the posited parameters is allowed to be at most $\gamma d_H(\st{P}(\cl{D}_{t_2}),\st{P}(\cl{D}_{t_1}))$.

\noindent Properties (C) and (D) are stronger versions of (A) and (B), called \textit{competitiveness}/\textit{weak competitiveness}. 
The relationship between the chasing properties are summarized in the following corollary:
\begin{coro}\label{coro:sel}  
Then following implications hold between the properties of \defref{def:Selprops}:
    \begin{align*}
        \begin{array}{ccc}
            \mathrm{(C)} &\Rightarrow& \mathrm{(A)} \\
            \rotatebox[origin=c]{270}{$\Rightarrow$}& &\rotatebox[origin=c]{270}{$\Rightarrow$}\\
            \mathrm{(D)} &\Rightarrow& \mathrm{(B)} 
        \end{array}
    \end{align*}
    The reverse (and any other) implications between the properties do not hold in general.
\end{coro}
From the above diagram, we can see that $\gamma$-competitiveness is the strongest chasing property as it implies all other properties. In section \secref{sec:selonline}, we discuss that in cases where the consistent set map $\st{P}$ returns convex sets, $\gamma$-competitive CMC algorithms can be designed via a reduction to the nested convex bodies chasing (NCBC) problem \cite{bubeck2020chasing}. On the other hand, for $T=1$ and any $\gamma\geq 1$, the weaker chasing property (D) can always be achieved, even if the map $\st{P}$ returns arbitrary non-convex sets. We show in \secref{sec:generalcaseapisel} a simple projection-based selection rule, which satisfies $(\gamma,1)$-weak competitiveness.

\subsection{Main Results}
Assuming that $\pi$ and $\sel$ meet the required specifications, we can provide the overall guarantees for the algorithm. Let $(\mathbb{T},\st{K},d)$ be a compact parametrization of a given uncertainty set $\cl{F}$. Let $\pi$ be robust per Definition \ref{def:pi-rhorobust} and $\sel$ return consistent parameters per Definition \ref{def:Selprops}. We apply the online control strategy $\cl{A}_{\pi}(\sel)$ described in Algorithm \ref{alg:ApiselOCO} to system $x_{t+1}=f^*(t,x_t,u_t)$ with unknown dynamics $f^*\in\cl{F}$ and denote $(\bm{x},\bm{u})$ as the corresponding state and input trajectories.
\subsubsection{Finite mistake guarantees}
\begin{thm}\label{thm:red}
    Assume that $\sel$ is a CMC-algorithm with chasing property (A) and that $\pi$ is an oracle for an objective $\bm{\cl{G}}$. Then, the following mistake guarantees hold:
    \begin{enumerate}[label=(\roman*)]
     \vspace{-0.05in}
     \item \label{it:conv1} If $\pi$ is robust then $(\bm{x},\bm{u})$ always satisfies: $\sum^{\infty}_{t=0}{\cl{G}}_t(x_t,u_t) < \infty.$
     \item \label{it:conv4} If $\pi$ is uniformly $\rho$-robust and $\sel$ is $\gamma$-competitive, then $(\bm{x},\bm{u})$ obey the inequality:
     $$ \sum^{\infty}_{t=0}{\cl{G}}_t(x_t,u_t) \leq \tpiv{\rho}\Big(\tfrac{2\gamma}{\rho} \dm(\st{K}) +1\Big).$$
    \end{enumerate}
\end{thm}
\noindent The strong competitiveness property is not necessary if instead stronger conditions on the oracle can be enforced: \thmref{thm:redweak} states that if the oracle $\pi$ is cost-invariant we can weaken the assumptions on $\sel$ and still provide finite mistake guarantees. 
\begin{thm}\label{thm:redweak}
    Assume that $\sel$ is a CMC-algorithm with chasing property (B) and that $\pi$ is an uniformly $\rho$-robust, {cost-invariant} oracle for an objective $\bm{\cl{G}}$. Then the following mistake guarantees holds:
    \begin{enumerate}[label=(\roman*)]
     \vspace{-0.05in}
     \item \label{it:conv2} $(\bm{x},\bm{u})$ always satisfies: $\sum^{\infty}_{t=0}{\cl{G}}_t(x_t,u_t) < \infty.$
     \item \label{it:conv3} If $\sel$ is $(\gamma,T)$-weakly competitive, then $(\bm{x},\bm{u})$ obey the following inequality, with $N$ denoting the packing number of $\st{K}$ (see definition below):
        $$\sum^{\infty}_{t=0} \cl{G}_t(x_t,u_t) \leq M^\pi_\rho\left(N(\st{K},r^*)+1\right),\: \quad r^*:=\frac{1}{2}\frac{\rho}{\gamma} \frac{T}{M^\pi_\rho+T}$$
    \end{enumerate}
\end{thm}
\begin{defn}[\citep{dudleypacking87}]\label{def:entropy}
    Let $(\cl{M}, d)$ be a metric space and $\st{S} \subset \cl{M}$ a compact set. For $r>0$, define $N(\st{S},r)$ as the $r$-\textit{packing number} of $\st{S}$:
    \begin{align*}
        N(\st{S},r):=\max\left\{n\in\mathbb{N}\left|\begin{array}{l} \exists \theta_1,\dots,\theta_N \in \st{S} \text{ s.t. }d(\theta_i,\theta_j)>r \text{ for all } 1\leq i, j \leq n,\; i \neq j \end{array}\right.\right\}.
    \end{align*}
\end{defn}

\textbf{Basic asymptotic mistake guarantees.} The guarantees stated in part $(i)$ of the above theorems are asymptotic ones, since they are equivalent to stating $\lim_{t \rightarrow \infty} \cl{G}_t(x_t,u_t)=0$. To provide this guarantee, we only impose weak asymptotic conditions on $\pi$ and $\sel$: $\sel$ needs to either satisfy the convergence condition (A) or (B), while $\pi$ has to be either robust or uniformly robust+cost-invariant for some non-zero margin, respectively. The appeal of the guarantee $(i)$ is that we do not require knowledge of any specific constants (such as $\rho$ or $\gamma$) to give this basic asymptotic guarantee. 

\textbf{Guarantees with explicit mistake bounds for convex and non-convex $\st{P}(\cl{D})$.} On the other hand, the guarantees stated in part $(ii)$ provide additional explicit mistakes bounds if stronger chasing-properties can be established for the CMC algorithm $\sel$. Notice that the mistake bound provided in \thmref{thm:red} provides an exponential improvement over the bound provided in \thmref{thm:redweak}. As an example, assume $\st{K}$ were the unit cube $[0,1]^n$ in $\mathbb{R}^n$ and take $d(x,y):=\|x-y\|_{\infty}$, then the corresponding packing number is at least $N(\st{K},r^*) \geq (M^\pi_\rho\tfrac{\gamma}{\rho}\dm(\st{K}))^n$. The stronger result \thmref{thm:red} requires $\sel$ to be a $\gamma$-competitive CMC-algorithm. We show in \secref{sec:selonline} that if the consistent map $\st{P}$ returns always convex sets, we can design $\gamma$-competitive CMC-algorithms through a reduction to the nested convex body chasing problem. However, it is unclear whether $\gamma$-competitiveness can be achieved for the case of non-convex consistent sets. The weaker mistake bound presented in \thmref{thm:redweak} on the other hand only requires the weak competitiveness property (D) which can always be achieved with $T=1$ ( see \secref{sec:generalcaseapisel}). Therefore, in contrast to \thmref{thm:red}, the result \thmref{thm:redweak} is directly applicable to general maps $\st{P}$. 

\textbf{A theoretical tool for performance analysis.} Theorem \ref{thm:red} and \ref{thm:redweak} can be invoked on any learning and control method that instantiates $\cl{A}_{\pi}(\sel)$. It offers a set of sufficient conditions to verify whether a learning agent $\cl{A}_{\pi}(\sel)$ can provide mistake guarantees: We need to show that w.r.t. some compact parametrization $\parttupl$ of the uncertainty set $\cl{F}$, $\pi$ operates as a robust oracle for some objective $\bm{\cl{G}}$, and that $\sel$ satisfies strong enough chasing properties. 

\textbf{A design guideline. } Theorem \ref{thm:red} also suggests a design philosophy of decoupling the learning and control problem into two separate problems while retaining the appropriate guarantees: (1) design a robust oracle $\pi$ for a specified control goal $\bm{\cl{G}}$; and (2) design an online selection procedure {\sel} that satisfies the chasing properties defined in \defref{def:Selprops}. 

\textbf{Extending the domain of robust control through augmentation.}
Robust control methods which only apply for small uncertainty settings, can be naturally extended to the large uncertainty setting. Provided that the design method can be embedded as an oracle sub-routine $\pi_{rc}$ and that we can find a suitable $\sel$ routine, the meta-algorithm $\cl{A}_{\pi_{rc}}(\sel)$ provides a simple extension of the original method to the large uncertainty setting: $\cl{A}_{\pi_{rc}}(\sel)$ carries any guarantees of the original control method (assuming they can be paraphrased as mistake guarantees) to the large uncertainty setting.

\subsubsection{A worst-case bound on the state norm}
Regardless of the objective $\bm{\cl{G}}$, we can provide worst-case state norm guarantees for $\cl{A}_\pi(\sel)$ in a normed state space $(\cl{X},\|\anon\|)$, if $\sel$ is a competitive or a weakly competitive CMC algorithm and $\pi$ provides sufficient robustness guarantees for a single time-step transition:
\begin{thm}\label{thm:transient}
    Assume that $\pi:\st{K}\mapsto \cl{C}$ is $(\alpha,\beta)$-single step robust in the space $(\cl{X},\|\anon\|)$. Then, the following state bound guarantees hold: 
    \begin{enumerate}[label=(\roman*)]
        \item If $\sel$ is a $\gamma$-competitive CMC algorithm, then:
         \begin{align*}
            \forall t:\quad \|x_t\| \leq e^{\alpha \gamma \phi(\st{K})}\left(e^{-t} \|x_0\| + \beta \frac{e}{e -1}\right).
        \end{align*}
        \item If $\sel$ is a $(\gamma,T)$-weakly competitive CMC algorithm, then:
        \begin{align*}
            \|\bm{x}\|_{\infty} \leq \inf\limits_{0<\mu<1} \left(1+(\alpha\phi(K))^{n^*}\right)\max\{\tfrac{\beta}{1-\mu}, \|x_0\|\} + \beta\sum\limits^{n^*}_{k=0} (\alpha\phi(K))^k
        \end{align*}
        where $n^*= N(\st{K},\tfrac{\mu}{\alpha\gamma})$ and $\phi(\st{K})$ denotes the diameter of $\st{K}$. 
    \end{enumerate}
\end{thm}
\noindent Notice that the above worst-case guarantee holds for \textit{arbitrarily large} diameter $\phi(\st{K})$ of the parameter space $\st{K}$ and applies naturally to scenarios where initially stabilizing control policies do not exist. To this end, it is also important to verify that the required property defined in \defref{def:smooth} does \textit{not} imply existence of an initially stabilizing control policy.
\subsubsection{Characterizing performance in terms of \#mistakes vs. information gain}

As shown in the appendix \secref{sec:app_results}, the mistake bounds presented in \thmref{thm:red}\ref{it:conv4-app} and \thmref{thm:redweak}\ref{it:conv3-app} are actually simplifications of the stricter mistake bound inequalities:
\begin{align*}
    &\textit{If Thm.\ref{thm:red}\ref{it:conv4-app}, then:}  &\forall \ol{\tau}>\tau:\:\:\sum^{\ol{\tau}}_{t=\tau}{\cl{G}}_t(x_t,u_t) &\leq M^\pi_\rho\Big(2\tfrac{\gamma}{\rho}d_H(\st{P}(\cl{D}_{\tau}),\st{P}(\cl{D}_{\ol{\tau}}))+1\Big)\\
    &\textit{If Thm.\ref{thm:redweak}\ref{it:conv3-app}, then:}  &\forall \ol{\tau}>\tau:\:\:\sum^{\ol{\tau}}_{t=\tau} \cl{G}_t(x_t,u_t) &\leq M^\pi_\rho\Big(N(\st{P}(\cl{D}_{\tau})\setminus \mathrm{int}(\st{P}(\cl{D}_{\ol{\tau}})),r^*)+1\Big)
\end{align*}
The above bounds provide more insight into the performance of an $\cl{A}_\pi(\sel)$ algorithm, by showing that the mistake guarantees scale in an efficient way with the system uncertainty: The total mistakes occurring in any time-interval $[\tau,\ol{\tau}]$ increase with the total information gained during the same interval! The amount of information gained in the time-frame $[\tau,\ol{\tau}]$ is measured by the terms $d_H(\st{P}(\cl{D}_{\tau}),\st{P}(\cl{D}_{\ol{\tau}})$ and $N(\st{P}(\cl{D}_{\tau})\setminus \mathrm{int}(\st{P}(\cl{D}_{\ol{\tau}})),r^*)$, which quantify (via Hausdorff-distance and packing numbers) how much parametric uncertainty has been reduced during $[\tau,\ol{\tau}]$, by comparing the consistent sets $\st{P}(\cl{D}_{\tau})$, $\st{P}(\cl{D}_{\ol{\tau}})$ at the start and the beginning of the time window.
Thus, the bounds allow for a simple interpretation: The algorithm $\cl{A}_\pi(\sel)$ only makes informative mistakes.
\subsubsection{Mistake guarantees with locally robust oracles}
The worst-case bound shown in \thmref{thm:transient} can be directly used to extend the result of \thmref{thm:red} to problem settings where we only have access to locally robust oracles.
\begin{thm}[Corollary of Thm. \ref{thm:red}]\label{thm:redext}
Consider the setting and assumptions of Thm.\ref{thm:red} and Thm.\ref{thm:redweak}, but relax the oracle robustness requirements to corresponding local versions and enforce the additional oracle assumption stated in Thm.\ref{thm:transient}.
\noindent Then all guarantees of Thm.\ref{thm:red} \ref{it:conv1},\ref{it:conv4} and Thm.\ref{thm:redweak} \ref{it:conv2},\ref{it:conv3} still hold, if we replace $M^\pi_\rho$ in Thm.\ref{thm:red} \ref{it:conv4} and Thm.\ref{thm:redweak} \ref{it:conv3} respectively by $M^\pi_\rho(\gamma_\infty)$ and $M^\pi_\rho(\gamma^w_\infty)$ with the constants:
\begin{align*}
    \gamma_\infty&=e^{\alpha \gamma \phi (\st{K})}\left(\|x_0\| + \beta \frac{e}{e -1}\right)\\
    \gamma^w_\infty &= \inf\limits_{0<\mu<1} \left(1+(\alpha\phi(K))^{n^*}\right)\max\{\tfrac{\beta}{1-\mu}, \|x_0\|\} + \beta\sum\limits^{n^*}_{k=0} (\alpha\phi(K))^k 
\end{align*}
where $n^*= N(\st{K},\tfrac{\mu}{\alpha\gamma})$ and $\phi(\st{K})$ denotes the diameter of $\st{K}$. 
\end{thm}
\subsubsection{Oracle policies with memory.}
\label{sec:result_ext} 
The previous results assume that $\pi$ returns static policies of the type $(t,x) \mapsto u$. However, this assumption is only made for ease of exposition. All previous results also hold in the case where $\pi$ returns policies which have an internal state, as long as we can define the internal state to be shared among all oracle policies, i.e.: as part of the oracle implementation online, we update the state $z_t$ at each step $t$ according to some fixed update rule $h$
$$z_t = h(t,z_{t-1},x_{t},u_t, \dots, x_0,u_0),$$
and control policies $\pi[\theta]$, $\theta\in\st{K}$ are maps $(t,x,z) \mapsto u$ which we evaluate at time $t$ as $u_t = \pi[\theta](t,x_t,z_t)$.\\

\noindent We discuss in section \ref{sec:oracleasrc} how robust oracle design is a pure robust control problem and briefly overview the main available methods. Designing procedures $\sel$ with the properties stated in \defref{def:Selprops} poses a new class of online learning problems. We propose in \secref{sec:selonline} a reduction to the well-known nested convex body chasing problem, which enables design and analysis of competitive CMC procedures in the scenario where the consistent sets $\st{P}(\cl{D}_t)$ are always convex. For the non-convex scenario, we highlight a general projection based selection which can ensure weak-competitiveness.

\section{Robust Control Oracle}\label{sec:oracleasrc}
Designing robust oracles $\pi$ introduced in \defref{def:pi-rhorobust} can be mapped to well-studied problems in robust control theory. We use a simplified problem setting to explain this correspondence.

Consider a class of system of the form $x_{t+1} = g(x_t,u_t;\theta^*) + w_t, \|w_t\| \leq 1$, where $\theta^*$ is an unknown system parameter which lies in a known compact set $\st{K}\subset \R^m$. We represent the uncertainty set as $\cl{F} = \cup_{\theta \in \st{K}} \mathbb{T}[\theta]$ with $\mathbb{T}[\theta]:= \{f^*: t,x,u \mapsto g(x,u;\theta)+w_t\:|\: \|\bm{w}\|_{\infty} \leq 1 \}$.
\begin{figure}[t]
        \centering
        \ifarxiv
        \includegraphics[width = 0.8\textwidth]{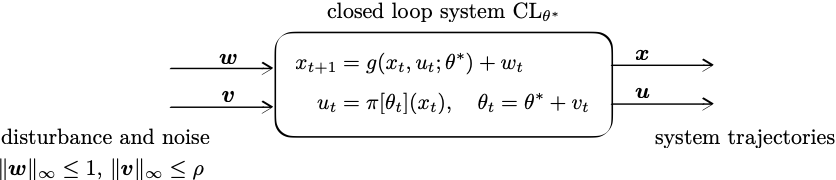}
        \fi
        \ifaistats
         \includegraphics[width = 0.45\textwidth]{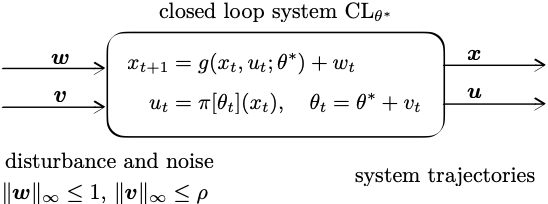}
        \fi
        \caption{closed loop $\mathrm{CL}_{\theta^*}$: An idealized setting in which we have noisy measurements of the true $\theta^*$.}
        \label{fig:block}
\end{figure}
Let $\pi:\st{K} \mapsto \cl{C}$ be a procedure which returns state feedback policies $\pi[\theta]:\cl{X}\mapsto\cl{U}$ for a given $\theta\in\st{K}$.
Designing an uniformly $\rho$-robust oracle $\pi$ can be equivalently viewed as making the closed loop system (described by \eqref{eq:Srhocond}) of the idealized  setting robust to disturbance and noise. For the considered example, the closed loop is depicted in \figref{fig:block} and is represented by $\mathrm{CL}_\theta^*$ which maps system perturbations ($\bm{w}$, $\bm{v}$) to corresponding system trajectories $\bm{x}$, $\bm{u}$. We call $\pi$   a uniformly $\rho$-robust oracle if the cost-performance (measured as $\sum^{\infty}_{t=0} \cl{G}_t (x_t,u_t)$) of the closed loop $\mathrm{CL}_\theta^*$ is robust to disturbances of size $1$ and noise of size $\rho$ for any $\theta^*\in\st{K}$.
For any noise $\|\bm{v}\|_{\infty} \leq \rho$ and disturbance $\|\bm{w}\|_{\infty} \leq 1$, the performance cost has to be bounded as:
$$\sum^{\infty}_{t=0} \cl{G}_t(x_t,u_t) \leq M^\pi_\rho\ \ \text{ or }\ \ \sum^{\infty}_{t=0} \cl{G}_t(x_t,u_t) \leq M^\pi_\rho(\|x_0\|),$$ 
for some fixed constant $M^\pi_\rho$ or fixed function $M^\pi_\rho(\anon)$, in case we can only establish local properties. Now, if we identify the cost functions $\cl{G}_t$ with their level sets $\st{S}_t := \{ (x,u)\:|\: \cl{G}_t(x,u) =0\}$, we can phrase the former equivalently as a form of robust trajectory tracking problem or a set-point control problem\footnote{in that case $\cl{G}_t$ would be not time dependent} \citep{khalil2002nonlinear}. It is common in control theory to provide guarantees in the form of convergence rates (finite-time or exponential convergence) on the tracking-error; these guarantees can be directly mapped to $M^\pi_\rho$ and $M^\pi_\rho(\anon)$, as shown in the next example:
\begin{ex}
    Assume we want to track a desired trajectory $\bm{x}^d$ within $\ep$ accuracy in some normed state space $(\cl{X},\|\anon\|)$. We can phrase this as a control objective $\bm{\cl{G}}$ by defining $\cl{G}_t(x,u) := 0, \text{ if } \|x^d_t-x\|\leq \ep$ and $\cl{G}_t(x,u) := 1, \text{ otherwise }$. Providing an exponential convergence guarantee of the type $\|x^d_t-x_t\|\leq c \mu^t$ with constants $c$, $\mu<1$ is a basic problem studied in control theory. It is easy to see that such a guarantee implies $\sum^{\infty}_{t=0}\cl{G}_t(x_t,u_t) \leq \tfrac{\log(c\|x_0\|)+\log(\ep^{-1})}{\log(\mu^{-1})}$. Hence, if design method $\pi$ provides a policy $\pi[\theta]$ which guarantees (for any $\|\bm{w}\|_\infty\leq 1$, $\|\bm{v}\|_\infty\leq\rho$) for any trajectory of the closed-loop $\mathrm{CL}_{\theta}$ the convergence condition $\forall t:\|x^d_t-x_t\|\leq c \mu^t$, then $\pi$ is a locally $\rho$-uniformly robust oracle with the mistake function $$ M^\pi_\rho = \tfrac{\log(c\|x_0\|)+\log(\ep^{-1})}{\log(\mu^{-1})}.$$
\end{ex}

\noindent \textbf{Available methods for robust oracle design.}
Many control methods exist for robust oracle design. Which method to use depends on the control objective $\bm{\cl{G}}$, the specific application, and the system class (linear/nonlinear/hybrid, etc.). For a broad survey, see \cite{zhou1996robust,spongrobustlinear87,Spongrobust92} and references therein.
We  characterize two general methodologies (which can also be combined):
\begin{itemize}[nosep]
    \item \textit{Robust stability analysis focus:} 
    In an initial step, we use analytical design principles from robust nonlinear and linear control design to propose an oracle $\pi[\theta](x)$ in closed-form for all $\theta$ and $x$. In a second step we prove robustness using analysis tools such as for example \textit{Input-to-State Stability} (ISS) stability analysis \cite{ISSSontag} or robust set invariance methods \cite{rakovic2006simple,rakovic2010parameterized}). 
    \item \textit{Robust control synthesis:} If the problem permits, we can also directly address the control design problem from a computational point of view, by formulating the design problem as an optimization problem and compute for a control law with desired guarantees directly. This can happen partially online, partially offline. Some common nonlinear approaches are robust (tube-based) MPC \cite{mayne2011tube, borrelli2017predictive}, SOS-methods \cite{parrilo2000structured},\cite{aylward2008stability}, Hamilton-Jacobi reachability methods \cite{bansal2017hamilton}.
\end{itemize}
There are different advantages and disadvantages to both approaches and it is important to point out that robust control problems are not always tractably solvable. See \cite{blondel2000survey, Braatzdoylecomplexity} for simple examples of robust control problems which are NP-hard. The computational complexity of robust controller synthesis tends to increase (or even be potentially infeasible) with the complexity of the system of interest; it also further increases as we try optimize for larger robustness margins $\rho$.

\textbf{The dual purpose of the oracle. }
In our framework, access to a robust oracle is a necessary prerequisite to design learning and control agents $\cl{A}_\pi(\sel)$ with mistake guarantees.
However this is a mild assumption and is often more enabling than restrictive.
 First, it represents a natural way to ensure well-posedness of the overall learning and control problem; If robust oracles cannot be found for an objective, then the overall problem is likely intrinsically hard or ill-posed (for example necessary fundamental system properties like stabilizability/ detectability are not satisfied).

 Second, the oracle abstraction enables a modular approach to robust learning and control problems, and directly leverages existing powerful methods in robust control:  
 Any model-based design procedure $\pi$ which works well for the small uncertainty setting (i.e.: acts as a robust oracle) can be augmented with an online chasing algorithm $\sel$ (with required chasing properties) to provide robust control performance (in the form of mistake guarantees) in the large uncertainty setting via the augmented algorithm $\cl{A}_\pi(\sel)$.


\section{Consistent models chasing}
\label{sec:selonline}
The ``chasing properties'' defined in \defref{def:Selprops} express different solution approaches to the same underlying online learning problem, which we refer broadly as a \textit{consistent models chasing} (CMC) problem. We discuss how this problem is linked to the online learning literature and discuss a reduction to the well-known nested convex body chasing (NCBC) problem. 


\subsection{Chasing consistent models competitively} 

Assume a given parameterization $(\mathbb{T},\st{K},d)$ and some fixed data set $\cl{D}_T = (d_1,\dots, d_T)$ which is presented to us in an online fashion and which has at least one consistent parameter $\theta^*$ (i.e: $\st{P}(\cl{D}_T)$ is non-empty).

Our goal is to find a consistent parameter $\theta^*\in\st{P}(\cl{D}_T)$ online, or equivalently, to find model $\mathbb{T}[\theta^*]$ consistent with all data $\cl{D}_T$. However, since we do not have access to all data upfront, the best we can do is posit at each time-step $t$ a parameter $\theta_t\in\st{{P}(\cl{D}_t)}$ consistent at least with all so far seen data and make the hypothesis that $\theta_t$ is also consistent with $\cl{D}_T$ until proven otherwise by new data. Our goal is to posit the parameters $\theta_t$ in an efficient way: we would like to change our hypothesis about the consistent parameter $\theta^*$ as little as possible and thus $\theta_t$ should change over time as little as possible. We call this problem \textit{consistent models chasing} and we view the chasing properties defined in \defref{def:Selprops} as different ways of expressing this goal. Evaluating the performance of the online selection via competitiveness leads to a problem closely related to online learning, which we term \textit{competitive} CMC:

\begin{defn}[weakly and strongly competitive consistent model chasing]\label{def:CCMC}
Given online data of the form $\cl{D}_t=(d_0,\dots,d_t), \cl{D}_t \in \mathbb{D}$, select at each time-step a consistent parameter $\theta_t \in \st{P}(\cl{D}_t)$ such that for any time-interval $[t_1,t_2]$ the following inequality holds:
\begin{align}\label{eq:comp}
\sum^{t_2}_{t=t_1+1} d(\theta_t,\theta_{t-1}) \leq \gamma d_H(\st{P}(\cl{D}_{t_2}),\st{P}(\cl{D}_{t_1}))
\end{align}
An algorithm $\sel:\mathbb{D}\mapsto \st{K}$ which fulfills this requirement for a fixed $\gamma$, is called a $\gamma$-competitive or -strongly competitive consistent model chasing CMC-algorithm. The constant $\gamma$ is called the competitive ratio. If the above requirement \eqref{eq:comp} is only fulfilled for time-intervals $[t_1,t_2]$ of length $T$, i.e.: $t_2-t_1\leq T$, we call $\sel$ a $(\gamma,T)$-weakly competitive CMC algorithm. 
\end{defn}
 \noindent \textbf{An online learning problem. } We quantify performance using competitiveness, which is a common performance objective in the design of online learning algorithms \cite{koutsoupias2000beyond, borodin2005online,chen2018smoothed, goel2019online,shi2020online,yu2020power}. 
 In these problem settings an online learning algorithm is benchmarked against the performance of an in-hindsight optimal offline algorithm. We can equivalently view the term $d_{H}(\st{P}(\cl{D}_{t_2}),\st{P}(\cl{D}_{t_1}))$ as such a benchmark: 
It represents the smallest bound we could provide on the variation $\sum^{t_2}_{t=t_1+1} d(\theta_{t},\theta_{t-1})$ with an offline algorithm, assuming we knew the data set $(d_{t_1},\dots, d_{t_2})$ in advance and were given some arbitrary start parameter $\theta_{t_1}\in\st{P}(\cl{D}_{t_1})$ (see Remark \ref{rem:equiv} below).
\begin{rem}\label{rem:equiv}
    If we knew the data $(d_{t_1},\dots, d_{t_2})$ in advance, the best possible selection is $\theta_{t_1+1}=\dots=\theta_{t_2}=\theta^*$ where $\theta^*\in\argmin_{\theta'\in\st{P}(\cl{D}_{t_2})}d(\theta_{t_1},\theta');$ If we are given an arbitrarily chosen $\theta_{t_1}\in\st{P}(\cl{D}_{t_1})$, the best possible bound we can provide is a worst-case one, i.e.: $\max_{\theta'\in\st{P}(\cl{D}_{t_1})} d(\theta',\st{P}(\cl{D}_{t_2}))=d_{H}(\st{P}(\cl{D}_{t_2}),\st{P}(\cl{D}_{t_1}))$ 
\end{rem}

\subsection{Reducing competitive CMC to competitive nested convex body chasing}
The main difficulty in selecting the parameters $\theta_t$ to solve CMC competitively is that, for any time $t<T$, we cannot guarantee to select a parameter $\theta_t$ which is guaranteed to lie in the future consistent set $\st{P}(\cl{D}_T)$. 
However, a sequence of consistent sets is always nested, i.e. $\st{K}\supset \st{P}(\cl{D}_1)\dots\supset \st{P}(\cl{D}_T)$. This inspires a competitive procedure for the CMC problem, through a reduction to a known online learning problem.  

If we can assume the consistent sets $\st{P}(\cl{D})$ to be always convex, we can reduce the CMC problem to a well-known problem of \textit{nested convex body chasing} (NCBC) \cite{bubeck2020chasing}. This requirement is necessary for the reduction and is stated in Assumption \ref{asp:convexuncertain}. 
We discuss in \secref{sec:psetonline} scenarios for which $\st{P}(\cl{D})$ can always be represented as polyhedrons and apply this to a general class of robotic manipulation problems in section \ref{sec:instantiation}.
\begin{asp}\label{asp:convexuncertain}
    Given a compact parametrization $\parttupl$ of the uncertainty set $\cl{F}$, the consistent sets $\st{P}(\cl{D})$ are always convex for any data set $\cl{D}\in\mathbb{D}$.
\end{asp} 

\noindent \textbf{Nested convex body chasing (NCBC). } In NCBC, we have access to online nested sequence $\st{S}_0,\st{S}_1,\dots,\st{S}_T$ of convex sets in some metric space $(\cl{M},d)$ (i.e.: $\st{S}_t \subset \st{S}_{t-1}$). The learner selects at each time $t$ a point $p_t$ from  $\st{S}_t$. The goal in competitive NCBC is to produce $p_1,\dots,p_T$ online such that the total moving cost $\sum^{T}_{j=1} d(p_{j},p_{j-1})$ at time $T$ is competitive with the offline-optimum, i.e. there is some $\gamma>0$ s.t. $\sum^{T}_{j=1} d(p_{j},p_{j-1})\leq \gamma \max_{p_0\in\st{S}_0}\mathrm{OPT}_T(p_0)$, where $\mathrm{OPT}_T(p_0) :=\min_{p \in \st{S}_T} d(p,p_0)$. 
\begin{rem}
    NCBC is a special case of the more general convex body chasing (CBC) problem, first introduced by \cite{friedman1993convex}, which studied competitive algorithms for metrical goal systems. 
 \end{rem}
%
Let the sequence of convex consistent sets $\st{P}(\cl{D}_t)$ be the corresponding $\st{S}_t$ of the NCBC problem, any $\gamma$-competitive agent $\cl{A}$ for the NCBC problem can instantiate a $\gamma$-competitive selection for competitive model-chasing, as summarized in the following reduction:

\begin{prop}\label{prop:reduction}
Consider the setting of \aspref{asp:convexuncertain}. Then any $\gamma$-competitive algorithm for NCBC 
in metric space $(\st{K},d)$ instantiates via Algorithm \ref{alg:ncbcsel} a $\gamma$-competitive CMC algorithm $\sel_{\ncbc}$ in the parametrization $\parttupl$.
\end{prop}
\begin{proof}
Since we set $\st{S}_t=\st{P}(\cl{D}_t)$, the terms $d_H(\st{P}(\cl{D}_{T}),\st{P}(\cl{D}_{0}))=\max_{p_0\in\st{S}_0} \text{OPT}_T(p_0)$ are equal. Hence $\gamma$-competitive NCBC implies $\gamma$-competitive CMC over all time-intervals $[t_0,T]$, with $t_0=1$. To see that this also holds for any choice of $t_0$, recall that the NCBC problem requires the competitiveness condition to hold for any sequence of nested convex bodies. Thus, for a fix sequence $\st{S}_t$ the condition has to be satisfied also for the shifted sequences $\st{S}'_{t} := \st{S}_{t+k}$.
\end{proof}

\begin{algorithm}[t]
    \caption{ $\gamma$-competitive CMC selection $\sel_{\ncbc}$}
    \label{alg:ncbcsel}
    \begin{small}
 \begin{algorithmic}[1]
    \Require{$\gamma$-competitive NCBC algorithm $\cl{A}_{\ncbc}$, consistent set procedure $\st{P}$}

    \Procedure{$\sel_{\ncbc}$}{$t,x^+,x,u$}
    \State $\cl{D}_t \leftarrow \cl{D}_{t-1} \cup (t,x^+,x,u)$
    \State $\st{S}_t \leftarrow \st{P}(\cl{D}_t)$ \Comment{construct/update new consistent set}
    \State present set $\st{S}_t$ to $\cl{A}_{\ncbc}$
    \State $\cl{A}_{\ncbc}$ chooses $\theta_t \in \st{S}_t$
    \State \textbf{return} $\theta_{t}$ 
    \EndProcedure
 \end{algorithmic}
 \end{small}
 \end{algorithm}


\textbf{Simple competitive NCBC-algorithms in euclidean space $\R^n$.}
When $(\st{K},d)$ is a compact euclidean finite dimensional space,
 recent exciting progress on the NCBC problem provides a variety of competitive algorithms \cite{argue2019nearly,argue2020chasing,bubeck2020chasing,sellke2020chasing} 
 that can instantiate competitive selections per Algorithm \ref{alg:ncbcsel}.

We highlight two simple instantiations based on results in \cite{argue2019nearly}, and \cite{bubeck2020chasing}. Both algorithms can be tractably implemented in the setting of Assumption \ref{asp:convexuncertain}. The selection criteria for $\sel_p(\cl{D}_t)$ and $\sel_s(\cl{D}_t)$ is defined as:
\begin{subequations}
    \label{eq:sel}
\begin{align}
    \label{eq:sela}\texttt{SEL}_p(\cl{D}_t) &:= \argmin\limits_{\theta \in \st{P}(\cl{D}_t)} \|\theta-\texttt{SEL}_p(\cl{D}_{t-1})\|,\\ 
    \label{eq:selb} \texttt{SEL}_s(\cl{D}_t)  &:= s(\st{P}(\cl{D}_t)),
\end{align}
\end{subequations}
where $\texttt{SEL}_p$ defines simply a greedy projection operator and where $\texttt{SEL}_s$ selects according to the \textit{Steiner-Point} $s(\st{P}(\cl{D}_t))$ of the consistent set $\st{P}(\cl{D}_t)$ at time $t$. 
\begin{defn}[Steiner Point]
    For a convex body $\st{K}$, the Steiner point is defined as the following integral over the $n-1$ dimensional sphere $\Sn^{n-1}$:
    \begin{equation}\label{eq:st}
        s(K) = n \int_{v \in \Sn^{n-1}} \max_{x \in \st{K}} \langle v,x \rangle v dv.
    \end{equation}
\end{defn}
\begin{rem}
    As shown in \cite{argue2020chasing}, the Steiner point can be approximated efficiently by solving randomized linear programs. We take this approach for our later empirical validation in section \ref{sec:empirical_validation}. 
\end{rem}
\noindent The competitive analysis presented in \cite{bubeck2020chasing} applies and appealing to Proposition \ref{prop:reduction} we can establish that $\texttt{SEL}_p$ and $\texttt{SEL}_s$ are competitive CMC algorithms:
\begin{coro}[of Theorem 1.3 \cite{argue2019nearly}, and Theorem 2.1 \cite{bubeck2020chasing}]\label{coro:pathlengthchasing}
Assume $\st{K}$ is a compact convex set in $\R^n$ and  $d(x,y):=\|x-y\|_2$. Then, the procedures $\texttt{SEL}_p$ and $\texttt{SEL}_s$ are competitive (CMC)-algorithms with competitive ratio $\gamma_p$ and $\gamma_s$: 
   \begin{align}
       &\gamma_p = (n-1)n^{\frac{n+1}{2}}, & \gamma_s = \frac{n}{2}.
   \end{align}
\end{coro}

\subsection{Constructing consistent sets online}\label{sec:psetonline}
Constructing consistent sets $\st{P}(\cl{D}_t)$ online can be addressed with tools from set-membership identification.
For a large collection of linear and nonlinear systems, the sets $\st{P}(\cl{D})$ can be constructed efficiently online. Such methods have been developed and studied in the literature of set-membership identification, for a recent survey see \citep{milanese2013bounding}. Moreover it is often possible to construct $\st{P}(\cl{D})$ as an intersection of finite half-spaces, allowing for tractable representations as LPs. To see a particularly simple example, consider the following nonlinear system with some 
 unknown parameters $\alpha^*\in\R^{M}$ and $\eta^*$, where $w_t$ is a vector with entries in the interval $[-\eta^*,\eta^*]$:
\begin{align}\label{eq:kernelized}
    x_{t+1} = \sum^{M}_{i=1} \alpha^*_i \psi_i(x_t,u_t) + w_t,
\end{align}
where $\psi_i:\cl{X}\times\cl{U} \mapsto \cl{X}$ are $M$ known nonlinear functions. If we represent the above system as an uncertain system $\mathbb{T}$ with parameter $\theta^* = [\alpha^*;\eta^*]$, it is easy to see that the consistent sets $\st{P}(\cl{D})$ for some data $\cl{D}=\{ (x^+_i,x_i,u_i)\:|\: 1\leq i\leq H\}$ takes the form of a polyhedron:
\begin{align*}
    \st{P}(\cl{D}) &= \left\{\theta=[\alpha;\eta]\:\left|\: \text{ s.t. }\eqref{eq:lpcondition}\text{ for all }1\leq i\leq H\right.\right\},
\end{align*}
defined by the inequalities:
\begin{subequations}
    \label{eq:lpcondition}
\begin{align}
     [\psi_1(x_i,u_i),\dots,\psi_M(x_i,u_i)]\alpha &\leq x^+_i + \mathds{1}\eta,\\
     [\psi_1(x_i,u_i),\dots,\psi_M(x_i,u_i)]\alpha &\geq x^+_i - \mathds{1}\eta. 
\end{align}
\end{subequations}
We can see that any linear discrete-time system can be put into the above form \eqref{eq:kernelized}. Moreover, as shown in section \ref{sec:instantiation} the above representation also applies for a large class of (nonlinear) robotics system.

\subsection{A general approach to weakly competitive consistent model chasing }\label{sec:generalcaseapisel}
\noindent In contrast to the stricter problem formulation in \defref{def:CCMC}, we can give a simple and general selection rule which is always $(\gamma,1)$-weakly competitive:
\begin{defn}[projection-based chasing]\label{def:selp2}
        Pick $\theta_t \in \st{P}(\cl{D}_t)$ always such that for some fixed $\gamma>0$, at every time-step $t$ holds $d(\theta_t,\theta_{t-1}) \leq \gamma d(\st{P}(\cl{D}_t),\theta_{t-1}).$ 
\end{defn}

\noindent This projection-based chasing algorithm might not always be tractable to implement, since it requires solving a potentially non-convex optimization problem with $\gamma$ relative accuracy. However, it describes a simple blueprint to design a general CMC algorithm $\sel$ which is weakly competitive in a general setting, that is, allowing for potentially \textit{infinite dimensional} metric spaces and \textit{non-convex} consistent sets $\st{P}(\cl{D}_t)$. Combined with a suitable oracle $\pi$, the resulting online control algorithm $\cl{A}_\pi(\sel_p)$ provides finite mistake guarantees for objectives $\bm{\cl{G}}$ according to \thmref{thm:redweak}.

\section{Examples}\label{sec:instantiation}
Here we demonstrate two illustrative examples of how to instantiate our approach: learning to stabilize a scalar linear system, and learning to track a trajectory on a fully actuated robotic system. Detailed derivations are provided in the Appendix \ref{app:instantiation}.

\subsection{Control of uncertain scalar linear system}
Consider the basic  setting of controlling a scalar linear system with unknown parameters and bounded disturbance $|w_k|\leq \eta<1$ : $$x_{k+1}= \alpha^* x_k+\beta^* u_k+w_k=:f^*(k,x_k,u_k),$$ with the goal to reach the interval $\mathcal{X}_{\mathcal{T}} = [-1,1]$ and remain there. Equivalently, this goal can be expressed as the objective $\bm{\cl{G}} = (\cl{G}_0, \cl{G}_1,\dots)$ with cost functions: $$\cl{G}_t(x,u):= \left\{\begin{array}{cl} 0,&\quad \text{if }|x| \leq 1\\ 1, &\quad\text{else} \end{array}\right., \quad \forall t \geq 0,$$
since "reaching and remaining in $\cl{X}_{\mathcal{T}}$" is equivalent to achieving $\bm{\cl{G}}$ within finite mistakes. 
The true parameter $\theta^* = (\alpha^*,\beta^*)$ lies in the set $\mathsf{K} = [-a,a]\times [1,1+2b_\Delta]$ with known $a>0$, $\eta<1$ and $b_{\Delta}>0$.

\textbf{Parametrization of uncertainty set. }
We describe the uncertainty set $\cl{F} = \cup_{\theta \in \st{K}}\mathbb{T}[\theta]$ through the compact parametrization $(\mathbb{T},\st{K}, d)$ with parameter space $(\mathsf{K},d)$, $d(x,y):= \|x-y\|$, $\|x\|:= |x_1|+a|x_2|$ and the collection of models: $$\mathbb{T}[\theta] := \{t,x,u\mapsto \theta_1 x + \theta_2 u +w_t\:|\: \|\bm{w}\|_{\infty} \leq \eta \}.$$

\textbf{Robust oracle $\pi$. }
We use the simple deadbeat controller: $\pi[\theta](t,x):= -(\theta_1/\theta_2)x$. It can be easily shown that $\pi$ is a locally $\rho$-uniformly robust oracle for $\bm{\cl{G}}$ for any margin in the interval $(0, 1-\eta)$.

\textbf{Construction of consistent sets via LPs.}
The consistent sets $\st{P}(\cl{D}_t)$ are convex, can be constructed online and are the intersection of $\st{K}$ with $2t$ halfspaces:
\begin{align*}
 \left\{\theta\in \st{K}\:|\: \text{s.t.: } \forall i< t: |\theta_1 x_i + \theta_2 u_i - x_{i+1}|\leq  \eta  \right\},
\end{align*}

\textbf{Competitive $\sel_s$ via NCBC. }
We instantiate  $\sel_s$ using Algorithm \ref{alg:ncbcsel} with Steiner point selection. From Corollary \ref{coro:pathlengthchasing}, we have $\gamma_s = \tfrac{n}{2} =1$.

\textbf{Mistake guarantee for $\cl{A}_\pi(\sel_s)$} 
The extension of the results in \thmref{thm:redext}, \ref{it:conv4} apply and we obtain for $\cl{A}_\pi(\sel_s)$ and the stabilization objective $\bm{\cl{G}}$, the following mistake guarantee:
\begin{align*}
    \sum^{\infty}_{t=0}{\cl{G}}_t(x_t,u_t) \leq 2e\ep^2 + \ep,\quad \ep =2(a+{b_\Delta}) .
\end{align*}
The above inequality shows that the worst-case total number of mistakes grows quadratically with the size of the initial uncertainty $\ep$ in the system parameters. Notice, however that the above inquality holds for arbitrary large choices of $a$ and $b_\Delta$. Thus, $\cl{A}_\pi(\sel_s)$ gives finite mistake guarantees for this problem setting for arbitrarily large system parameter uncertainties.

\subsection{Trajectory following in robotic systems}

We consider uncertain fully-actuated robotic systems. 
A vast majority of robotic systems can be modeled via the robotic equation of motion \citep{murray2017mathematical}:
\begin{equation}
\label{eq:equation_of_motion}
    \mathbf{M}_\eta(q)\ddot{q}+\mathbf{C}_\eta(q,\dot{q})\dot{q} + \mathbf{N}_\eta(q,\dot{q}) = \tau + \tau_d,
\end{equation}
where $q\in\R^{n}$ is the multi-dimensional generalized coordinates of the system, $\dot{q}$ and $\ddot{q}$ are its first and second (continuous) time derivatives, $\mathbf{M}_\eta(q), \mathbf{C}_\eta(q,\dot{q}), \mathbf{N}_\eta(q,\dot{q})$ are matrix and vector-value functions that depend on the parameters $\eta\in\mathbb{R}^m$ of the robotic system, i.e. $\eta$ comes from parametric physical model. 
 Often, $\tau$ is the control action (e.g., torques and forces of actuators), which acts as input of the system. Disturbances and other uncertainties present in the system can be modeled as additional torques $\tau_d \in \R^n$ perturbing the equations.  
%
The disturbances are bounded as $|\tau_d(t)| \leq \omega$, where $\omega\in\R^n$ and the inequality is entry-wise. 

Consider a system with unknown $\eta^*$, $\omega^*$, where the parameter $\theta^*=[\eta^*;\omega^*]$ is known to be contained in a bounded set $\st{K}$. Our goal is to track a desired trajectory $q_d$ given as a function of time $q_d:\mathbb{R}\mapsto \R^{n}$, within $\epsilon$ precision. Hence, denoting $x = [q^\top,\dot{q}^\top]^\top$ as the state vector and $x_d = [q^\top_d,\dot{q}^\top_d]^\top$ as the desired state, we want the system state trajectory $x(t)$ to satisfy:
 \begin{align}
    \limsup\limits_{t\rightarrow \infty}\left\|  x(t)-x_d(t)\right\|\leq \epsilon.
\end{align}
As common in practice, we assume we can observe the sampled measurements $x_k :=x(t_k)$, $x^d_k :=x^d(t_k)$ 
 and apply a constant control action (zero-order-hold actuation) $\tau_k:=\tau(t_k)$ at the discrete time-steps $t_k=k T_s$ with small enough sampling-time $T_s$ to allow for continuous-time control design and analysis. 

\textbf{Control objective $\bm{\cl{G}}^{\epsilon}$.}
We phrase trajectory tracking as a control objective $\bm{\cl{G}}^{\ep}$ with the cost functions: 
$$\cl{G}^\epsilon_k(x,u):= \left\{\begin{array}{cl} 0,&\quad \text{if }\|x-x^d_k\| \leq \epsilon\\ 1, &\quad\text{else} \end{array}\right., \quad \forall k \geq 0.$$
\textbf{Robust oracle design.} We outline in Appendix \ref{sec:app_experiment} how to design a control oracle using established methods for robotic manipulators \cite{Spongrobustcontrol}. 

\textbf{Constructing consistent sets.}
 For many robotic systems (for example robot manipulators),  one can derive from first principles \cite{murray2017mathematical} that the left-hand-side of \eqref{eq:equation_of_motion} can be factored into a $n\times m$ matrix of known functions $\mathbf{Y}(q,\dot{q},\ddot{q})$ and a constant vector $\eta\in\R^m$. We can then construct consistent sets at each time $t$ as polyhedrons of the form:
 \begin{equation}\label{eq:roboconsistx}
    \st{P}(\cl{D}_t) = \bigg\{ \theta\in\st{K} \in\mathbb{R}^{m+n}\bigg|\forall k\leq t:\: \mathbf{A}_k \theta \leq \mathbf{b}_k \bigg\}, 
    \end{equation}
    where $\mathbf{A}_k = \mathbf{A}_k(x_k,\tau_k)$ and $\mathbf{b}_k = \mathbf{b}_k(x_k,\tau_k)$ are a matrix and vector of ``features'' constructed from $u_k = \tau_k$ and $x_t$ via the known functional form of $\mathbf{Y}$.

\textbf{Designing $\pi$ and $\sel$. } 
We outline in Appendix \ref{app:instantiation} how to design a robust oracle based on a well-established robust control method for robotic manipulators proposed in \cite{Spongrobustcontrol}. Since the consistent sets are convex and can be constructed online, we can implement procedures $\sel_p$ and $\sel_s$ defined in \eqref{eq:sel} as competitive algorithms for the CMC problem.

\textbf{Mistake guarantee for $\cl{A}_\pi(\sel_{p/s})$} 
The resulting online algorithm $\cl{A}_\pi(\sel_p)$ or $\cl{A}_\pi(\sel_s)$ guarantees finiteness of the total number of mistakes $\sum^{\infty}_{k=0} \cl{G}^\epsilon_k(x_k,\tau_k)$ which implies the desired tracking behavior $\limsup_{k\rightarrow \infty}\|x-x^d\| \leq \epsilon$. Moreover, if we can provide a bound $M$ on the mistake constant $M^\pi_\rho< M$, we obtain from \thmref{thm:red}, \ref{it:conv4} the mistake guarantee: 
$$ \sum\limits^{\infty}_{k=0} \cl{G}^\epsilon_k(x_k,\tau_k)\leq M(\tfrac{2L}{\rho}+1),$$
which bounds the number of times the system could have a tracking error larger than $\epsilon$.

\section{Empirical Validation}\label{sec:empirical_validation}

We illustrate the practical potential for of our approach on a challenging cart-pole swing-up goal from limited amount of interaction. Compared to the standard cart-pole domain that is commonly used in RL \citep{brockman2016openai}, we introduce modifications motivated by real-world concerns in several important ways:
\begin{enumerate}[nosep]
    \item \emph{Goal specification}: the goal is to swing up and balance the cart-pole from a down position, which is significantly harder than balancing from the up-right position (the standard RL benchmark).
    \item \emph{Realistic dynamics}: we use a high-fidelity continuous-time nonlinear model, with noisy measurements of discrete-time state observations.
    \item \emph{Safety}: cart position has to be kept in a bounded interval for all time. In addition, acceleration should not exceed a specified maximum limit.
    \item \emph{Robustness to adversarially chosen system parameters}: We evaluate on 900 uncertainty settings, each with a different $\theta^*$ reflecting mass, length, and friction. The tuning parameter remains the same for all experiments. This robustness requirement amounts to a generalization goal in contemporary RL. 
    \item \emph{Other constraints}: no system reset is allowed during learning (i.e., a truly continual goal).
\end{enumerate}
 Our introduced modification make this goal significantly more challenging from both online learning and adaptive control perspective. Table \ref{tab:robustness} summarizes the results over 900 different parameter conditions (corresponding to 900 adversarial settings). See Appendix \ref{sec:app_experiment} for additional description of our  setup and results.

 \begin{table}[t]
    \centering
     \begin{tabular}{| c || c | c | c | c | c |} 
        \hline
     $\pi[\theta^*]$ & $0$ & $0.4$ & $0.99$ & $1$ & $1$ \TBstrut \\ 
     \hline
     $\cl{A}_\pi(\sel)$ & $0$ & $0.2$ & $0.8$ & $0.95$ & $1$ \TBstrut \\ 
     \hline
     \hline
     $T$ & $3\,s$ & $6\,s$ & $12\,s$ & $30\,s$ & $50\,s$ \TBstrut \\ 
     \hline
     \end{tabular}
     \caption{Fraction of experiments completing the swing up before time $T$: ideal policy $\pi[\theta^*]$ vs.  $\cl{A}_\pi(\sel)$ }
     \label{tab:robustness}
\end{table}
We employ well-established techniques to synthesize model-based oracles. The expert controllers are a hybrid combination of a linear state-feedback LQR around the upright position, a so-called energy-based swing-up controller (See \cite{aastrom2000swinging}) and a control barrier-function to respect the safety constraints \cite{SafetyCBFAmes2017}. 
%
As also described in \cite{dulac2019challenges}, adding constraints on state and acceleration makes learning the swing-up of the cart-pole a significantly harder goal for state-of-the art learning and control algorithms. 


Table \ref{tab:robustness} compares the online algorithm to the corresponding ideal oracle policy $\pi[\theta^*]$ shows that the online controller is only marginally slower.

\clearpage
\bibliography{neurips_paper}

\begin{thebibliography}{78}
\providecommand{\natexlab}[1]{#1}
\providecommand{\url}[1]{\texttt{#1}}
\expandafter\ifx\csname urlstyle\endcsname\relax
  \providecommand{\doi}[1]{doi: #1}\else
  \providecommand{\doi}{doi: \begingroup \urlstyle{rm}\Url}\fi

\bibitem[Abbasi-Yadkori and Szepesv{\'a}ri(2011)]{abbasi2011regret}
Yasin Abbasi-Yadkori and Csaba Szepesv{\'a}ri.
\newblock Regret bounds for the adaptive control of linear quadratic systems.
\newblock In \emph{Proceedings of the 24th Annual Conference on Learning
  Theory}, pages 1--26, 2011.

\bibitem[Abbasi-Yadkori et~al.(2018)Abbasi-Yadkori, Lazic, and
  Szepesv{\'a}ri]{abbasi2018regret}
Yasin Abbasi-Yadkori, Nevena Lazic, and Csaba Szepesv{\'a}ri.
\newblock Regret bounds for model-free linear quadratic control.
\newblock \emph{arXiv preprint arXiv:1804.06021}, 2018.

\bibitem[Agarwal et~al.(2019{\natexlab{a}})Agarwal, Bullins, Hazan, Kakade, and
  Singh]{agarwal2019online}
Naman Agarwal, Brian Bullins, Elad Hazan, Sham~M Kakade, and Karan Singh.
\newblock Online control with adversarial disturbances.
\newblock \emph{arXiv preprint arXiv:1902.08721}, 2019{\natexlab{a}}.

\bibitem[Agarwal et~al.(2019{\natexlab{b}})Agarwal, Hazan, and
  Singh]{agarwal2019logarithmic}
Naman Agarwal, Elad Hazan, and Karan Singh.
\newblock Logarithmic regret for online control.
\newblock In \emph{Advances in Neural Information Processing Systems}, pages
  10175--10184, 2019{\natexlab{b}}.

\bibitem[{Ames} et~al.(2017){Ames}, {Xu}, {Grizzle}, and
  {Tabuada}]{SafetyCBFAmes2017}
A.~D. {Ames}, X.~{Xu}, J.~W. {Grizzle}, and P.~{Tabuada}.
\newblock Control barrier function based quadratic programs for safety critical
  systems.
\newblock \emph{IEEE Transactions on Automatic Control}, 62\penalty0
  (8):\penalty0 3861--3876, Aug 2017.
\newblock ISSN 1558-2523.
\newblock \doi{10.1109/TAC.2016.2638961}.

\bibitem[Anderson et~al.(2001)Anderson, Brinsmead, Liberzon, and
  Stephen~Morse]{anderson2001multiple}
Brian~DO Anderson, Thomas Brinsmead, Daniel Liberzon, and A~Stephen~Morse.
\newblock Multiple model adaptive control with safe switching.
\newblock \emph{International journal of adaptive control and signal
  processing}, 15\penalty0 (5):\penalty0 445--470, 2001.

\bibitem[Argue et~al.(2019)Argue, Bubeck, Cohen, Gupta, and
  Lee]{argue2019nearly}
CJ~Argue, S{\'e}bastien Bubeck, Michael~B Cohen, Anupam Gupta, and Yin~Tat Lee.
\newblock A nearly-linear bound for chasing nested convex bodies.
\newblock In \emph{Proceedings of the Thirtieth Annual ACM-SIAM Symposium on
  Discrete Algorithms}, pages 117--122. SIAM, 2019.

\bibitem[Argue et~al.(2020)Argue, Gupta, Guruganesh, and
  Tang]{argue2020chasing}
CJ~Argue, Anupam Gupta, Guru Guruganesh, and Ziye Tang.
\newblock Chasing convex bodies with linear competitive ratio.
\newblock In \emph{Proceedings of the Fourteenth Annual ACM-SIAM Symposium on
  Discrete Algorithms}, pages 1519--1524. SIAM, 2020.

\bibitem[{\AA}str{\"o}m and Furuta(2000)]{aastrom2000swinging}
Karl~Johan {\AA}str{\"o}m and Katsuhisa Furuta.
\newblock Swinging up a pendulum by energy control.
\newblock \emph{Automatica}, 36\penalty0 (2):\penalty0 287--295, 2000.

\bibitem[Aubin and Frankowska(2009)]{aubin2009set}
Jean-Pierre Aubin and H{\'e}l{\`e}ne Frankowska.
\newblock \emph{Set-valued analysis}.
\newblock Springer Science \& Business Media, 2009.

\bibitem[Aylward et~al.(2008)Aylward, Parrilo, and
  Slotine]{aylward2008stability}
Erin~M Aylward, Pablo~A Parrilo, and Jean-Jacques~E Slotine.
\newblock Stability and robustness analysis of nonlinear systems via
  contraction metrics and sos programming.
\newblock \emph{Automatica}, 44\penalty0 (8):\penalty0 2163--2170, 2008.

\bibitem[Bansal et~al.(2017)Bansal, Chen, Herbert, and
  Tomlin]{bansal2017hamilton}
Somil Bansal, Mo~Chen, Sylvia Herbert, and Claire~J Tomlin.
\newblock Hamilton-jacobi reachability: A brief overview and recent advances.
\newblock In \emph{2017 IEEE 56th Annual Conference on Decision and Control
  (CDC)}, pages 2242--2253. IEEE, 2017.

\bibitem[Blanchini(1999)]{blanchini1999set}
Franco Blanchini.
\newblock Set invariance in control.
\newblock \emph{Automatica}, 35\penalty0 (11):\penalty0 1747--1767, 1999.

\bibitem[Blondel and Tsitsiklis(2000)]{blondel2000survey}
Vincent~D Blondel and John~N Tsitsiklis.
\newblock A survey of computational complexity results in systems and control.
\newblock \emph{Automatica}, 36\penalty0 (9):\penalty0 1249--1274, 2000.

\bibitem[Boffi et~al.(2020)Boffi, Tu, and Slotine]{boffi2020regret}
Nicholas~M Boffi, Stephen Tu, and Jean-Jacques~E Slotine.
\newblock Regret bounds for adaptive nonlinear control.
\newblock \emph{arXiv preprint arXiv:2011.13101}, 2020.

\bibitem[Borodin and El-Yaniv(2005)]{borodin2005online}
Allan Borodin and Ran El-Yaniv.
\newblock \emph{Online computation and competitive analysis}.
\newblock cambridge university press, 2005.

\bibitem[Borrelli et~al.(2017)Borrelli, Bemporad, and
  Morari]{borrelli2017predictive}
Francesco Borrelli, Alberto Bemporad, and Manfred Morari.
\newblock \emph{Predictive control for linear and hybrid systems}.
\newblock Cambridge University Press, 2017.

\bibitem[{Braatz} et~al.(1994){Braatz}, {Young}, {Doyle}, and
  {Morari}]{Braatzdoylecomplexity}
R.~P. {Braatz}, P.~M. {Young}, J.~C. {Doyle}, and M.~{Morari}.
\newblock Computational complexity of /spl mu/ calculation.
\newblock \emph{IEEE Transactions on Automatic Control}, 39\penalty0
  (5):\penalty0 1000--1002, 1994.
\newblock \doi{10.1109/9.284879}.

\bibitem[Brockman et~al.(2016{\natexlab{a}})Brockman, Cheung, Pettersson,
  Schneider, Schulman, Tang, and Zaremba]{brockman2016openai}
Greg Brockman, Vicki Cheung, Ludwig Pettersson, Jonas Schneider, John Schulman,
  Jie Tang, and Wojciech Zaremba.
\newblock Openai gym.
\newblock \emph{arXiv preprint arXiv:1606.01540}, 2016{\natexlab{a}}.

\bibitem[Brockman et~al.(2016{\natexlab{b}})Brockman, Cheung, Pettersson,
  Schneider, Schulman, Tang, and Zaremba]{openaigym}
Greg Brockman, Vicki Cheung, Ludwig Pettersson, Jonas Schneider, John Schulman,
  Jie Tang, and Wojciech Zaremba.
\newblock Openai gym, 2016{\natexlab{b}}.

\bibitem[Bubeck et~al.(2020)Bubeck, Klartag, Lee, Li, and
  Sellke]{bubeck2020chasing}
S{\'e}bastien Bubeck, Bo'az Klartag, Yin~Tat Lee, Yuanzhi Li, and Mark Sellke.
\newblock Chasing nested convex bodies nearly optimally.
\newblock In \emph{Proceedings of the Fourteenth Annual ACM-SIAM Symposium on
  Discrete Algorithms}, pages 1496--1508. SIAM, 2020.

\bibitem[Chen et~al.(2018)Chen, Goel, and Wierman]{chen2018smoothed}
Niangjun Chen, Gautam Goel, and Adam Wierman.
\newblock Smoothed online convex optimization in high dimensions via online
  balanced descent.
\newblock In \emph{Conference On Learning Theory (COLT)}, 2018.

\bibitem[Chen and Hazan(2020)]{chen2020black}
Xinyi Chen and Elad Hazan.
\newblock Black-box control for linear dynamical systems.
\newblock \emph{arXiv preprint arXiv:2007.06650}, 2020.

\bibitem[Cohen et~al.(2018)Cohen, Hasidim, Koren, Lazic, Mansour, and
  Talwar]{cohen2018online}
Alon Cohen, Avinatan Hasidim, Tomer Koren, Nevena Lazic, Yishay Mansour, and
  Kunal Talwar.
\newblock Online linear quadratic control.
\newblock In \emph{International Conference on Machine Learning}, pages
  1029--1038, 2018.

\bibitem[Corless and Leitmann(1981)]{corless1981continuous}
Martin Corless and George Leitmann.
\newblock Continuous state feedback guaranteeing uniform ultimate boundedness
  for uncertain dynamic systems.
\newblock \emph{IEEE Transactions on Automatic Control}, 26\penalty0
  (5):\penalty0 1139--1144, 1981.

\bibitem[Dahleh et~al.(1993)Dahleh, Theodosopoulos, and
  Tsitsiklis]{dahleh1993sample}
Munther~A Dahleh, Theodore~V Theodosopoulos, and John~N Tsitsiklis.
\newblock The sample complexity of worst-case identification of fir linear
  systems.
\newblock \emph{Systems \& control letters}, 20\penalty0 (3):\penalty0
  157--166, 1993.

\bibitem[Dean et~al.(2017)Dean, Mania, Matni, Recht, and Tu]{dean2017sample}
Sarah Dean, Horia Mania, Nikolai Matni, Benjamin Recht, and Stephen Tu.
\newblock On the sample complexity of the linear quadratic regulator.
\newblock \emph{Foundations of Computational Mathematics}, pages 1--47, 2017.

\bibitem[Dean et~al.(2018)Dean, Mania, Matni, Recht, and Tu]{dean2018regret}
Sarah Dean, Horia Mania, Nikolai Matni, Benjamin Recht, and Stephen Tu.
\newblock Regret bounds for robust adaptive control of the linear quadratic
  regulator.
\newblock In \emph{Advances in Neural Information Processing Systems}, pages
  4188--4197, 2018.

\bibitem[Donti et~al.(2021)Donti, Roderick, Fazlyab, and
  Kolter]{donti2021enforcing}
Priya~L Donti, Melrose Roderick, Mahyar Fazlyab, and J~Zico Kolter.
\newblock Enforcing robust control guarantees within neural network policies.
\newblock In \emph{International Conference on Learning Representations
  (ICLR)}, 2021.

\bibitem[Dudley(1987)]{dudleypacking87}
R.~M. Dudley.
\newblock Universal donsker classes and metric entropy.
\newblock \emph{The Annals of Probability}, 15\penalty0 (4):\penalty0
  1306--1326, 1987.
\newblock ISSN 00911798.
\newblock URL \url{http://www.jstor.org/stable/2244004}.

\bibitem[Dulac-Arnold et~al.(2019)Dulac-Arnold, Mankowitz, and
  Hester]{dulac2019challenges}
Gabriel Dulac-Arnold, Daniel Mankowitz, and Todd Hester.
\newblock Challenges of real-world reinforcement learning.
\newblock \emph{arXiv preprint arXiv:1904.12901}, 2019.

\bibitem[Fiechter(1997)]{fiechter1997pac}
Claude-Nicolas Fiechter.
\newblock Pac adaptive control of linear systems.
\newblock In \emph{Proceedings of the tenth annual conference on Computational
  learning theory}, pages 72--80, 1997.

\bibitem[Freeman and Kokotovic(2008)]{freeman2008robust}
Randy Freeman and Petar~V Kokotovic.
\newblock \emph{Robust nonlinear control design: state-space and Lyapunov
  techniques}.
\newblock Springer Science \& Business Media, 2008.

\bibitem[Friedman and Linial(1993)]{friedman1993convex}
Joel Friedman and Nathan Linial.
\newblock On convex body chasing.
\newblock \emph{Discrete \& Computational Geometry}, 9\penalty0 (3):\penalty0
  293--321, 1993.

\bibitem[Goel and Wierman(2019)]{goel2019online}
Gautam Goel and Adam Wierman.
\newblock An online algorithm for smoothed regression and lqr control.
\newblock In \emph{International Conference on Artificial Intelligence and
  Statistics (AISTATS)}, 2019.

\bibitem[{Gurriet} et~al.(2018){Gurriet}, {Mote}, {Ames}, and
  {Feron}]{GurrietAmes2018}
T.~{Gurriet}, M.~{Mote}, A.~D. {Ames}, and E.~{Feron}.
\newblock An online approach to active set invariance.
\newblock In \emph{2018 IEEE Conference on Decision and Control (CDC)}, pages
  3592--3599, Dec 2018.
\newblock \doi{10.1109/CDC.2018.8619139}.

\bibitem[Hazan et~al.(2020)Hazan, Kakade, and Singh]{hazan2019nonstochastic}
Elad Hazan, Sham~M Kakade, and Karan Singh.
\newblock The nonstochastic control problem.
\newblock In \emph{Conference on Algorithmic Learning Theory (ALT)}, 2020.

\bibitem[Hespanha et~al.(2003)Hespanha, Liberzon, and
  Morse]{hespanha2003overcoming}
Joao~P Hespanha, Daniel Liberzon, and A~Stephen Morse.
\newblock Overcoming the limitations of adaptive control by means of
  logic-based switching.
\newblock \emph{Systems \& control letters}, 49\penalty0 (1):\penalty0 49--65,
  2003.

\bibitem[Hill et~al.(2018)Hill, Raffin, Ernestus, Gleave, Traore, Dhariwal,
  Hesse, Klimov, Nichol, Plappert, et~al.]{hill2018stable}
Ashley Hill, Antonin Raffin, Maximilian Ernestus, Adam Gleave, Rene Traore,
  Prafulla Dhariwal, Christopher Hesse, Oleg Klimov, Alex Nichol, Matthias
  Plappert, et~al.
\newblock Stable baselines.
\newblock \emph{GitHub repository}, 2018.

\bibitem[Ho and Doyle(2019)]{ho2019robust}
Dimitar Ho and John~C Doyle.
\newblock Robust model-free learning and control without prior knowledge.
\newblock In \emph{2019 IEEE 58th Conference on Decision and Control (CDC)},
  pages 4577--4582. IEEE, 2019.

\bibitem[Ioannou and Fidan(2006)]{ioannou2006adaptive}
Petros Ioannou and Bari{\c{s}} Fidan.
\newblock \emph{Adaptive control tutorial}.
\newblock SIAM, 2006.

\bibitem[Ioannou and Sun(2012)]{ioannou2012robust}
Petros~A Ioannou and Jing Sun.
\newblock \emph{Robust adaptive control}.
\newblock Courier Corporation, 2012.

\bibitem[Jiang et~al.(1999)Jiang, Sontag, and Wang]{ISSSontag}
Zhong-Ping Jiang, Eduardo Sontag, and Yuan Wang.
\newblock Input-to-state stability for discrete-time nonlinear systems.
\newblock \emph{IFAC Proceedings Volumes}, 32\penalty0 (2):\penalty0 2403 --
  2408, 1999.
\newblock 14th IFAC World Congress 1999, Beijing, Chia, 5-9 July.

\bibitem[Kakade et~al.(2020)Kakade, Krishnamurthy, Lowrey, Ohnishi, and
  Sun]{kakade2020information}
Sham Kakade, Akshay Krishnamurthy, Kendall Lowrey, Motoya Ohnishi, and Wen Sun.
\newblock Information theoretic regret bounds for online nonlinear control.
\newblock In \emph{Advances in Neural Information Processing Systems
  (NeurIPS)}, 2020.

\bibitem[Khalil and Grizzle(2002)]{khalil2002nonlinear}
Hassan~K Khalil and Jessy~W Grizzle.
\newblock \emph{Nonlinear systems}, volume~3.
\newblock Prentice hall Upper Saddle River, NJ, 2002.

\bibitem[Koutsoupias and Papadimitriou(2000)]{koutsoupias2000beyond}
Elias Koutsoupias and Christos~H Papadimitriou.
\newblock Beyond competitive analysis.
\newblock \emph{SIAM Journal on Computing}, 30\penalty0 (1):\penalty0 300--317,
  2000.

\bibitem[Krstic et~al.(1995)Krstic, Kokotovic, and
  Kanellakopoulos]{backstepkrstic}
Miroslav Krstic, Petar~V. Kokotovic, and Ioannis Kanellakopoulos.
\newblock \emph{Nonlinear and Adaptive Control Design}.
\newblock John Wiley \& Sons, Inc., 1995.

\bibitem[Liu et~al.(2009)Liu, Su, Yao, and Chu]{liu2009adaptive}
Xiangbin Liu, Hongye Su, Bin Yao, and Jian Chu.
\newblock Adaptive robust control of nonlinear systems with dynamic
  uncertainties.
\newblock \emph{International Journal of Adaptive Control and Signal
  Processing}, 23\penalty0 (4):\penalty0 353--377, 2009.

\bibitem[Ljung(1999)]{lennart1999system}
Lennart Ljung.
\newblock System identification: theory for the user.
\newblock \emph{PTR Prentice Hall, Upper Saddle River, NJ}, 28, 1999.

\bibitem[Mania et~al.(2020)Mania, Jordan, and Recht]{mania2020active}
Horia Mania, Michael~I Jordan, and Benjamin Recht.
\newblock Active learning for nonlinear system identification with guarantees.
\newblock \emph{arXiv preprint arXiv:2006.10277}, 2020.

\bibitem[Mayne et~al.(2011)Mayne, Kerrigan, Van~Wyk, and Falugi]{mayne2011tube}
David~Q Mayne, Erric~C Kerrigan, EJ~Van~Wyk, and P~Falugi.
\newblock Tube-based robust nonlinear model predictive control.
\newblock \emph{International Journal of Robust and Nonlinear Control},
  21\penalty0 (11):\penalty0 1341--1353, 2011.

\bibitem[Milanese et~al.(2013)Milanese, Norton, Piet-Lahanier, and
  Walter]{milanese2013bounding}
Mario Milanese, John Norton, H{\'e}l{\`e}ne Piet-Lahanier, and {\'E}ric Walter.
\newblock \emph{Bounding approaches to system identification}.
\newblock Springer Science \& Business Media, 2013.

\bibitem[Moore and Tedrake(2014)]{moore2014adaptive}
Joseph Moore and Russ Tedrake.
\newblock Adaptive control design for underactuated systems using
  sums-of-squares optimization.
\newblock In \emph{2014 American Control Conference}, pages 721--728. IEEE,
  2014.

\bibitem[Murray(2017)]{murray2017mathematical}
Richard~M Murray.
\newblock \emph{A mathematical introduction to robotic manipulation}.
\newblock CRC press, 2017.

\bibitem[Nguyen and Dankowicz(2015)]{nguyen2015adaptive}
Kim-Doang Nguyen and Harry Dankowicz.
\newblock Adaptive control of underactuated robots with unmodeled dynamics.
\newblock \emph{Robotics and Autonomous Systems}, 64:\penalty0 84--99, 2015.

\bibitem[Ortega and Spong(1989)]{OrtegaSpong89}
Romeo Ortega and Mark~W. Spong.
\newblock Adaptive motion control of rigid robots: A tutorial.
\newblock \emph{Automatica}, 25\penalty0 (6):\penalty0 877 -- 888, 1989.
\newblock ISSN 0005-1098.

\bibitem[Parrilo(2000)]{parrilo2000structured}
Pablo~A Parrilo.
\newblock \emph{Structured semidefinite programs and semialgebraic geometry
  methods in robustness and optimization}.
\newblock PhD thesis, California Institute of Technology, 2000.

\bibitem[Polycarpou and Ioannou(1993)]{polycarpou1993robust}
Marios~M Polycarpou and Petros~A Ioannou.
\newblock A robust adaptive nonlinear control design.
\newblock In \emph{1993 American Control Conference}, pages 1365--1369. IEEE,
  1993.

\bibitem[Rakovic and Baric(2010)]{rakovic2010parameterized}
Sa{\v{s}}a~V Rakovic and Miroslav Baric.
\newblock Parameterized robust control invariant sets for linear systems:
  Theoretical advances and computational remarks.
\newblock \emph{IEEE Transactions on Automatic Control}, 55\penalty0
  (7):\penalty0 1599--1614, 2010.

\bibitem[Rakovic et~al.(2006)Rakovic, Teel, Mayne, and
  Astolfi]{rakovic2006simple}
SV~Rakovic, AR~Teel, DQ~Mayne, and A~Astolfi.
\newblock Simple robust control invariant tubes for some classes of nonlinear
  discrete time systems.
\newblock In \emph{Proceedings of the 45th IEEE Conference on Decision and
  Control}, pages 6397--6402. IEEE, 2006.

\bibitem[Rudin et~al.(1976)]{rudin1976principles}
Walter Rudin et~al.
\newblock \emph{Principles of mathematical analysis}, volume~3.
\newblock McGraw-hill New York, 1976.

\bibitem[Sastry(2013)]{sastry2013nonlinear}
Shankar Sastry.
\newblock \emph{Nonlinear systems: analysis, stability, and control},
  volume~10.
\newblock Springer Science \& Business Media, 2013.

\bibitem[Schulman et~al.(2017)Schulman, Wolski, Dhariwal, Radford, and
  Klimov]{schulman2017proximal}
John Schulman, Filip Wolski, Prafulla Dhariwal, Alec Radford, and Oleg Klimov.
\newblock Proximal policy optimization algorithms, 2017.

\bibitem[Sellke(2020)]{sellke2020chasing}
Mark Sellke.
\newblock Chasing convex bodies optimally.
\newblock In \emph{Proceedings of the Fourteenth Annual ACM-SIAM Symposium on
  Discrete Algorithms}, pages 1509--1518. SIAM, 2020.

\bibitem[Shi et~al.(2020)Shi, Lin, Chung, Yue, and Wierman]{shi2020online}
Guanya Shi, Yiheng Lin, Soon-Jo Chung, Yisong Yue, and Adam Wierman.
\newblock Online optimization with memory and competitive control.
\newblock In \emph{Advances in Neural Information Processing Systems
  (NeurIPS)}, 2020.

\bibitem[Simchowitz et~al.(2018)Simchowitz, Mania, Tu, Jordan, and
  Recht]{simchowitz2018learning}
Max Simchowitz, Horia Mania, Stephen Tu, Michael~I Jordan, and Benjamin Recht.
\newblock Learning without mixing: Towards a sharp analysis of linear system
  identification.
\newblock In \emph{Conference On Learning Theory}, pages 439--473. PMLR, 2018.

\bibitem[Slotine et~al.(1991)Slotine, Li, et~al.]{slotine1991applied}
Jean-Jacques~E Slotine, Weiping Li, et~al.
\newblock \emph{Applied nonlinear control}, volume 199.
\newblock Prentice hall Englewood Cliffs, NJ, 1991.

\bibitem[{Spong} and {Vidyasagar}(1987)]{spongrobustlinear87}
M.~{Spong} and M.~{Vidyasagar}.
\newblock Robust linear compensator design for nonlinear robotic control.
\newblock \emph{IEEE Journal on Robotics and Automation}, 3\penalty0
  (4):\penalty0 345--351, 1987.

\bibitem[{Spong}(1992{\natexlab{a}})]{Spongrobust92}
M.~W. {Spong}.
\newblock On the robust control of robot manipulators.
\newblock \emph{IEEE Transactions on Automatic Control}, 37\penalty0
  (11):\penalty0 1782--1786, 1992{\natexlab{a}}.

\bibitem[{Spong}(1992{\natexlab{b}})]{Spongrobustcontrol}
M.~W. {Spong}.
\newblock On the robust control of robot manipulators.
\newblock \emph{IEEE Transactions on Automatic Control}, 37\penalty0
  (11):\penalty0 1782--1786, 1992{\natexlab{b}}.
\newblock \doi{10.1109/9.173151}.

\bibitem[Tassa et~al.(2018)Tassa, Doron, Muldal, Erez, Li, de~Las~Casas,
  Budden, Abdolmaleki, Merel, Lefrancq, Lillicrap, and
  Riedmiller]{tassa2018deepmind}
Yuval Tassa, Yotam Doron, Alistair Muldal, Tom Erez, Yazhe Li, Diego
  de~Las~Casas, David Budden, Abbas Abdolmaleki, Josh Merel, Andrew Lefrancq,
  Timothy Lillicrap, and Martin Riedmiller.
\newblock Deepmind control suite, 2018.

\bibitem[Tedrake(2020)]{underactuated}
Russ Tedrake.
\newblock Underactuated robotics: Algorithms for walking, running, swimming,
  flying, and manipulation.
\newblock (Course Notes for MIT 6.832), 2020.
\newblock Downloaded on 2020-03-30 from http://underactuated.mit.edu.

\bibitem[Vaidyanathan et~al.(2016)Vaidyanathan, Volos,
  et~al.]{vaidyanathan2016advances}
Sundarapandian Vaidyanathan, Christos Volos, et~al.
\newblock \emph{Advances and applications in nonlinear control systems}.
\newblock Springer, 2016.

\bibitem[Yao and Tomizuka(1995)]{yao1995robust}
Bin Yao and Masayoshi Tomizuka.
\newblock Robust adaptive nonlinear control with guaranteed transient
  performance.
\newblock In \emph{Proceedings of 1995 American Control Conference-ACC'95},
  volume~4, pages 2500--2504. IEEE, 1995.

\bibitem[Yu et~al.(2020)Yu, Shi, Chung, Yue, and Wierman]{yu2020power}
Chenkai Yu, Guanya Shi, Soon-Jo Chung, Yisong Yue, and Adam Wierman.
\newblock The power of predictions in online control.
\newblock 2020.

\bibitem[Zhou and Doyle(1998)]{zhou1998essentials}
Kemin Zhou and John~Comstock Doyle.
\newblock \emph{Essentials of robust control}, volume 104.
\newblock Prentice hall Upper Saddle River, NJ, 1998.

\bibitem[Zhou et~al.(1996)Zhou, Doyle, and Glover]{zhou1996robust}
Kemin Zhou, John~C Doyle, and Keith Glover.
\newblock Robust and optimal control.
\newblock 1996.

\bibitem[Ziemann and Sandberg(2020)]{Zieman}
Ingvar Ziemann and Henrik Sandberg.
\newblock {Regret Lower Bounds for Unbiased Adaptive Control of Linear
  Quadratic Regulators}.
\newblock working paper or preprint, February 2020.
\newblock URL \url{https://hal.archives-ouvertes.fr/hal-02404014}.

\end{thebibliography}
\clearpage
\appendix 

\section{Examples}\label{app:instantiation}
Here we demonstrate couple examples of how the framework presented in \secref{sec:mainresult} can be applied to analyze and design learning and control agents with worst-case guarantees. We first discuss the problem of stabilizing an uncertain scalar linear system. In the later section we discuss implications for learning and control of uncertain robotic systems: We show a simple design for $\cl{A}_\pi(\sel)$ agents with finite mistake guarantees; we use well-known control methods in robotics to design robust oracles $\pi$ and couple it with $\sel$ selections based on competitive (NCBC) algorithms.
\subsection{Control of uncertain scalar linear system}
Let's consider the very basic problem setting, where we are given an unknown scalar linear system $$x_{k+1}= \alpha^* x_k+\beta^* u_k+w_k=:f^*(k,x_k,u_k),$$ s.t. $|w_k| \leq \gamma^*\leq \eta <1$ and $\alpha^* \in [-a,a]$, $\beta^*\in [1,1+ 2b_{\Delta}]$, and our goal is to reach the target interval $\mathcal{X}_{\mathcal{T}} = [-1,1]$ and remain there. We can equivalently phrase this as to achieve the objective $\bm{\cl{G}} = (\cl{G}_0, \cl{G}_1,\dots)$ with cost functions $$\cl{G}_t(x,u):= \left\{\begin{array}{cl} 0,&\quad \text{if }|x| \leq 1\\ 1, &\quad\text{else} \end{array}\right., \quad \forall t \geq 0$$
after finitely many mistakes. 

\paragraph{Compact parametrization of uncertainty set. }
We define our parameter space as $\mathsf{K} = [-a,a]\times [1,1+2b_\Delta]$, and define the true parameter as $\theta^*=(\theta^*_x,\theta^*_u)=(\alpha^*,\beta^*)$ and parametrize the uncertainty set as $\cl{F}=\cup_{\theta\in\st{K}}\mathbb{T}[\theta]$ with 
$$\mathbb{T}[\theta] := \{t,x,u\mapsto \theta_x x + \theta_u u +w_t\:|\: \|\bm{w}\|_{\infty} \leq \eta \} $$
 We choose the metric as $d(\theta,\theta'):= |\theta_x-\theta'_x|+a|\theta_u-\theta'_u|$.
 The diameter of the metric space $(\mathsf{K},d)$ is $\phi(\mathsf{K})=d((-a,1),(a,1+2b_{\Delta})) = 2(a+b_{\Delta})$. 

\paragraph{A locally uniformly robust oracle. }
 As an oracle, we take the simple deadbeat controller: $\pi[\theta](t,x):= -(\theta_x/\theta_u)x$. It can be easily shown that $\pi$ is a locally $\rho$-uniformly robust oracle for $\bm{\cl{G}}$ for any margin in the interval $(0, \ol{\rho})$, $\ol{\rho}:=1-\eta$, by noticing the inequality:
 \begin{align}
    |x_{t+1}| &\leq |\theta^*_x x_t + \theta^*_u \pi[\theta_t](t,x_t)| + \eta = |((\theta^*_x-\theta_{x,t}) - (\theta^*_u-\theta_{u,t}) \tfrac{\theta_{x,t}}{\theta_{u,t}})x_t| + \eta\\ 
    & \leq (|\theta^*_x-\theta_{x,t}| + |\theta^*_u-\theta_{u,t}||\tfrac{\theta_{x,t}}{\theta_{u,t}}|)|x_t| + \eta \leq d(\theta^*,\theta_t)|x_t| + \eta
 \end{align}
 To obtain the mistake function $M^\pi_\rho$ for a fixed $\rho\in (0,1-\eta)$, notice that if 
 $d(\theta^*,\theta_t)\leq \rho$, then 
 \begin{align}
|x_{t+1}| \leq \rho |x_t| + \eta = \rho |x_t| + (1-\rho)\tfrac{1}{1-\rho} \eta 
\quad&\Leftrightarrow\quad |x_{t+1}|-\tfrac{\eta}{1-\rho}\leq \rho(|x_t| -\tfrac{\eta}{1-\rho}) \\
&\Rightarrow \quad |x_t| \leq \tfrac{\eta}{1-\rho} + \rho^t (|x_0| - \tfrac{\eta}{1-\rho})
 \end{align}
 Notice that $$ \tfrac{\eta}{1-\rho} + \rho^t|x_0|<1 \Leftrightarrow t> \tfrac{\log(|x_0|)}{\log(\rho^{-1})} + \underbrace{\tfrac{\log(1-\rho) - \log(1-\rho - \eta)}{\log(\rho^{-1})}}_{c(\rho)} $$
 which implies the mistake function 
 \begin{align}
    M^\pi_\rho(\gamma) \leq \tfrac{\log(\gamma)}{\log(\rho^{-1})}+c(\rho)
 \end{align}
\paragraph{ Construction of consistent sets via LPs.}
The consistent set for the data set $\cl{D}_N$ of $N$ observed system transitions $(x^+_i,x_i,u_i)$ can be written as an intersection of $\st{K}$ with $2N$ halfspaces: 
\begin{align*}
\st{P}(\cl{D}_N) = \left\{\theta\in \st{K}\:|\: \text{s.t.: } \forall 1\leq i\leq N:\: x^+_i - \eta  \leq \theta_x x_i + \theta_u u_i \leq x^+_i + \eta  \right\},
\end{align*}
It can be constructed online and is convex.
\paragraph{ Competitive consistent model chasing via steiner point.}
We can construct a competitive CMC-algorithm by using algorithms for competitive NCBC. Assume we use the steiner point and denote the selection procedure $\sel_s$ as in \eqref{eq:sel}. 
 $\sel_s$ is a $\tfrac{n}{2} = 1$-competitive CMC algorithm in euclidean space and since the euclidean norm is bounded above by the 1-norm, $\sel_s$ is also $1$-competitive w.r.t. the metric space $(\st{K},d)$. 
\paragraph{Mistake guarantee for $\cl{A}_\pi(\sel_s)$} 
We apply the extension of the results in \thmref{thm:redext}, \ref{it:conv4}. It is easy to see that our $\pi$ satisfies the extra condition with $\alpha = 1$, $\beta = \eta$. Assuming $|x_0| = 0$, the constant $\gamma_\infty$ takes the value
\begin{align}
    \gamma_\infty = e^{\alpha \phi(\mathsf{K})}(\|x_0\| + \beta \tfrac{e}{e -1}) =  \tfrac{\eta e}{e -1}e^{\phi(\mathsf{K})}.
\end{align}
For ease of exposition, assume that $\eta=e^{-1}$ and that we picked $\rho = e^{-1}$. This gives us $M^\pi_\rho(\gamma_\infty)=\phi(\mathsf{K})-\log(e-2)$ and substituting all constants gives us a finite mistake guarantee for the objective $\bm{\cl{G}}$:
\begin{align}
    \sum^{\infty}_{t=0}{\cl{G}}_t(x_t,u_t) \leq M^\pi_\rho(\gamma_\infty)\Big(\tfrac{2L}{\rho} +1\Big) \approx  \phi(\mathsf{K})(1+2e\phi(\mathsf{K})) = 8e(a+{b_\Delta})^2 + 2(a+b_{\Delta})
\end{align}
The above inequality shows that the worst-case total number of mistakes grows quadratically with the size of the initial uncertainty in the system parameters $\theta_x$ and $\theta_u$. Notice, however that the above inquality holds for arbitrary large choices of $a$ and $b_\Delta$. Thus, $\cl{A}_\pi(\sel_s)$ gives finite mistake guarantees for this problem setting for arbitrarily large system parameter uncertainties.

\subsection{Learning to follow a trajectory for a class of robotic systems}

Consider a general case of online control of uncertain fully-actuated robotic systems. A vast majority of robotic systems can be modeled via the robotic equation of motion \citep{murray2017mathematical}:
\begin{equation}
\label{eq:equation_of_motion}
    \mathbf{M}_\eta(q)\ddot{q}+\mathbf{C}_\eta(q,\dot{q})\dot{q} + \mathbf{N}_\eta(q,\dot{q}) = \tau + \tau_d
\end{equation}
where $q\in\R^{n}$ is the multi-dimensional generalized coordinates of the system, $\dot{q}$ and $\ddot{q}$ are its first and second (continuous) time derivatives, $\mathbf{M}_\eta(q), \mathbf{C}_\eta(q,\dot{q}), \mathbf{N}_\eta(q,\dot{q})$ are matrix and vector-value functions that depend on the parameters $\eta\in\mathbb{R}^m$ of the robotic system, i.e. $\eta$ comes from a parametric physical model. 
 Often, $\tau$ is the control action (e.g., torques and forces of actuators), which acts as input of the system. Disturbances and other uncertainties present in the system can be modeled as additional torques $\tau_d \in \R^n$ perturbing the equations.  Moreover, one can derive from first principles \cite{murray2017mathematical}, that for many robotic systems (for example robot manipulators) the following two properties hold:
\begin{subequations}
    \label{eq:lagrange}
 \begin{align}
    \dot{\mathbf{M}_\eta}(q)-2\mathbf{C}_\eta(q,\dot{q})\text{ is skew-symmetric} \\
    \label{eq:laglab}\mathbf{M}_\eta(q)\ddot{q} + \mathbf{C}_\eta(q,\dot{q})\dot{q} + \mathbf{N}_\eta(q,\dot{q}) = \mathbf{Y}(q,\dot{q},\ddot{q})\eta
 = \tau +\tau_d \end{align}  
\end{subequations}  
The second equation says that the left-hand-side of equation \eqref{eq:equation_of_motion} can always be factored into a $n\times m$ matrix of known functions $\mathbf{Y}(q,\dot{q},\ddot{q})$ and a constant vector $\eta\in\R^m$.
Assume that the disturbances are bounded at each time $t$, as $|\tau_d(t)| \leq \omega$, $\omega\in\R^n$ and where the inequality is to be read entry-wise. Consider that we are given a system with unknown $\eta^*$, $\omega^*$, where the parameter $\theta^*=[\eta^*;\omega^*]$ is known to be contained in a bounded set $\st{K}$. Assume that our goal is to track a desired trajectory $q_d$, which is given as a function of time $q_d:\mathbb{R}\mapsto \R^{n}$, within $\epsilon$ precision. Denoting $x = [q^\top,\dot{q}^\top]^\top$ as the state vector and $x_d = [q^\top_d,\dot{q}^\top_d]^\top$ as the desired state, we want the system state trajectory $x(t)$ to satisfy:
 \begin{align}
    \limsup\limits_{t\rightarrow \infty}\left\|  x(t)-x_d(t)\right\|\leq \epsilon
\end{align}
As common in practice, we assume we can observe the sampled measurements $x_k :=x(t_k)$, $x^d_k :=x^d(t_k)$ 
 and apply a constant control action (zero-order-hold actuation) $\tau_k:=\tau(t_k)$ at the discrete time-steps $t_k=k T_s$ with small enough sampling-time $T_s$ to allow for continuous-time control design and analysis. 

\paragraph{Control objective $\bm{\cl{G}}^{\epsilon}$.}
We phrase trajectory tracking as a control objective $\bm{\cl{G}}^{\ep}$ with the cost functions 
$$\cl{G}^\epsilon_k(x,u):= \left\{\begin{array}{cl} 0,&\quad \text{if }\|x-x^d_k\| \leq \epsilon\\ 1, &\quad\text{else} \end{array}\right., \quad \forall k \geq 0,$$
which we wish to achieve online with finite mistake guarantees against the uncertainty set $\cl{F}$:
\begin{align*}
    \cl{F}=\bigcup_{\theta\in\st{K}} \mathbb{T}[\theta],\text{ where }\mathbb{T}[\theta]:=\{k,x_k,\tau_k \mapsto f^*(x_k,\tau_k,\tau_d(\anon);\theta)\:|\:\tau_d:[0,T_s]\mapsto \R^n, \|\tau_d\|_\infty \leq \omega \},
\end{align*} 
The function $f^*$ denotes the discretized dynamics of \eqref{eq:equation_of_motion} w.r.t. the sampling time $T_s$.

\paragraph{Robust oracle design.} We outline how to design a robust oracle based on a well-established robust control method for robotic manipulators  proposed in \cite{Spongrobustcontrol}. 
Define $v$, $a$ and $r$ as the quantities
\begin{align}\label{eq:spongdefs}
v = \dot{q}_d-\Lambda \dot{q},\quad a=\dot{v},\quad r = \dot{\tilde{q}} + \Lambda \tilde{q}, \quad \tilde{q}=q-q_d
\end{align}
and denote $\mathbf{Y}'(q,\dot{q},v,a)$ as the corresponding $n\times m$ matrix which allows the factorization:
\begin{align}
    \mathbf{M}_\eta(q)a + \mathbf{C}_\eta(q,\dot{q})v + \mathbf{N}_\eta(q,\dot{q}) = \mathbf{Y}'(q,\dot{q},v,a)\eta
\end{align}
Based on the control law presented in \cite{Spongrobustcontrol}, we define the oracle $\pi[\theta](k,x_k)$ for $x = [q;\dot{q}]$ and $\theta=[\eta;\omega]$ through the equations:
\begin{align}\label{eq:oracspong}
\pi[\theta](k,x_k) &= \mathbf{Y}'(q_k,\dot{q}_k,v_k,a_k)(\eta + u_k) - K_\omega r_k, \quad u = \left\{\begin{array}{lc} -\rho\frac{\mathbf{Y}'^\top r_k}{\|\mathbf{Y}'^\top r_k\|_2} & \text{if }\|\mathbf{Y}'^\top r_k\|_2 > \ep \\ -\frac{\rho}{\ep} \mathbf{Y}'^\top r_k   & \text{if }\|\mathbf{Y}'^\top r_k\|_2 \leq \ep\end{array} \right.
\end{align}
where $\Lambda, K_\omega \succ 0$ are diagonal positive definite design and where $\rho$, $\ep$ are design variables.
Following the the analysis in \cite{Spongrobustcontrol} and  \cite{corless1981continuous} one can design a suiting gain $K_\omega$ in terms of $\omega$, such that $\pi$ is a uniformly $\rho$ robust oracle for $\bm{\cl{G}}^\epsilon$ in the compact parametrization $(\mathbb{T},\st{K},d)$.
\begin{rem}
    The analysis in \cite{Spongrobustcontrol} shows that uniform ultimate boundedness properties of the tracking error $\tilde{x} = [q-q_d;\dot{q}-\dot{q}_d]$ are preserved, if we replace $\eta$ in equation \eqref{eq:oracspong} with some perturbation $\eta + \delta(t)$, $\|\delta(t)\|_2 \leq \rho$ for all $t$. In \cite{Spongrobustcontrol}, the disturbance $\tau_d$ is assumed zero, i.e. the $\omega=0$ case, and the gain $K_0$ is left as a tuning variable. However, with standard Lyapunov arguments the analysis of \cite{Spongrobustcontrol} can be extended to consider the non-zero disturbance case and specify gains $K_\omega$ for each $\omega$ such that the above oracle $\pi$ becomes a uniformly $\rho$ robust oracle for the above objective $\bm{\cl{G}}^\epsilon$: For each $\omega$, increase the gain $K_\omega$ until the uniform ultimate boundedness guarantee implies the desired $\epsilon$-tracking behavior described by $\bm{\cl{G}}^\epsilon$.
 \end{rem}

\paragraph{Constructing consistent sets.}
The linear factorization property \eqref{eq:laglab} can be exploited to construct convex consistent sets. Denote $\mathbb{T}$ as an uncertain robotic system \eqref{eq:lagrange} with some convex compact uncertainty $\st{K}$ in euclidean space $(\R^{m+n},\|\anon\|_2)$. Recall that we parameterize the bound on the disturbance by $\omega\in\R^n$, i.e., $\lvert\tau_d\rvert\leq \omega$ holds entry-wise and that our system parameter is represented by $\theta = [\eta^\top, w^\top]^\top \in \st{K}$. At the sampled time-steps $t_k$, equations \eqref{eq:laglab} says that the measurements $q_k,\dot{q}_k, \ddot{q}_k, \tau_k$ enforce the following entry-wise condition on consistent parameters $\eta$ and $\omega$:
\begin{align}
    \tau_k - \omega \leq \mathbf{Y}(q_k,\dot{q}_k,\ddot{q}_k)\eta \leq \tau_k + \omega.
\end{align}
\noindent In matrix form, the consistent set is captured via the following relationship:
\begin{align}
 \underbrace{\begin{bmatrix}\mathbf{Y}(q_k,\dot{q}_k,\ddot{q}_k) & -\mathbf{I}_n \\ -\mathbf{Y}(q_k,\dot{q}_k,\ddot{q}_k) & -\mathbf{I}_n \end{bmatrix}}_{\mathbf{A}_k}\begin{bmatrix} \eta\\ \omega\end{bmatrix} \leq \underbrace{\begin{bmatrix} \tau_k \\ -\tau_k \end{bmatrix}}_{\mathbf{b}_k}
\end{align}
Consequently, we have a concrete construction of consistent set at each time $t$:
\begin{equation}\label{eq:roboconsist}
\st{P}(\cl{D}_t) = \bigg\{ \theta = \begin{bmatrix} \eta \\ w \end{bmatrix} \in\mathbb{R}^{m+n}\bigg| \mathbf{A}_t \begin{bmatrix} \eta \\ w \end{bmatrix} \leq \mathbf{b}_t \bigg\}     \cap \st{P}(\cl{D}_{t-1}),\quad \st{P}(\cl{D}_{0}) = \st{K}
\end{equation}
where $\mathbf{A}_k = \mathbf{A}_k(x_k,u_k)$ and $\mathbf{b}_k = \mathbf{b}_k(x_k,u_k)$ are matrix and vector of ``features'' constructed from current control policy and state at time $t$ via the known functional form of $\mathbf{Y}$. The data sets $\cl{D}_k$ are tuples of the form $ \cl{D}_k = (d_1,\dots, d_k),\: d_k = (q_k,\dot{q}_k, \ddot{q}_k, \tau_k).$

\paragraph{Designing a competitive chasing selection.}
The above consistent sets are simply an intersection of halfspaces, hence we are in the setting of \aspref{asp:convexuncertain} and we can instantiate competitive selections from the (NCBC) competitive greedy and Steiner point algorithm algorithms:
\begin{itemize}[nolistsep]
    \item \textbf{Greedy projection.} $\sel_p$ selects $\theta_t=\texttt{SEL}_p(\cl{D}_t)$ as the solution to the following convex optimization problem, which can be solved efficiently:
\begin{align*}
\theta_t= \underset{\theta\in\mathbb{R}^p \cap \st{K}}{\argmin}& \quad \frac{1}{2}\norm{\theta-\theta_{t-1}}^2,\\
\text{s.t:}&\quad  \mathbf{A}_i \theta \leq \mathbf{b}_i, \forall i = 1, \ldots, t.
\end{align*}
    \item \textbf{Steiner-Point.} Alternatively, $\texttt{SEL}_s$ outputs the Steiner point of the polyhedron $\st{P}_{\mathbb{T}}(\cl{D}_t)$, which in principle requires calculating an integral over multi-dimensional sphere. Fortunately, as shown in \cite{argue2020chasing}, the Steiner point can be approximated efficiently by solving randomized linear programs. (We take this approach in our empirical validation.)
\end{itemize}
\paragraph{Mistake guarantee for $\cl{A}_\pi(\sel_{p/s})$} 
Since $\pi$ is a robust oracle for $\bm{\cl{G}}$ and both $\sel_{p}$ and $\sel_s$ are $\gamma$-competitive CMC algorithms in $(\mathbb{T},\st{K},\|\anon\|)$  for some $\gamma>0$, our result \thmref{thm:red} tells us that $\cl{A}_\pi(\sel_p)$ and $\cl{A}_\pi(\sel_s)$ guarantees upfront finiteness of the total number of mistakes $\sum^{\infty}_{k=0} \cl{G}^\epsilon_k(x_k,\tau_k)$, which implies the desired tracking behavior guarantee $\limsup_{k\rightarrow \infty}\|x-x^d\| \leq \epsilon$. Moreover, if we can provide a bound $M$ on the mistake constant $M^\pi_\rho< M$, we obtain from \thmref{thm:red}, \ref{it:conv4} an explicit performance bound for the tracking performance in the form of the mistake guarantee 
$$ \sum\limits^{\infty}_{k=0} \cl{G}^\epsilon_k(x_k,\tau_k)\leq M(\tfrac{2\gamma}{\rho}\dm(\st{K})+1).$$

\begin{algorithm}[t]
    \caption{ design of $\cl{A}_\pi(\texttt{SEL}_{p/s})$ for $\epsilon$-trajectory tracking for fully actuated robots}
    \label{alg:Apiselrobotic}
    \begin{small}
 \begin{algorithmic}[1]
    \For{$t=0,T_s,\dots,kT_s$ {\bfseries to} $\infty$}
    \State measure $q_k,\dot{q}_k,\ddot{q}_k$
    \State update polyhedron $\st{P}(\cl{D}_k)$ as in \eqref{eq:roboconsist}
    \State select according to \eqref{eq:sela} or \eqref{eq:selb}\Comment{selection $\sel_p$ or $\sel_s$}
    \State choose $\tau_k = \mathbf{Y}'(q_k,\dot{q}_k,v_k,a_k)(\eta_k + u_k) - K_{\omega_k} r_k$ using \eqref{eq:spongdefs}, \eqref{eq:oracspong} \Comment{ use oracle $\pi(x_k;\theta_k)$}
    \EndFor
 \end{algorithmic}
 \end{small}
 \end{algorithm}

\section{Proofs of the main results}
\label{sec:app_results}
For convenience of the readers, we recap briefly the most important definitions and corollaries used in the proofs.
\subsection{Overview of main definitions and notational conventions} \label{sec:gendef}
\begin{notset*}
    $2^\st{S}$ denotes the set of all subsets of a set $\st{S}$.
   A \textit{compact parametrization} $(\mathbb{T},\st{K},d)$ of an uncertainty set $\cl{F}$ consists of the compact metric space $(\st{K},d)$ and mapping $\mathbb{T}:\st{K}\mapsto 2^\cl{F}$ such that: $\cl{F} \subset \bigcup_{\theta \in \st{K}} \mathbb{T}(\theta)$. An \textit{objective} $\bm{\cl{G}}$ is a sequence of $\{0,1\}$-valued functions $(\cl{G}_0,\cl{G}_1,\dots)$. $\mathbb{D}$ is the space of finite data sets, $\cl{C}$ the space of time-varying state-feedback policies and $d_H:2^\st{K}\times 2^\st{K}\mapsto \mathbb{R}^+$ is the Hausdorff metric: 
   \begin{align*}
    \mathbb{D} &:=\left\{(d_1,\dots,d_N)\:|\:d_i \in \mathbb{N}\times\cl{X}\times\cl{X}\times \cl{U},\: N<\infty \right\},\\
    \cl{C} &:= \left\{\mathbb{N}\times \cl{X} \mapsto \cl{U} \right\}\\
    d_H(\st{S},\st{S}') &:= \max\left\{\max\limits_{x\in\st{S}} d(x,\st{S}')\:,\: \max\limits_{y\in\st{S}'} d(y,\st{S})\right\}\:\:\text{ for sets }\st{S},\st{S}'\subset \st{K}
    \end{align*} 
    A data sequence $\bm{\cl{D}}=(\cl{D}_1,\cl{D}_2,\dots)$ is a data stream if the data sets $\cl{D}_t=(d_1,\dots,d_t)$ are formed from a sequence of observations $(d_1,d_2,\dots)$.
    Procedure $\pi$ is a map $\st{K}\mapsto \cl{C}$ which expects a parameter $\theta\in\st{K}$ as input and returns a policy $\pi[\theta]\in\cl{C}$. Procedure $\sel$ is a map $\mathbb{D}\mapsto \st{K}$ which takes as input a data set $\cl{D}\in\mathbb{D}$ and returns a parameter $\theta=\sel[\cl{D}]$. For a fixed choice of $\sel$ and $\pi$, we instantiate $\cl{A}_\pi(\sel)$ as the online algorithm described in Algorithm \ref{alg:ApiselOCO}. For a time-interval $\cl{I}=[t_1,t_2]\subset \mathbb{N}$, ($x_{\cl{I}}$, $u_{\cl{I}}$) represent the tuples $(x_{t_1},\dots,x_{t_2})$, $(u_{t_1},\dots,u_{t_2})$. For a fixed objective $\bm{\cl{G}}$, the set of admissable states at time $t$ are $\st{X}_{t}= \{x\:|\: \exists u': \:\cl{G}_{t}(x,u')=0\}$. 
  \end{notset*}
\begin{defn}\label{app-def:pi-rhorobust}
    For a time-interval $\cl{I}=[t_1,t_2]\subset \mathbb{N}$ and fixed $\theta^*\in\st{K}$, the set $\cl{S}_{\cl{I}}[\rho;\theta^*]$ is defined as
    \begin{align*}
        \cl{S}_{\cl{I}}[\rho;\theta^*]:=\left\{(x_{\cl{I}}, u_{\cl{I}})\:\left|
        \begin{array}{c}\: \exists f\in\mathbb{T}(\theta^*)\exists(\theta_{t_1},\dots,\theta_{t_2})\subset \st{B}_\rho(\theta^*):\forall k\in\cl{I}:\\
            u_k=\pi[\theta_k](k,x_k), \text{ and }
         x_{k+1}=f(k,x_k,u_k),\text{ if }k\neq t_2
         \end{array}\right.\right\}
    \end{align*}
\end{defn}
\begin{defn*}
    A function $f:X\mapsto Y$ is a selection of the set-valued map $F:X \mapsto 2^Y$, if $\forall x\in X:\: f(x) \in F(x)$.
\end{defn*}
\subsubsection{Robust oracles}

\begin{table}[h]
    \centering
        \begin{tabular}{| r | c |} 
            \hline
            {$\rho$-robust} & $\forall \theta\in\st{K}:\: \sup_{\gamma \geq 0} m^\pi_\rho(\gamma;\theta)<\infty$ \TBstrut \\
            \hline
            {uniformly $\rho$-robust} & ${M}^\pi_\rho:=\sup_{\gamma \geq 0, \theta \in \st{K}} {m}^\pi_\rho(\gamma;\theta)<\infty$  \TBstrut \\
            \hline 
             {locally $\rho$-robust} & $\forall \gamma \geq 0,\: \theta \in \st{K}:\: m^\pi_\rho(\gamma;\theta)<\infty$ \TBstrut \\ 
             \hline
             {locally uniformly $\rho$-robust} & $\forall \gamma \geq 0:\:M^\pi_\rho(\gamma):= \sup_{\theta\in\st{K}} m^\pi_\rho(\gamma;\theta)<\infty$ \TBstrut \\
             \hline
        \end{tabular}
     \caption{Different notions of robustness of an oracle $\pi$ }
     \label{tab:robustness2}
\end{table}

\begin{defn}\label{app-def:pi-rhorobust}
    Equip $\cl{X}$ with some norm $\|\anon\|$ and consider a fix parameterization $(\mathbb{T},\st{K},d)$ for $\cl{F}$.
    Define the quantity $m^\pi_\rho(\gamma;\theta)$ for each $\gamma\geq 0$, $\rho\geq 0$ and $\theta\in\st{K}$ as
    $$m^\pi_\rho(\gamma;\theta) := \sup\limits_{\cl{I}=[t,t']\::\:t<t'}\:\:\sup\limits_{(x_\cl{I},u_\cl{I})\in\cl{S}_{\cl{I}}[\rho;\theta], \|x_t\| \leq \gamma}\sum\limits_{t\in\cl{I}}\cl{G}_t(x_t,u_t) $$
    If $m^\pi_0(\gamma;\theta)$ is finite everywhere, we call $\pi$ an \textit{oracle} for $\bm{\cl{G}}$ and call it (locally) (uniformally) $\rho$-robust, if the corresponding property presented in Table \ref{tab:robustness2} is satisfied. If it exists, ${M}^\pi_\rho$ will be called the \textit{mistake constant} or \textit{mistake function} (for local properties) of $\pi$.
\end{defn}
\begin{coro*}[Def.\ref{app-def:pi-rhorobust}]
    Let procedure $\pi$ be a uniformly $\rho$-robust oracle for $\bm{\cl{G}}$ in $\parttupl$. Let $\bm{x},\bm{u}$ be some trajectory and denote its restriction to some time-interval $\cl{I}$ as $(x_{\cl{I}},u_{\cl{I}})$. 
    If $(x_{\cl{I}},u_{\cl{I}})\in\bigcup_{\theta\in\st{K}}\cl{S}_{\cl{I}}[\rho;\theta]$, then $\sum_{t\in\cl{I}} \cl{G}_t(x_t,u_t) \leq M^\pi_\rho $. If instead, $\pi$ is locally uniformly $\rho$-robust, the previous inequality changes to $\sum_{t\in\cl{I}} \cl{G}_t(x_t,u_t) \leq M^\pi_\rho(\|x_{t_0}\|),\:t_0:=\min(\cl{I})$.
\end{coro*}
\begin{defn}\label{app-def:pi-invariant}
  For an objective $\bm{\cl{G}}$, we call a $\rho$-robust oracle $\pi$ \textit{cost invariant}, if for all ${\theta}\in\st{K}$ and $t\geq 0$ the following holds 
    \begin{itemize}[nosep]
        \item For all $x\in\st{X}_{t}$ holds $\cl{G}_t(x,\pi[{\theta}](t,x))=0$.   
         \item For all $x\in \st{X}_{t}$, $f\in\mathbb{T}[\theta]$ and $\theta' $ s.t. $d(\theta',\theta)\leq \rho$, holds $f(t,x,\pi[{\theta'}](t,x))\in\st{X}_{t+1}$, 
    \end{itemize}
\end{defn}
\begin{rem}
    The above condition is related to the well-known notion of \textit{positive set/tube invariance} in control theory \cite{blanchini1999set}:
    The above condition requires that the oracle policies $\pi[\theta]$ can ensure for their nominal model $\mathbb{T}[\theta]$ the following closed loop condition: $x_t\in\st{X}_t \implies x_{t+1}\in\st{X}_{t+1}$, $\forall t$. 
\end{rem}
\begin{coro*}[Def. \ref{app-def:pi-invariant}]
    Assume that $\pi$ is cost invariant as in Def.\ref{app-def:pi-invariant} and that $\pi$ is uniformly $\rho$-robust. Let $\bm{x},\bm{u}$ be some trajectory and denote its restriction to some time-interval $\cl{I}=[t_1,t_2]$ as $(x_{\cl{I}},u_{\cl{I}})$. If $(x_{\cl{I}},u_{\cl{I}})\in\bigcup_{\theta\in\st{K}}\cl{S}_{\cl{I}}[\rho;\theta]$, then the following holds:
    \begin{enumerate}[label=(\roman*),nosep]
        \item \label{it:coroinv1} If $s\in\cl{I}$ s.t. $\cl{G}_s(x_s,u_s)=0$ and $s<t_2$, then $\cl{G}_{s+1}(x_{s+1},u_{s+1})=0$.
        \item \label{it:coroinv2} For all $s$ in the range $t_1 + M^\pi_\rho \leq s\leq t_2 $ holds $\cl{G}_s(x_s,u_s) =0$.
        \item \label{it:coroinv3} If $\cl{G}_{t_1}(x_{t_1},u_{t_1})=0$, then $\sum_{t\in\cl{I}} \cl{G}_{t}(x_t,u_t) =0$
        \item \label{it:coroinv4} If $\cl{G}_{t_2}(x_{t_2},u_{t_2}) = 1$, then $|\cl{I}| \leq M^\pi_\rho$.
    \end{enumerate}
    \end{coro*}
    \begin{proof}
        {\ref{it:coroinv1}}:   $\cl{G}_s(x_s,u_s)=0$, implies $x_s\in\st{X}_s$. The condition $(x_{\cl{I}}, u_{\cl{I}})\in\cl{S}_\cl{I}[\rho;\theta^*]$ means that $x_{s+1}=f(s,x_s,\kappa(x_s))$, for some $f\in\mathbb{T}[\theta^*]$ and some policy $\pi[\theta_s]$, with $d(\theta_s,\theta^*)\leq \rho$. Due to the cost-invariant property, we conclude $x_{s+1}\in\st{X}_{s+1}$ and since we pick $u_{s+1}=\pi[\theta_s](s+1,x_{s+1})$ we are also guaranteed that $\cl{G}_{s+1}(x_{s+1},u_{s+1})=0$.\\
        \noindent {\ref{it:coroinv2}}: If there exists an $s$ s.t. $t_1 + M^\pi_\rho \leq s\leq t_2 $ then the interval size is bounded below as $|\cl{I}|> M^\pi_\rho$. The claim follows by using property \ref{it:coroinv1} and that $\sum_{t \in \cl{I}} \cl{G}_t(x_t,u_t) \leq M^\pi_\rho$.\\
        \noindent {\ref{it:coroinv3}}: Follows from \ref{it:coroinv1}. \ref{it:coroinv4}: by contraposition of \ref{it:coroinv2} follows $t_2<t_1+M^\pi_\rho$ and $|\cl{I}|=t_2-t_1+1\leq M^\pi_\rho$.
    \end{proof}


\subsubsection{Consistent sets}
\begin{defn}\label{app-def:Pt}
    Given a compact parameterization $(\mathbb{T},\st{K},d)$ of the uncertainty set $\cl{F}$ and a data set $\cl{D} =(d_1,\dots, d_N)$, $d_k=(t_k,x^+_k,x_k,u_k)$ define the consistent set map $\st{P}:\mathbb{D}\mapsto 2^\st{K}$ as
    \begin{align}\label{eq:app_poly_0}
        \st{P}(\cl{D}) :=\mathrm{closure}\big( \bigcap\limits_{(t,x^+,x,u) \in \cl{D}}\left\{\theta \in\st{K} \ \left| \:\exists f\in \mathbb{T}(\theta)\text{ s.t. : }\begin{array}{l} x^+ = f(t,x,u) \end{array}  \right. \right\}\big).
    \end{align}
\end{defn}
\begin{coro*}[Def. \ref{app-def:Pt}]
    Assume $\bm{\cl{D}}$ is a data stream with at least one consistent $f\in \cl{F}$
    Then, the following holds for the sequence of consistent sets $\st{P}(\bm{\cl{D}}) = (\st{P}(\cl{D}_1),\st{P}(\cl{D}_2),\dots)$:
    \begin{enumerate}[label=(\roman*),nosep]
        \item \label{it:nest} The sequence of consistent sets is nested in $\st{K}$, i.e:  $\st{P}(\cl{D}_t)\supset \st{P}(\cl{D}_{t+1})$ 
        \item \label{it:Pinf} $\st{P}(\cl{D}_\infty):=  \left(\bigcap^{\infty}_{k=1}\st{P}(\cl{D}_k)\right) \cap \st{K}$ is non-empty 
        \item \label{it:Pconv} If $\theta_t\in\st{P}(\cl{D}_t)$, and $\lim_{t \rightarrow \infty} \theta_t = \theta_\infty$, then $\theta_\infty\in\st{P}(\cl{D}_\infty)$.
    \end{enumerate}
\end{coro*}
\begin{proof}
    The first property is clear since $\cl{D}_t = \cl{D}_{t-1}\cup\{d_t\}$. For the second, notice that $\st{P}(\cl{D}_t)$ is always a non-empty compact set and recall the basic real analysis fact \cite{rudin1976principles}: The intersection of a nested sequence of non-empty compact sets is always non-empty \cite{rudin1976principles}.
     To verify property \ref{it:Pconv}, notice that nestedness implies that the sub-sequence $(\theta_T,\theta_{T+1},\dots)$ is contained in $\st{P}(\cl{D}_T)$; Since $\st{P}(\cl{D}_T)$ is closed, we have that $\theta_\infty \in \st{P}(\cl{D}_T)$. We chose $T$ arbitrary, so we conclude that $\theta_{\infty}\in\st{P}(\cl{D}_T)$ for all $T\geq 0$, i.e.: $\theta_\infty \in \st{P}(\cl{D}_\infty)$.
\end{proof}
\subsubsection{Consistent model chasing conditions}
\begin{defn}\label{def:Selprops-app}
    Let $\sel:\mathbb{D}\mapsto \st{K}$ be a selection of $\st{P}$. Let 
    $\bm{\cl{D}}=(\cl{D}_1,\cl{D}_2,\dots)$ be an online data stream and let $\bm{\theta}$ be a sequence defined for each time $t$ as $\theta_t=\sel[\cl{D}_t]$. Assume that there always exists an $f\in\cl{F}$ consistent with $\bm{\cl{D}}$ and consider the following statements:
    \begin{enumerate}[label=(\Alph*),nosep]
        \item $\theta^*=\lim_{t \rightarrow \infty}\theta_t$ exists.
        \item $\lim_{t \rightarrow \infty}d(\theta_t,\theta_{t-1})=0.$
        \item $\gamma$-\textit{competitive}:
        $ \sum^{t_2}_{t=t_1+1} d(\theta_t,\theta_{t-1}) \leq \gamma \:d_H(\st{P}(\cl{D}_{t_2}),\st{P}(\cl{D}_{t_1}))$ holds for any $t_1<t_2$.
        \item $(\gamma,T)$-\textit{weakly competitive}: $\sum^{t_2}_{t=t_1+1} d(\theta_t,\theta_{t-1}) \leq \gamma \:d_H(\st{P}(\cl{D}_{t_2}),\st{P}(\cl{D}_{t_1}))$ holds for any $t_2-t_1 \leq T$.
    \end{enumerate}
    We will say that $\sel$ is a \textit{consistent model chasing} (CMC) algorithm, if one of the above \textit{chasing} properties can be guaranteed for any pair of $\bm{\cl{D}}$ and corresponding $\bm{\theta}$. 
\end{defn}
\begin{coro}\label{coro:sel}  
    Then following implications hold between the properties of \defref{def:Selprops}:
        \begin{align*}
            \begin{array}{ccc}
                \mathrm{(C)} &\Rightarrow& \mathrm{(A)} \\
                \rotatebox[origin=c]{270}{$\Rightarrow$}& &\rotatebox[origin=c]{270}{$\Rightarrow$}\\
                \mathrm{(D)} &\Rightarrow& \mathrm{(B)} 
            \end{array}
        \end{align*}
        The reverse (and any other) implications between the properties do not hold in general.
\end{coro}
\begin{proof}
(A) $\implies$ (B), (C) $\implies$ (D) are obvious. (C) $\implies$ (A) follows by noticing that (C) implies $\sum^{\infty}_{t=1}d(\theta_t,\theta_{t-1}) \leq \gamma \dm (\st{K})$. To prove (D) $\implies$ (B), we use \lemref{lem:compactpack}. First notice that $(\gamma,T)$-weak competitiveness (w.c.) implies $(\gamma,1)$-w.c.. Pick some $\ep>0$ and let $\mathbb{I}=\{t_1,\dots,\}$ be all time-steps at which $d(\theta_t,\theta_{t-1}) \geq \ep$. Now, for each $t\in\cl{I}$ holds $d_H(\st{P}(\cl{D}_{t}),\st{P}(\cl{D}_{t-1})) \geq \tfrac{\ep}{\gamma}$. We can verify that the collection of sets $\{\st{P}(\cl{D}_{t-1})\:|\:t\in\mathbb{I}\}$ is a $\tfrac{\ep}{\gamma}$-separated set in the metric space $(2^\st{K},d_H)$, hence by \lemref{lem:compactpack} it follows that the index set $\mathbb{I}$ is finite, i.e.: $|\mathbb{I}| \leq N(\st{K}, \tfrac{\ep}{\gamma})$. Since this holds for all $\ep>0$, we proved $\forall \ep>0\exists N \text{ s.t. }\forall t \geq N:\: d(\theta_t,\theta_{t-1})<\ep$, i.e.: $\lim_{t \rightarrow \infty} d(\theta_{t},\theta_{t-1})=0$. 
\end{proof}
\begin{coro}\label{coro:comp}
$(\gamma,T)$-weakly competitiveness implies $(k\gamma,kT)$-weakly competitiveness for any $k\in\mathbb{N}$.
\end{coro}
\begin{proof}
Assume $(\gamma,T)$-w.c. and pick $k$ consecutive intervals $\cl{I}_1,\dots,\cl{I}_k$, $\cl{I}_j:=[\tau_j,\ol{\tau}_j]$, each of length $T$. Notice that $\st{P}(\cl{D}_{\ol{\tau}_k})\subset \dots \subset \st{P}(\cl{D}_{\tau_1})$ implies $d_H(\st{P}(\cl{D}_{\ol{\tau}_j}),\st{P}(\cl{D}_{\tau_j})) \leq d_H(\st{P}(\cl{D}_{\ol{\tau}_k}),\st{P}(\cl{D}_{\tau_1}))$ which leads to:
$$\sum^{k}_{j=1}\sum_{t\in\cl{I}_j} d(\theta_{t+1},\theta_{t}) \leq \gamma \sum^{k}_{j=1}\:d_H(\st{P}(\cl{D}_{\ol{\tau}_j}),\st{P}(\cl{D}_{\tau_j})) \leq \gamma k d_H(\st{P}(\cl{D}_{\ol{\tau}_k}),\st{P}(\cl{D}_{\tau_1}))$$
\end{proof}
\begin{rem}
    Notice that the above relation does \textit{not} improve the bounds of Thm.\ref{thm:redweak}\ref{it:conv3-app}.
\end{rem}
\subsubsection{Competitiveness of Steiner Point Selection}
\begin{lem}[Steiner-Point]\label{lem:steiner}
Let $s(\st{K})$ denote the steiner point of a convex body. The following inequalities hold for a nested sequence of convex bodies $\st{K}_1 \supset \st{K}_2 \dots \supset \st{K}_T$:
\begin{align*}
    &\sum^{T-1}_{t=1} \|s(\st{K}_t)-s(\st{K}_{t-1})\|_2 \leq \tfrac{n}{2}\dm(\st{K}_1), \quad &\sum^{T-1}_{t=1} \|s(\st{K}_t)-s(\st{K}_{t-1})\|_2 \leq n d_H(\st{K}_1,\st{K}_T)
\end{align*}
\end{lem}
\begin{proof}
Assume $\st{K}_1 \supset \st{K}_2 \dots \supset \st{K}_T$ is a nested sequence of convex bodies in $\R^n$, then as shown in \cite{bubeck2020chasing} Thm 2.1, it holds that
\begin{align}
\sum^{T-1}_{t=1} \|s(\st{K}_t)-s(\st{K}_{t-1})\|_2 \leq \tfrac{n}{2}(w(\st{K}_1)-w(\st{K}_T))
\end{align}
where $w(\st{S})$ denotes \textit{mean-width} of the set $\st{S}$. $w(\st{S})$ can be written as the average length of a random 1-dimensional projection of $\st{S}$: Let $v$ be distributed uniformly on the sphere $\mathbb{S}^{n-1}$ and let $w(\st{S})$ be defined as the expectation of $l_v(\st{S})$,
\begin{align}
    w(\st{S}):= \mathbb{E}_{v \sim \texttt{Unif}(\mathbb{S}^{n-1})} l_v(\st{S}),\quad l_v(\st{S}) := \dm(\mathrm{Proj}_v(\st{S}))
\end{align}
where $l_v(\st{S})$ is evaluated by first projecting the set $\st{S}$ onto the subspace spanned by the vector $v$, denoted as $\mathrm{Proj}_v(\st{S})$,  and taking its length. 
By linearity of expectation we can write 
\begin{align}
    \notag\sum^{T-1}_{t=1} \|s(\st{K}_t)-s(\st{K}_{t-1})\|_2 &\leq \tfrac{n}{2}\mathbb{E}_{v \sim \texttt{Unif}(\mathbb{S}^{n-1})} l_v(\st{K}_1)-l_v(\st{K}_T)\\
     &\leq \tfrac{n}{2}\mathbb{E}_{v \sim \texttt{Unif}(\mathbb{S}^{n-1})} 2d_{H}(\st{K}_1,\st{K}_T) = n d_H(\st{K}_1,\st{K}_T)
\end{align}
where the last inequality comes from noticing that per definition of the Hausdorff distance, it holds $\st{K}_1 \subset \st{K}_T \oplus d_H(\st{K}_1,\st{K}_T) \st{B}_{\|\anon\|_2}$. Similarly, we obtain
\begin{align}
    \notag\sum^{T-1}_{t=1} \|s(\st{K}_t)-s(\st{K}_{t-1})\|_2 &\leq \tfrac{n}{2}\mathbb{E}_{v \sim \texttt{Unif}(\mathbb{S}^{n-1})} l_v(\st{K}_1) \leq \tfrac{n}{2}\dm(\st{K}_1)
\end{align}
\end{proof}


\subsubsection{Separation of a collection of nested sets}
\begin{lem}\label{lem:compactpack}
    Let $(\st{K},d)$ be a compact metric space and let $\{\st{S}_1, \dots, \st{S}_T\}$, $\forall t:\:\st{S}_{t+1}\subset \st{S}_t$, $\st{S}_t\subset \st{K}$ be a collection of nested subsets in $\st{K}$ which are $\ep$-separated w.r.t. to the Hausdorff-metric $d_H$, i.e.: $d_H(\st{S}_{i},\st{S}_j)>\ep,\forall i\neq j$. Then it holds $T\leq N(\st{K},\ep)$.
\end{lem}
\begin{proof}
    Assume that $\{\st{S}_1, \dots,\st{S}_T\}$ is a $\ep$-separated subset of the metric space $(2^\st{K},d_H)$, where $d_H$ is the Hausdorff metric.
    Since for all $t$ holds that $d_H(\st{K}_{t},\st{K}_{t+1}) >\ep$ and $\st{K}_{t+1}\subset \st{K}_{t}$, it means that there exists at least one point $p_{t}\in \st{K}_{t}$ such that $d(p_{t},\st{K}_{t+1})>\ep$. Since for all $j>t$ holds $p_j\in \st{K}_{j}\subset \st{K}_{t+1}$, we conclude $d(p_t,p_j)\geq d(p_{t},\st{K}_{t+1})  >\ep$ for all $j>t$. This establishes that $\{p_1,\dots, p_T\}$ is a $\ep$-separated subset of $\st{K}$. Therefore, we can bound the size of the set $T$ by the packing number $T\leq N(\st{K},\ep)$.
\end{proof}
\subsection{Theorem \ref{thm:red} and \ref{thm:redweak}}
\begin{thmset*}
    Let $\partpl = (\mathbb{T},\st{K},d)$ be a compact parametrization of the uncertainty set $\cl{F}$ corresponding to the partially unknown system $x_{t+1} = f^*(t,x_t,u_t), f^* \in \cl{F}$.
    Given two procedures of the type $\pi:\st{K}\mapsto\cl{C}$ and $\sel:\mathbb{D}\mapsto \st{K}$, we apply the online control strategy $\cl{A}_{\pi}(\sel):\mathbb{D}\mapsto\cl{U}$ described in Meta-Algorithm \ref{alg:ApiselOCO} to the system. The corresponding online state and input trajectories, denoted $(\bm{x},\bm{u})$, follow the equations:
    \begin{subequations}
        \label{eq:app-cl-eq}
    \begin{align}
        x_{t+1} &= f^*(t,x_t,u_t)\\
        u_t &= \pi[\theta_t](t,x_t),\\
        \theta_t &= \sel[d_t,\dots,d_1],\text{ where }d_t = (t,x_t,x_{t-1},u_{t-1})
    \end{align}
    \end{subequations}
   Assume we are given an objective $\bm{\cl{G}}$. The next results state conditions on $\pi$ and $\sel$, which are sufficient for bounding a-priori the corresponding worst-case online cost $\sum^\infty_{t=0} \cl{G}_t(x_t,u_t)$. 
   The condition on $\pi$ depends on $\bm{\cl{G}}$ and $\cl{T}$, while the condition on $\sel$ depends only on $\cl{T}$.
    
\end{thmset*}
\begin{thm}\label{thm:red-app}
    Assume that $\sel$ chases consistent models and that $\pi$ is an oracle for an objective $\bm{\cl{G}}$. Then the following mistake guarantees hold:
    \begin{enumerate}[label=(\roman*)]
     \vspace{-0.05in}
     \item \label{it:conv1-app}If $\pi$ is robust then $(\bm{x},\bm{u})$ always satisfy: $\sum^{\infty}_{t=0}{\cl{G}}_t(x_t,u_t) < \infty.$
     \item \label{it:conv4-app} If $\pi$ is uniformly $\rho$-robust and $\sel$ is $\gamma$-competitive, then $(\bm{x},\bm{u})$ obey the inequality:
     $$ \sum^{\infty}_{t=0}{\cl{G}}_t(x_t,u_t) \leq \tpiv{\rho}\Big(\tfrac{2\gamma}{\rho} \dm(\st{K}) +1\Big).$$
    \end{enumerate}
\end{thm}

\begin{thm}\label{thm:redweak-app}
    Assume that $\sel$ chases consistent models {weakly} and that $\pi$ is an uniformly $\rho$-robust, {cost-invariant} oracle for an objective $\bm{\cl{G}}$. Then the following mistake guarantees hold:
    \begin{enumerate}[label=(\roman*)]
     \vspace{-0.05in}
     \item \label{it:conv2-app} $(\bm{x},\bm{u})$ always satisfy: $\sum^{\infty}_{t=0}{\cl{G}}_t(x_t,u_t) < \infty.$
     \item \label{it:conv3-app} If $\sel$ is $(\gamma,T)$-weakly competitive, then $(\bm{x},\bm{u})$ obey the inequality:
        $$\sum^{\infty}_{t=0} \cl{G}_t(x_t,u_t) \leq M^\pi_\rho\left(N(\st{K},r^*)+1\right),\: \quad r^*:=\frac{1}{2}\frac{\rho}{\gamma} \frac{T}{M^\pi_\rho+T}$$
    where $N$ is defined in \defref{def:entropy} as  the $r^*$-packing number of $\st{K}$.
    \end{enumerate}
\end{thm}
\subsubsection{ Part \ref{it:conv2-app} - Theorem \ref{thm:red-app} and \ref{thm:redweak-app} }

\begin{thm*}
    Assume that $\sel$ chases consistent models and that $\pi$ is an oracle for $\bm{\cl{G}}$. Then $(\bm{x},\bm{u})$ always satisfy: $\sum^{\infty}_{t=0}{\cl{G}}_t(x_t,u_t) < \infty.$ 
\end{thm*}
\begin{proof}
Denote the online data at time $t$ as the tuple $\cl{D}_t := (d_1,\dots,d_t)$. 
Recall \corref{app-def:Pt}\ref{it:Pinf} to verify that $\st{P}({\cl{D}_\infty})$ is non-empty. Since $\theta_{\infty}\in \st{P}_{\infty}$, 
there exists some $f'\in \mathbb{T}(\theta_{\infty})$ such that the trajectories satisfy for all time $t\geq 0$ the dynamics $x_{t+1} = f'(t,x_t,u_t)$.
Since $\theta_t \rightarrow \theta_{\infty}$, there exists a time $T$, such that for all $t \geq T$, $d(\theta_t,\theta_{\infty}) < \rho$, i.e.: for $t\geq T$, we apply policies $\pi(t;\theta_{t})$ with parameters $\rho$-close to $\theta_{\infty}$. 
Per definition, this tells us that for the time interval $\cl{I}=[T,\infty)$, the tail of the trajectory $x_{\cl{I}},u_{\cl{I}}$ is contained in $\cl{S}_{\cl{I}}[\rho;\theta_\infty]$. 
Since we assume that $\pi$ is a $\rho$-robust oracle for $\bm{\cl{G}}$, (\defref{app-def:pi-rhorobust}), we have to conclude that $\sum^{\infty}_{t=T}\cl{G}_t(x_t,u_t)\leq M$ for some finite number $M$. This proves the desired claim since $\sum^{\infty}_{t=0}\cl{G}_t(x_t,u_t) \leq \sum^{T-1}_{t=0}\cl{G}_t(x_t,u_t) + M$. 
\end{proof}
    
\begin{thm*}
    Assume that $\sel$ chases consistent models {weakly} and that $\pi$ is an uniformly $\rho$-robust, {cost-invariant} oracle for $\bm{\cl{G}}$. Then, $(\bm{x},\bm{u})$ always satisfy: $\sum^{\infty}_{t=0}{\cl{G}}_t(x_t,u_t) < \infty.$
\end{thm*}
\begin{proof}
Pick an arbitrary $0 < \ep < \rho$ and some $T \geq M^{\pi}_\rho$. There exists $N>0$ such that $\forall t \geq N$: $\dist(\st{P}(\cl{D}_t),\theta_t) \leq \ep/4$ and $d(\theta_t,\theta_{t-1}) \leq \ep/ (2T)$.
Pick an arbitrary time-step $s>N+T$, there exists a $\theta'_{s} \in \st{P}(\cl{D}_{s})$ such that $d(\theta'_{s},\theta_{s}) \leq \ep/4$. Consider now the timesteps $t$ in the time-window $\cl{I}_{s} = [s-T,s-1]$. Per assumption and triangle inequality, we have for all $t \in \cl{I}_{s}$:
\begin{align}
    \notag d(\theta'_{s},\theta_t) &\leq d(\theta'_{s},\theta_{s})+\sum\limits^{s-1}_{j=t} d(\theta_{j+1},\theta_j) \leq \frac{\ep}{4} + (s-t)\frac{\ep}{2T} \leq \frac{\ep}{4} + T\frac{\ep}{2T} \leq \ep < \rho.
\end{align}
Since $\theta'_{s}\in \st{P}(\cl{D}_{\infty})$, the truncation ($\bm{x}_{\cl{I}_{s}}$, $\bm{u}_{\cl{I}_{s}}$) is contained in the set $\cl{S}_{\cl{I}_s}[\rho;\theta'_s]$. Since we picked $T$ to be larger than the mistake constant $M^\pi_\rho$, we can use \corref{app-def:pi-invariant} to conclude that the cost at time $s$ has to be zero, i.e.: $\cl{G}_{s}(x_{s},u_{s}) = 0$. Finally, since $s$ was chosen arbitrary in the interval $[N+T+1,\infty)$, we know that $\sum^{\infty}_{t=N+T+1}\cl{G}_t(x_t,u_t)=0$. The total cost is therefore $\sum^{\infty}_{t=0}{\cl{G}}_t(x_t,u_t)=\sum^{N+T}_{t=0}{\cl{G}}_t(x_t,u_t)$ and is finite.

\end{proof}
\subsubsection{ Part \ref{it:conv4-app} - Theorem \ref{thm:red-app} and \ref{thm:redweak-app} }

\begin{thm*}
    Assume procedure $\sel$ is $(\gamma,T)$-weakly competitive for some $\gamma>0$, $T\geq 1$ and that procedure $\pi$ is a uniform $\rho$-robust, cost-invariant oracle for $\bm{\cl{G}}$. 
    Then, the total number of mistakes is guaranteed to be bounded above by:
    \begin{align*}
        \sum^{\infty}_{t=0} \cl{G}_t(x_t,u_t) \leq M^\pi_\rho\left(N(\st{K}^{\circ},r^*)+1\right)\leq M^\pi_\rho\left(N(\st{K},r^*)+1\right),\quad\text{with  } r^*:=\frac{1}{2}\frac{\rho}{\gamma} \frac{T}{M^\pi_\rho+T}
    \end{align*}
    where $\st{K}^\circ:=\st{K}\setminus \mathrm{int}(\st{P}(\cl{D}_\infty))$.
\end{thm*}
\begin{proof} Denote $\bm{x},\bm{u}$ to be some fixed online trajectories and denote $\bm{\theta}$ as the corresponding parameter sequence selected by procedure $\sel$.
    The sequence $\bm{\theta}$ satisfies $\theta_t \in\st{P}(\cl{D}_t)$ and $\sum^{t_2}_{t=t_1+1}d(\theta_t,\theta_{t-1}) \leq \gamma d(\st{P}(\cl{D}_{t_2}),\st{P}(\cl{D}_{t_1}))$ for all $t_2-t_1\leq T$. For some time-step $\tau>0$, we derive bounds on the mistakes $\sum^{\tau}_{t=0}\cl{G}_t(x_t,u_t)$. Set $t_0 = 0$ and construct the index-sequence $t_0,t_1,t_2,\dots,t_N$ as follows:
    \begin{align}
        t_k&:= \left\{\begin{array}{ll}\min\left\{t\leq \tau\:\left|\: t > t_{k-1} \text{ and }d(\theta_{t},\theta_{t_{k-1}})>\tfrac{1}{2}\rho \right.\right\} &\text{, if }k>1\\
            0 & \text{, if }k=0  \end{array}\right. 
    \end{align}
    until for some $N$, the condition $t\leq \tau, t > t_{N}, d(\theta_{t},\theta_{t_{N}})>\tfrac{1}{2}\rho$ becomes infeasible and we terminate the construction.
    Define the intervals $\cl{I}_k:=[t_k,\ol{t}_k]$, where $\ol{t}_k:=t_{k+1}-1$ for $k<N$ and $\ol{t}_N = \tau$. 
     The intervals $\cl{I}_0,\dots,\cl{I}_N$ are a non-overlapping cover of the time-interval $[0,\tau]$: $$\bigcup\limits_{0\leq k\leq N}\cl{I}_k = [0,\tau],\quad \cl{I}_k\cap\cl{I}_{k-1}=\emptyset,\forall k: 1\leq k \leq N.$$
    Let $(a_0,\dots,a_N)$ and $(b_0,\dots,b_N)$  be the parameters selected at the start and end of each interval $\cl{I}_k$, respectively: $a_k := \theta_{t_k}$ and $b_k := \theta_{\ol{t}_k}$. Per construction, we know that 
    \begin{align}\label{eq:centertopar-2}
        d(a_k,\theta_{t})&\leq \tfrac{1}{2}\rho\text{ for all }t \in \cl{I}_k\\
        \label{eq:centertocenter-2} d(a_k,a_{k-1})&> \tfrac{1}{2}\rho\text{ for all }1 \leq k\leq N
    \end{align}
    Inequality \eqref{eq:centertopar} states that $d(a_k,b_{k}) \leq \tfrac{1}{2}\rho$ and implies via triangle inequality that for all $ t \in \cl{I}_k$ holds $$d(\theta_t,b_k)\leq d(\theta_t,a_k)+ d(a_k,b_k)\leq \rho.$$

    \noindent Since we picked $b_k=\theta_{\ol{t}_k}$ and the procedure $\sel$ assures $\theta_{\ol{t}_k}\in\st{P}(\cl{D}_{\ol{t}_k})$, it means that for some $f_k\in\mathbb{T}[b_k]$, the partial trajectory $(x_{\cl{I}_k},u_{\cl{I}_k})$ satisfies the following equations for the time-steps $t\in\cl{I}_k$:
    \begin{align}
     x_{t+1} = f'(t,x_t,u_t),\quad u_{t} = \pi[\theta_t](t,x_t)
    \end{align}
 We can therefore conclude that $(x_{\cl{I}_k},u_{\cl{I}_k})\in\cl{S}_{\cl{I}_k}[\rho;b_k]$. We apply \corref{app-def:pi-invariant} to conclude that \begin{align}\label{eq:firstmbound}\sum_{t\in\cl{I}_k} \cl{G}_t(x_t,u_t)\leq M^\pi_\rho
\end{align}
 for each $k\in\{0,\dots,N\}$. Now, define $\mathbb{S}$ as the following collection of intervals
 \begin{align}\label{eq:Scollection}
     \mathbb{S}:= \{ \cl{I}_k\:|\: \cl{G}_{t_{k+1}}(x_{t_{k+1}},u_{t_{k+1}})=1 \},
 \end{align}
i.e.: all intervals $\cl{I}_k$ where at the start of the \textit{next} interval $\cl{I}_{k+1}$ the cost is $1$. Combining this with the former bound \eqref{eq:firstmbound}, we can decompose the total mistake sum as
\begin{align}
    \notag \sum^{T}_{t=0} \cl{G}_t(x_t,u_t) &= \sum_{t\in\cl{I}_{0}} \cl{G}_t(x_t,u_t)+\sum_{\cl{I}_j \in \mathbb{S}}\sum_{t\in\cl{I}_{j+1}} \cl{G}_t(x_t,u_t) + \underbrace{\sum_{\cl{I}_j \notin \mathbb{S}}\sum_{t\in\cl{I}_{j+1}} \cl{G}_t(x_t,u_t)}_{0} \\
    \label{eq:costsum-2}&=\sum_{t\in\cl{I}_{0}} \cl{G}_t(x_t,u_t)+\sum_{\cl{I}_j \in \mathbb{S}}\sum_{t\in\cl{I}_{j+1}} \cl{G}_t(x_t,u_t) \leq M^\pi_\rho \left(|\mathbb{S}| + 1\right)
\end{align}
Notice that the last term $\sum_{\cl{I}_j \notin \mathbb{S}}\sum_{t\in\cl{I}_{j+1}} \cl{G}_t(x_t,u_t)$ in the first equation is zero because $\cl{I}_j\notin \mathbb{S}$ implies that the next interval $\cl{I}_{j+1}$ start with zero cost; due to the cost-invariance property \corref{app-def:pi-invariant} it follows that
$ \sum_{t\in\cl{I}_{j+1}} \cl{G}_t(x_t,u_t) = 0$. 
The remainder of the proof is concerned with bounding the cardinality of the collection $\mathbb{S}$.\\

\noindent \textbf{Bounding $|\mathbb{S}|$}: We know that for each $l$ in the range $1\leq l\leq |\cl{I}_k|$, there exists at least one sub-interval $\cl{I}_l'\subset \cl{I}_k,\:|\cl{I}_l'|=l$ of length $l$, such that
    \begin{align}\label{eq:subseqtau2}
        \sum_{t\in\cl{I}'} d(\theta_{t},\theta_{t+1}) > \tfrac{1}{2}\rho \tfrac{l}{(|\cl{I}_k|+l)}.
    \end{align}
    The above has to be true, since otherwise we would contradict \eqref{eq:centertocenter-2}:
    \begin{itemize}
        \item  Let $\cl{I}'_1,\dots, \cl{I}'_m$, $m=\lceil{|\cl{I}_k|}/{l}\rceil, |\cl{I}'_i|=l, \cl{I}'_i \subset \cl{I}_k$ be an overlapping cover of $\cl{I}_k$, then 
    \begin{align}
        d(a_k,a_{k+1}) &\leq \sum_{t\in\cl{I}_k}d(\theta_{t+1},\theta_{t})\leq \sum^{m}_{j=1}\sum_{t\in\cl{I}'_j}d(\theta_{t+1},\theta_{t})\leq \tfrac{1}{2}\rho\left\lceil\tfrac{|\cl{I}_k|}{l}\right\rceil \tfrac{l}{|\cl{I}_k|+l} \\
        & \leq \tfrac{1}{2}\rho\left(\tfrac{|\cl{I}_k|}{l}+1\right) \tfrac{l}{|\cl{I}_k|+l}=\tfrac{1}{2}\rho,\quad\quad\text{ (recall that }a_{k+1} = \theta_{\ol{t}_k+1}\text{)}
    \end{align}
which is a contradiction to \eqref{eq:centertocenter-2}
\end{itemize}
Hence, we can always pick a sequence of sub-intervals $[\tau^l_k,\ol{\tau}^l_k] = \cl{I}^{(l)}_k \subset \cl{I}_k$ (either of length $l$ or identical to $\cl{I}_k$ if $|\cl{I}_k|\leq l$) such that 
\begin{align}
    \sum_{t\in\cl{I}^{(l)}_k}d(\theta_{t}, \theta_{t+1}) > \tfrac{1}{2}\rho \tfrac{l}{|\cl{I}_k| + l}
\end{align}
Notice that if $|\cl{I}_k|\leq l$, we pick $\cl{I}^{(l)}_k=\cl{I}_k$, and therefore the above inequality is vacuously true since, $\sum_{t\in\cl{I}_k}d(\theta_{t}, \theta_{t+1}) \geq d(a_k,a_{k+1}) > \tfrac{1}{2}\rho \geq \tfrac{1}{2}\rho \tfrac{l}{|\cl{I}_k|+l}$.

Now, the $(\gamma,T)$-weak competitiveness property ensures that 
for all $k$ and all $t \geq \ol{\tau}_k+1$ holds: 
\begin{align}
    \notag &&\tfrac{1}{2}\rho \tfrac{T}{|\cl{I}_k| +T} &< \sum_{t\in\cl{I}^{(T)}_k}d(\theta_{t}, \theta_{t+1}) \leq \gamma d_H(\st{P}(\cl{D}_{{\tau}_k}),\st{P}(\cl{D}_{\ol{\tau}_k+1})) \leq \gamma d_H(\st{P}(\cl{D}_{{\tau}_k}),\st{P}(\cl{D}_{t}))\\
    &\Leftrightarrow& d_H(\st{P}(\cl{D}_{{\tau}_k}),\st{P}(\cl{D}_{t})) &> \tfrac{1}{2}\tfrac{\rho}{\gamma} \tfrac{T}{|\cl{I}_k| +T}. 
\end{align}
where $d_H(\st{P}(\cl{D}_{{\tau}_k}),\st{P}(\cl{D}_{\ol{\tau}_k+1})) \leq d_H(\st{P}(\cl{D}_{{\tau}_k}),\st{P}(\cl{D}_{t}))$ follows from nestedness. 
From now on, we use the abbreviation $\st{P}_t$ to refer to the sets $\st{P}(\cl{D}_t)$. 

 \noindent Recall the definition $\mathbb{S}:= \{ \cl{I}_k\:|\: \cl{G}_{t_{k+1}}(x_{t_{k+1}},u_{t_{k+1}})=1 \}$ and let $k_j$ denote the $j$-th interval that belongs to $\mathbb{S}$, i.e.: $\cl{I}_{k_j} \subset \mathbb{S}$. Now set $l = T$ and define $\st{S}_j$ as a subsequence of $\st{P}_1,\st{P}_2,\dots$ as follows:
 \begin{align}
    \st{S}_{j}:=\left\{ \begin{array}{cl}\st{P}_{\ol{t}_{k_j}} &\text{if }x_{\ol{t}_{k_j}} \in \st{X}_{\ol{t}_{k_j}}\\
        \st{P}_{\tau^{T}_{k_j}} &\text{if }x_{\ol{t}_{k_j}} \notin \st{X}_{\ol{t}_{k_j}} \end{array}\right.
\end{align}
We will show that this collection $P=\{\st{S}_1,\st{S}_2,\dots\}$ of sets $\st{S}_j$ is a $\tfrac{1}{2}\tfrac{\rho}{\gamma} \tfrac{T}{M^\pi_\rho+T}$ - separated set in the metric space $(2^\st{K},d_H)$ via the following inequality:
$$\forall j<i: d_H(\st{S}_j,\st{S}_i)> \tfrac{1}{2}\tfrac{\rho}{\gamma} \tfrac{T}{M^\pi_\rho+T}.$$
This is proven below:
\begin{itemize}[nosep]
 \item Recall $\sel$ is defined to always pick $\theta_t\in\st{P}(\cl{D}_t)$ and $\pi$ is $\rho$-uniformly robust and cost-invariant. Due to the $\sel$ property, there always exists a function $f'\in\mathbb{T}[\theta_{t_{k+1}}]$ such that $x_{t_{k+1}} = f'(\ol{t}_k,x_{\ol{t}_k},\pi[\theta_{\ol{t}_k}](\ol{t}_k,x_{\ol{t}_k}))$. On the other hand, because of the $\pi$ property, the statement $x_{t_{k+1}}\notin \st{X}_{t_{k+1}}$ implies that one of the following two has to hold at time $\ol{t}_k$:\\
 \noindent 1. Assume $x_{\ol{t}_k}\in \st{X}_{\ol{t}_k}$, then it has to hold that $d(\theta_{\ol{t}_k},\theta_{t_{k+1}})>\rho$. Notice due to $(\gamma,H)$-w.c. property, that $d(\theta_{\ol{t}_k},\theta_{t_{k+1}})\leq \gamma d_H(\st{P}_{\ol{t}_k},\st{P}_{t_{k+1}})$ which gives us $$d_H(\st{P}_{\ol{t}_k},\st{P}_{t_{k+1}})> \tfrac{1}{2}\tfrac{\rho}{\gamma} $$ 2. Assume $x_{\ol{t}_k} \notin \st{X}_{\ol{t}_k}$, then via \corref{app-def:pi-invariant}, it follows that $|\cl{I}_k|\leq M^\pi_\rho$, which then implies that $d_H(\st{P}_{\tau^T_k},\st{P}_{t_{k+1}}) > \tfrac{1}{2}\frac{\rho}{\gamma} \tfrac{T}{|\cl{I}_k|+T}\geq \tfrac{1}{2}\tfrac{\rho}{\gamma} \tfrac{T}{M^\pi_\rho+T}$. 
 \item Taking the minimum of both cases we can see that $ d_H(\st{S}_j,\st{P}_{t_{k_j+1}})> \tfrac{1}{2}\tfrac{\rho}{\gamma} \tfrac{T}{M^\pi_\rho+T}.$
 Due to nestedness, it holds for $i>j$ that $\st{S}_i\subset \st{P}_{t_{k_j+1}} \subset \st{S}_j$. Thus, it holds $d_H(\st{S}_j,\st{S}_i)\geq d_H(\st{S}_j,\st{P}_{t_{k_j+1}})$ and we arrive at the separation condition:
 $$\forall j<i:\: d_H(\st{S}_j,\st{S}_i)> \tfrac{1}{2}\tfrac{\rho}{\gamma} \tfrac{T}{M^\pi_\rho+T}.$$
\end{itemize}
We conclude from \lemref{lem:compactpack}, that $|\mathbb{S}| = |P|$ is bounded by the packing number $N(\st{K},\tfrac{1}{2}\tfrac{\rho}{\gamma} \tfrac{T}{M^\pi_\rho+T})$. Substituting into the bound \eqref{eq:costsum-2} and taking the limit $T\rightarrow \infty$, we get the total number of mistakes as:
\begin{align*}
    \sum^{\infty}_{t=0} \cl{G}_t(x_t,u_t) \leq M^\pi_\rho\left(N(\st{K},\tfrac{1}{2}\tfrac{\rho}{\gamma} \tfrac{T}{M^\pi_\rho+T})+1\right)
\end{align*}
\textbf{A tighter bound. }
We can define $\st{K}^\circ=\st{K} \setminus \mathrm{int}(\st{P}(\cl{D}_\infty))$ and $\st{S}^\circ_j = \st{S}_j \setminus \mathrm{int}(\st{P}(\cl{D}_\infty))$ and notice that $\st{S}^\circ_j$ is non-empty for all $j$: $\st{S}^\circ_j$ and $\st{P}(\cl{D}_\infty)$ are closed, so $\st{S}_j\subset \st{P}(\cl{D}_\infty)$ implies that $\st{S}^\circ_j $ contains at least the boundary of $\st{P}(\cl{D}_\infty)$. Moreover, we can verify that the corresponding collection $P^\circ=\{\st{S}^\circ_1,\st{S}^\circ_2,\dots\}$ of sets $\st{S}^o_j$ is still a $\tfrac{1}{2}\tfrac{\rho}{\gamma} \tfrac{H}{M^\pi_\rho+H}$-separated set in the compact metric space $(2^{\st{K}^\circ},d_H)$. Therefore we can improve the previous mistake guarantee and state the tighter inequality:
 \begin{align*}
    \sum^{\infty}_{t=0} \cl{G}_t(x_t,u_t) \leq M^\pi_\rho\left(N(\st{K}\setminus \mathrm{int}(\st{P}(\cl{D}_\infty)),\tfrac{1}{2}\tfrac{\rho}{\gamma} \tfrac{H}{M^\pi_\rho+H})+1\right)
\end{align*}

\end{proof}

\begin{thm*}
    Assume that procedure $\sel$ chases consistent models in $\partpl$ and is $\gamma$-competitive and that procedure $\pi$ is a uniformly $\rho$-robust, cost-invariant oracle for $\bm{\cl{G}}$. 
    Then, the total number of mistakes is guaranteed to be bounded above by:
    \begin{align*}
        \sum^{\infty}_{t=0}{\cl{G}}_t(x_t,u_t)  \leq M^\pi_\rho(2\tfrac{\gamma}{\rho}d_H(\st{K},\st{P}(\cl{D}_\infty))+1) \leq M^\pi_\rho(2\tfrac{\gamma}{\rho}\dm(\st{K})+1) 
    \end{align*}
\end{thm*}

\begin{proof}
    The parameter sequence $\bm{\theta}$ provided by $\sel$ satisfies $\theta_t \in\st{P}(\cl{D}_t)$, $\forall t$ and $\sum^{T}_{t=1}d(\theta_t,\theta_{t-1})\leq \gamma d_H(\st{K}, \st{P}(\cl{D}_T))$.  Set $t_0 = 0$ and construct the index-sequence $t_0,t_1,t_2,\dots,t_N$ as follows:
    \begin{align}
        t_k&:= \left\{\begin{array}{ll}\min\left\{t\leq T\:\left|\: t > t_{k-1} \text{ and }d(\theta_{t},\theta_{t_{k-1}})>\tfrac{1}{2}\rho \right.\right\} &\text{, if }k>1\\
            0 & \text{, if }k=0  \end{array}\right. 
    \end{align}
    until for some $N$, the condition $t\leq T, t > t_{N}, d(\theta_{t},\theta_{t_{N}})>\tfrac{1}{2}\rho$ becomes infeasible and we terminate the construction.
    Define the intervals $\cl{I}_k:=[t_k,\ol{t}_k]$, where $\ol{t}_k:=t_{k+1}-1$ for $k<N$ and $\ol{t}_N = T$. 
     The intervals $\cl{I}_0,\dots,\cl{I}_N$ are a non-overlapping cover of the time-interval $[0,T]$: $$\bigcup\limits_{0\leq k\leq N}\cl{I}_k = [0,T],\quad \cl{I}_k\cap\cl{I}_{k-1}=\emptyset,\forall k: 1\leq k \leq N.$$
    Let $(a_0,a_1,\dots, a_N)$ and $(b_0,b_1,\dots, b_N)$  be the parameters selected at the start and end of each interval $\cl{I}_k$, respectively: $a_k := \theta_{t_k}$ and $b_k := \theta_{\ol{t}_k}$. Per construction, we know that 
    \begin{align}\label{eq:centertopar}
        d(a_k,\theta_{t})&\leq \tfrac{1}{2}\rho\text{ for all }t \in \cl{I}_k\\
        \label{eq:centertocenter} d(a_k,a_{k-1})&> \tfrac{1}{2}\rho\text{ for all }\: 1 \leq k \leq N.
    \end{align}
    Inequality \eqref{eq:centertopar} states that $d(a_k,b_{k}) \leq \tfrac{1}{2}\rho$ and implies via triangle inequality that for all $ t \in \cl{I}_k$ holds $$d(\theta_t,b_k)\leq d(\theta_t,a_k)+d(a_k,b_k) \leq \rho.$$
    
    \noindent Since we picked $b_k=\theta_{\ol{t}_k}$ and the procedure $\sel$ assures $\theta_{\ol{t}_k}\in\st{P}(\cl{D}_{\ol{t}_k})$, it means that for some $f'\in\mathbb{T}[b_k]$, the partial trajectory $(x_{\cl{I}_k},u_{\cl{I}_k})$ satisfies the following equations for the time-steps $t\in\cl{I}_k$:
    \begin{align}
     x_{t+1} = f'(t,x_t,u_t),\quad u_{t} = \pi[\theta_t](t,x_t).
    \end{align}
 We can therefore conclude that $(x_{\cl{I}_k},u_{\cl{I}_k})\in\cl{S}_{\cl{I}_k}[\rho;b_k]$. Now, \corref{app-def:pi-invariant} applies and we conclude that $\sum_{t\in\cl{I}_k} \cl{G}_t(x_t,u_t)\leq M^\pi_\rho$. Therefore we can bound the total mistakes as:
\begin{align}\label{eq:costsum}
    \sum^{T}_{t=0} \cl{G}_t(x_t,u_t) = \sum^{N}_{k=0} \sum_{j\in\cl{I}_k} \cl{G}_j(x_j,u_j) \leq  M^\pi_\rho(N+1) .
\end{align}
We can now bound $N$ using the $\gamma$-competitiveness chasing property. Recalling \eqref{eq:centertocenter}, we obtain the chain of inequalities
$$ \tfrac{1}{2}\rho N\leq \sum^{N}_{k=1} d(a_{k},a_{k-1})\leq \sum^{N-1}_{k=0}\sum_{t\in\cl{I}_k} d(\theta_{t},\theta_{t+1}) \leq \sum^{T}_{t=1}d(\theta_t,\theta_{t-1}) \leq \gamma d_H(\st{K},\st{P}(\cl{D}_T)).$$
which lead to the bound $N\leq 2\tfrac{\gamma}{\rho}d_H(\st{K},\st{P}(\cl{D}_T))$. We substitute this into \eqref{eq:costsum} to obtain the desired bound on the total number of mistakes:
\begin{align}
    \notag \sum^{T}_{t=0} \cl{G}_t(x_t,u_t) &\leq M^\pi_\rho(2\tfrac{\gamma}{\rho}d_H(\st{K},\st{P}(\cl{D}_T))+1)\\
    &\leq M^\pi_\rho(2\tfrac{\gamma}{\rho}d_H(\st{K},\st{P}(\cl{D}_\infty))+1) \leq M^\pi_\rho(2\tfrac{\gamma}{\rho}\dm(\st{K})+1)
\end{align}
We can take the limit $T\rightarrow \infty$ and arrive at the desired result.
\end{proof}
\subsection{Worst-case bounds on the state norm}

\begin{thm}\label{thm:transient-app}
    Assume that for procedure $\pi:\st{K}\mapsto\cl{C}$ there are constants $\alpha,\beta>0$, such that $\forall \theta,\theta'\in\st{K}, x\in\cl{X},f \in \mathbb{T}[\theta]:\:\|f(t,x,\pi[\theta'](t,x))\| \leq \alpha d(\theta,\theta')\|x\| + \beta.$ The following state bound guarantees hold: 
    \begin{enumerate}[label=(\roman*)]
        \item If $\sel$ is a $\gamma$-competitive CMC algorithm, then:
         \begin{align*}
            \forall t:\quad \|x_t\| \leq e^{\alpha \gamma \phi(\st{K})}\left(e^{-t} \|x_0\| + \beta \frac{e}{e -1}\right).
        \end{align*}
        \item If $\sel$ is a $(\gamma,T)$-weakly competitive CMC algorithm, then:
        \begin{align*}
            \|\bm{x}\|_{\infty} \leq \inf\limits_{0<\mu<1} \left(1+(\alpha\phi(K))^{n^*}\right)\max\{\tfrac{\beta}{1-\mu}, \|x_0\|\} + \beta\sum\limits^{n^*}_{k=0} (\alpha\phi(K))^k
        \end{align*}
        where $n^*= N(\st{K},\tfrac{\mu}{\alpha\gamma})$ and $\phi(\st{K})$ denotes the diameter of $\st{K}$. 
    \end{enumerate}
\end{thm}
\begin{proof}
    Part 1: In each time-step $t$, it holds  $x_{t+1} = f(t,x_t,\pi[t,\theta_t])$ for some $f \in \mathbb{T}(\theta_{t+1})$. Therefore, the following inequality holds at each time-step:
    \begin{align}
    \|x_{t+1}\| = \|f(t,x_t,\pi[\theta_t](t,x_t))\| \leq \alpha d(\theta_{t+1},\theta_t)\|x_t\| + \beta
    \end{align}
    We apply \lemref{lem:ssup} with the substitution $s_t:=\|x_t\|$, $\delta_t := \alpha d(\theta_{t+1},\theta_t)$ and $c:=\beta$, to obtain
    \begin{align}
        \|x_t\| \leq e^{\alpha L}\left(e^{-t} \|x_0\| + \beta \frac{e}{e -1}\right)
    \end{align}
    \noindent Part 2: We follow the proof technique used in the main result of \cite{ho2019robust} to prove boundedness. Given a closed-loop trajectory $\bm{x}$, at each time-step $t$ holds $x_{t+1} = f(t,x_t,\pi[\theta_t](t,x_t))$ for some $f \in \mathbb{T}(\theta_{t+1})$. Take an arbitrary time step $T$. Define $\mathbb{I}$ the set of all indeces $k<T$ for which holds $\|x_{k+1}\| > \mu \|x_k\| + \beta$. Notice that for each $k\notin\mathbb{I}$ holds $\|x_{k+1}\| \leq \mu \|x_k\| + \beta$ while for $k\in\mathbb{I}$ we have at least the inequality $\|x_{k+1}\|\leq \alpha \mathrm{diam}(\st{K})\|x_k\| +\beta$. Now, for each $k\in\mathbb{I}$ holds
    $$\mu \|x_k\| + \beta <\|x_{k+1}\|=\|f(k,x_k,\pi[\theta_k](k,x_k))\| \leq \alpha d(\theta_{k+1},\theta_k)\|x_k\|+\beta$$ 
    which leads to $d(\theta_{k+1},\theta_k)>\tfrac{\mu}{\alpha}$. Now, the $(\gamma,T)$ weak competitive property ensures that 
    $d(\theta_{k} ,\theta_{k+1}) \leq \gamma d_H(\st{P}(\cl{D}_k),\st{P}(\cl{D}_{k+1}))$. We can therefore conclude that:
    \begin{align}
        \tfrac{\mu}{\alpha\gamma} <\tfrac{1}{\gamma}d(\theta_{k},\theta_{k+1}) \leq d_H(\st{P}(\cl{D}_k),\st{P}(\cl{D}_{j})),\quad \text{for }j>k
    \end{align}
    Hence, $\{\st{P}(\cl{D}_k)\:|\: k\in\mathbb{I}\}$ is a $\tfrac{\mu}{\alpha\gamma}$-separated set in $\st{K}$. Therefore $|\mathbb{I}|\leq N(\st{K},\tfrac{\mu}{\alpha\gamma})$. Recall again that for each $k\notin\mathbb{I}$ holds $\|x_{k+1}\| \leq \mu \|x_k\| + \beta$ while for $k\in\mathbb{I}$ it holds $\|x_{k+1}\|\leq \alpha \dm(\st{K})\|x_k\| +\beta$. Following the same arguments as in the appendix of \cite{ho2019robust}, we obtain the presented $\|\anon\|_{\infty}$-bound on $\bm{x}$.
    \end{proof}
    \begin{lem}\label{lem:ssup}
    Let $\bm{s} = (s_0, s_1, \dots)$, $\bm{\delta} = (\delta_0, \delta_1, \dots)$ be non-negative scalar sequences such that $s_{k+1} \leq \delta_k s_k + c$, with $c\geq 0$ and $\sum^{\infty}_{t=0}\delta_t \leq L $. Then $s_t$ is bounded by:
    $$s_t \leq e^{L}\left(e^{-t}s_0+
    c\tfrac{e}{e-1}\right).$$
    \end{lem}
    \begin{proof}
    Using comparison lemma and \lemref{lem:aux1} we can bound $s_t$ as 
    \begin{align}s_t \leq \prod^{t-1}_{k=0}\delta_k s_0 + c \left(1+ \sum^{t-1}_{j=1} \prod^{t-1}_{k=j} \delta_k\right)  &\leq e^{-t}e^Ls_0 + c\left(1+\sum^{t-1}_{j=1} e^{-j}\right)e^L \\
        &\leq e^{-t}e^Ls_0 + ce^L\sum^{\infty}_{j=0} e^{-j} \leq e^{-t}e^Ls_0+
        c\frac{e^{L+1}}{e-1}
    \end{align}
    \end{proof}
    \begin{lem}\label{lem:aux1}
    Let $\bm{\delta} = (\delta_0,\delta_1,\dots)$ be a non-negative scalar sequence such that $\sum^{\infty}_{t=0} \delta_t \leq L$, then $\prod^{t-1}_{j=0} \delta_j \leq e^{-t}e^L$
    \end{lem}
    \begin{proof}
    Recall the basic fact $1+x \leq e^x$. Then 
    \begin{align*}
        \prod^{t-1}_{j=0} \delta_j = \prod^{t-1}_{j=0} (1+ (\delta_j-1)) \leq \prod^{t-1}_{j=0} \exp(\delta_j-1) = \exp\left(\sum^{t-1}_{j=0}(\delta_j-1)\right) = \exp(L-t) =e^{-t}e^{L}
    \end{align*}
\end{proof}
\subsubsection{Mistake guarantees with locally robust oracles}
\begin{thm*}[Corollary of Thm. \ref{thm:red}]
Consider the setting and assumptions of Thm.\ref{thm:red} and Thm.\ref{thm:redweak}, but relax the oracle robustness requirements to corresponding local versions and enforce the additional oracle assumption stated in \thmref{thm:transient}.

\noindent Then all guarantees of \thmref{thm:red} \ref{it:conv1},\ref{it:conv4} and \thmref{thm:redweak} \ref{it:conv2},\ref{it:conv3} still hold, if we replace $M^\pi_\rho$ in \thmref{thm:red} \ref{it:conv4} and \thmref{thm:redweak} \ref{it:conv3} respectively by $M^\pi_\rho(\gamma_\infty)$ and $M^\pi_\rho(\gamma^w_\infty)$ with the constants:
\begin{align*}
    \gamma_\infty&=e^{\alpha \gamma \phi (\st{K})}\left(\|x_0\| + \beta \frac{e}{e -1}\right)\\
    \gamma^w_\infty &= \inf\limits_{0<\mu<1} \left(1+(\alpha\phi(K))^{n^*}\right)\max\{\tfrac{\beta}{1-\mu}, \|x_0\|\} + \beta\sum\limits^{n^*}_{k=0} (\alpha\phi(K))^k 
\end{align*}
where $n^*= N(\st{K},\tfrac{\mu}{\alpha\gamma})$ and $\phi(\st{K})$ denotes the diameter of $\st{K}$. 
\end{thm*}
\begin{proof}
    The results are obtained by combining \thmref{thm:transient} with Corollary of Def.\ref{app-def:pi-rhorobust} and repeating the proofs of each part of \thmref{thm:red}, however this time, replacing $M^\pi_\rho$ by $M^\pi_\rho(\gamma_\infty)$ and substituting $\gamma_\infty$ with the corresponding bounds of \thmref{thm:transient}. 
\end{proof}

\section{Additional Discussion on Related Work}
\label{sec:app_related}

\emph{Online learning of optimal control for linear systems.} Many recent learning approaches for control of dynamical systems have focused on the setting of linear optimal control: One is given a linear system and the control objective is to minimize a specified cost functional. To relate our problem setting to other approaches in this field, we can view our problem setting as an instance of optimal control where we restrict the cost-function to be $\{0,1\}$-valued.

Recent learning and control approaches have focused on the case of Linear Quadratic Regulator (LQR) \cite{fiechter1997pac, abbasi2011regret, dean2017sample, dean2018regret, cohen2018online}, or linear dynamical system with convex costs \cite{agarwal2019online, agarwal2019logarithmic, hazan2019nonstochastic}. Our work is instead suitable as well for the nonlinear control setting. In addition, even when restricted to the linear system setting, recent line of work on online learning for control differ from our approach in the following aspects:

\begin{itemize}[noitemsep]

    \item Performance Criteria: We focus on bounding the total cost $\sum^{\infty}_{t=0}\cl{G}_t(x_t,u_t)$ as defined in section \ref{sec:intro} of the main paper. Our notion of control objective is natural to define in control applications e.g., most popular robotic goals can be formulated as driving the systems towards a desirable set or trajectories. This differs from, but is not incompatible with, the cost metric formulation that is often seen in optimal control and online learning for control work. Specifically, previous effort on learning LQR has been on improving the regret bound of the learning algorithm \cite{dean2017sample, dean2018regret, abbasi2011regret,abbasi2018regret, hazan2019nonstochastic}. Bounding the regret on the average cost, which is natural for LQR, is not sufficient to guarantee finite mistakes in our problem setting. In section \ref{sec:regret_vs_stability}, we discuss counterexamples which discuss the relationship between finite mistakes, sublinear regret and asymptotic guarantees and show that finite mistake guarantees imply sublinear regret, however, sublinear regret does not imply finite mistakes. 
    \item Approach: Our proposed approach does not depend on accurate online system identification, which is the focus of several recent work in learning for LQR \cite{fiechter1997pac, dean2017sample,hazan2019nonstochastic,cohen2018online}. As we consider parametric uncertainty, it is plausible to also adopt system identification approach for the non-linear control settings. However, online system identification with arbitrarily small error is known to be very challenging. As shown by \cite{dahleh1993sample}, the sample complexity for identifying linear systems under bounded adversarial noises can be exponential in the worst case. 
    \item Assumptions about parameter uncertainty: Some previous work in linear systems \cite{cohen2018online,dean2018regret, hazan2019nonstochastic} assume knowledge of a stabilizing controller $\pi_{\mathrm{safe}}:\cl{X}\mapsto\cl{U}$ for the true unknown system parameter $\theta^*$. In our setting, we do not require such an assumption but merely that for each possible parameter $\theta\in\st{K}$ one can find a robust policy $\pi[\theta]$ which stabilizes the small uncertainty model $\mathbb{T}[\theta]\subset \cl{F}$.
\end{itemize}

\emph{Robust Adaptive Nonlinear Control.} Naturally, our problem setting is of great interest to the control community, and have had a relatively long history \cite{polycarpou1993robust, yao1995robust, liu2009adaptive, backstepkrstic, ioannou2006adaptive}. Yet, most of traditional adaptive control approaches can not be applied to the general problem setting we consider without making restrictive assumptions. As an example, in contrast to most adaptive control methods, our framework applies to non-feedback linearizable nonlinear system (see also discussion in \secref{subsec:adaptivenonlinearchallenge}). Furthermore, many adaptive control techniques can not build on top of methods from other areas of control, like robust control theory, but rather propose separate control algorithms for each problem setting. Also, robust stability analysis and thorough empirical validation is for most methods largely unavailable. In fact, most relevant empirical results are only presented for arguably much simpler settings than the cart-pole swing-up problem, which is considered in this work. In addition, we highlight the two methodological distinctions:
\begin{itemize}
    \item Our proposed technique takes a significantly different perspective from traditional adaptive control approach. We provide a modular framework, which allows to combine robust control tools with online learning algorithms to provide desired guarantees online.
     We unify the treatment of both uncertain system parameters and unknown disturbance via the construction of confidence sets of candidate systems that are consistent with the historical collected observations. The estimation of such consistent sets is also easily attainable for most robotic systems and allows for nonasymptotic convergence guarantee. Among relevant adaptive control literature, perhaps most closely related to ours is Multi-Model Adaptive Control (MMAC) from \cite{anderson2001multiple}. The MMAC principle needs to run a high-dimensional Multi-Estimation routine online, which requires the design of nonlinear observers (with the Matching and Detectability property - see \cite{hespanha2003overcoming}) for a sufficiently dense covering set of the parameter space. A general construction of such a family is only shown for linear systems (See \cite{hespanha2003overcoming} and references therein) and it is not clear whether designing a tractable Multi-Estimator for the cart-pole system is possible.
\end{itemize}

\section{Supplementary Details - Empirical Validation} \label{sec:app_experiment} 

We will demonstrate that the presented method can learn very challenging control tasks efficiently from limited amount of data online. We will demonstrate this on the problem of learning to swing up an inverted pendulum on a cart, also often referred to as the cart-pole system. Yet, in contrast to prior empirical results on this system, we will consider a much more challenging setting by incorporating additional constraints representing necessary restrictions that arise when interacting with real physical systems.
\begin{figure}[t]
    \vspace{-0.1in}
        \centering
        \includegraphics[width = 0.51\textwidth]{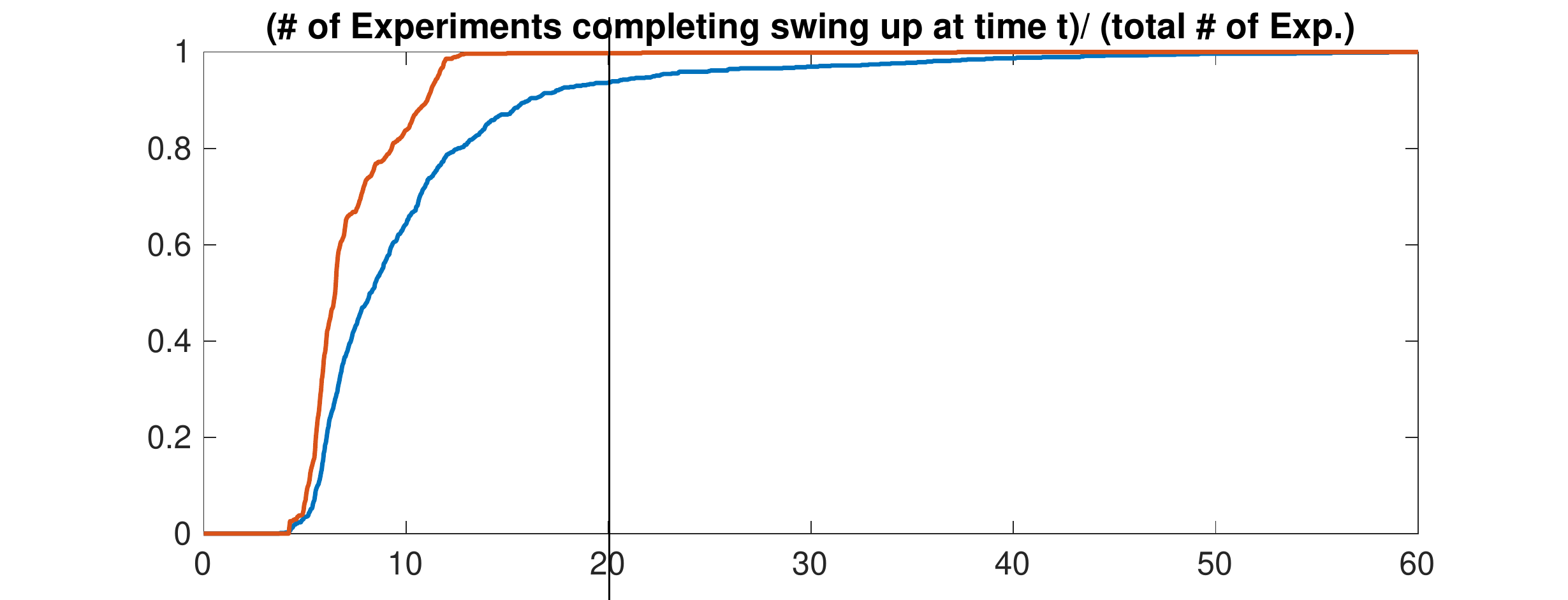}
        \includegraphics[width = 0.4\textwidth]{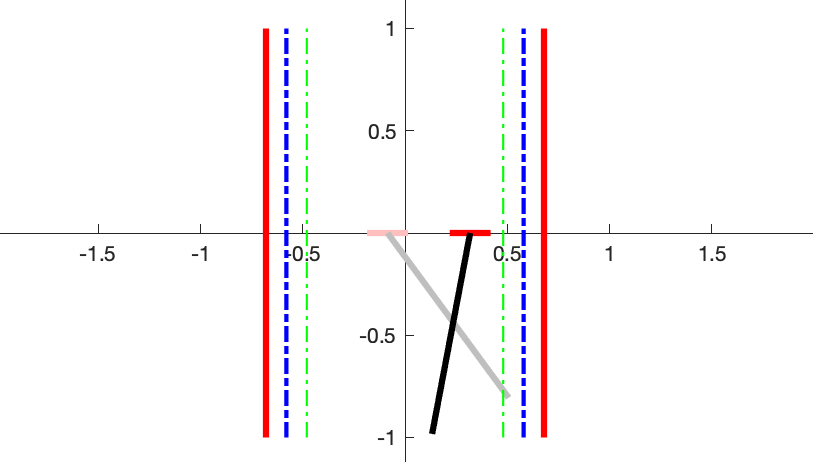}
        \vspace{-0.15in}
        \caption{{(Left) Percentage of runs that completed swing up goal before time $t$: perfect candidate controller $\pi[\theta^*]$ (red) vs. online algorithm (blue). (Right) Experiment setup.}}
        \label{fig:onlinevsoffline}
        \vspace{-0.1in}
    \end{figure}
\subsection{A hard nonlinear control task to learn: Swinging up the cart-pole system} 
Here, we will discuss the cart-pole model and the control task we will be tasking $\cl{A}_\pi(\texttt{SEL})$ to learn. In addition, we will highlight why this particular control task is a very challenging one. 

\textbf{A realistic model of the cart-pole. } 

We will model the cart-pole system in a way that is consistent with how a learning agent $\cl{A}$ would be deployed on a real dynamical system $\cl{M}$: The state $x(\tau)$ of the dynamical system $\cl{M}$ evolves in continuous-time $\tau$ according to its continuous-time system equations (usually in the form of an ODE); With respect to some fixed sampling $T_s$, the agent $\cl{A}$ receives observations at discrete-time instances $\tau_k = k T_s$ and decides actions that are kept fixed for time-window $[k T_s,(k+1) T_s]$ between sampling instances. This is depicted below in \figref{fig:realinteract}.
\begin{figure}[h]
    \begin{center}
        \fbox{\parbox{0.95\textwidth}{
            Given a discrete-time agent $\cl{A}$ and dynamical system $\cl{M}$ with state $x$ and input $u$ with continuous model $\frac{dx}{d \tau}(\tau)=g(x(\tau),u(\tau))$ in continuous time $\tau$ \\

            \textbf{At time instances} $\tau = 0, T_s, \dots, kT_s, \dots$:\\
            \phantom{AAAA}$\cl{A}$ collects observation $x(kT_s)$\\
            \phantom{AAAA}$\cl{A}$ selects action $u_k$ that is frozen for the time window $[k T_s,(k+1)T_s]$\\
            \phantom{AAAA}  $x([k T_s,(k+1)T_s]) \leftarrow $ system $\cl{M}$ evolves $\frac{dx}{d \tau}(\tau)=g(x(\tau),u(\tau))$ from $x(kT_s)$\\
            \phantom{AAAA}\phantom{AAAA}\phantom{AAAA} with input held constant $u([k T_s,(k+1)T_s]) = u_k $.
            }}
    \end{center}
    \vspace{-0.1in}
    \caption{The usual interaction protocol with real-world dynamical systems }
    \label{fig:realinteract}
    \vspace{-0.1in}
\end{figure}

The cart-pole system is governed by the following nonlinear differential equations: 
\begin{subequations}
\label{eq:cartpole}
\begin{align}
      (m_c + m_p)\ddot{x}_c + m_p l \ddot\theta_p \cos\theta_p - m_p l \dot\theta^2_p \sin\theta_p - b_x \dot{x}_c= f_u \\
      m_p l \ddot{x}_c \cos\theta_p + m_p l^2 \ddot\theta_p + m_p g l \sin\theta_p - b_{\theta}\dot{\theta} = 0
\end{align}
\end{subequations}
We are using the notation in \cite{underactuated} and refer to the same reference for detailed derivations. The variables $x_c$ and $\theta_p$ stand for cart position and pole angle in counterclockwise direction and $f_u$ represents the force exerted on the cart. Furthermore, $\theta =0$ denotes the downward position. The uncertain parameters are ($m_c$, $m_p$, $l$, $b_x$, $b_\theta$) and represent cart and pole mass, pole length, friction coefficients $b_x$ and $b_{\theta}$, respectively. $g=9.81$ is the gravity constant.

Furthermore, \eqref{eq:cartpole} can be converted into the input affine standard form 
\begin{align}\label{eq:inputaffine}
\dot{x} = F(x)+g(x)u
\end{align}
where $x = [x_c,\theta_p,\dot{x}_c,\dot{\theta}_p]^T$, $u = f_u$, (see \cite{underactuated} for description of $F(x)$ and $g(x)$). 
We will choose a \textbf{sampling-time} $\tau_s = 0.02 \mathrm{sec}$ ($1/\tau_s= 50\mathrm{Hz}$) that would be easy to realize with current technology. Under the above interaction protocol \figref{fig:realinteract}, we can abstract the transitions between discrete-time instances $x(k T_s), u(k T_s) \rightarrow x((k+1)T_s)$ as a discrete-time system. Thus, we can equivalently represent the system \eqref{eq:inputaffine} at sampling times through the discrete-time system 
\begin{align}\label{eq:dtphi}
&x_{t} = \phi_{T_s}(x_{t-1},u_{t-1}), \quad \phi_{T_s}(x,u):= \alpha(T_s),\text{ s.t. : } \dot{\alpha} = F(\alpha,u), \alpha(0) = x,
\end{align}
where we will denote $x_t := x(t T_s)$ and $u_t := u(t T_s)$ ($t \in \mathbb{N}$) to be samples of the continuous-time signals $x(\tau)$, $u(\tau)$ at time $t \tau_s$. 
We assume that \textbf{only noisy measurements} $\hat{x}_t = x_t + n_t$ are obtainable, where we assume $n_t$ to be bounded noise. Furthermore, we will \textbf{approximate knowledge of the derivative} $\dot{x}_t$ simply as $(\hat{x}_t-\hat{x}_{t-1})/T_s$. Knowledge of the exact transient function $\phi_{T_s}(\anon,\anon)$ is usually not avaiable and approximations are necessary. In simulations, we will \textbf{accurately approximate} $\phi_{T_s}(\anon,\anon)$ by using a Runge-Kutta method of order 4 with a fix step-size chosen an order of magnitude smaller than $T_s$. This is in contrast to for example OpenAI Gym \cite{openaigym}, a popular implementation of the cart-pole environment which approximates \eqref{eq:dtphi} using only the forward euler method and does not model the real physical cart-pole system accurately.


\textbf{The control task. } We will consider the \ul{swing up control task}: Starting off with $x(0)=0$, $\theta(0)=\pi$, i.e. the pole in the \textbf{downward} position and cart position in the center, choose the force $u_t$ to bring the pole to the upward position $\theta=0$ and cart position at $x=0$ and balance the system there. More specifically, by balancing we mean to keep the state $x$ within a target set  $\cl{X}_{\cl{G}}=[-\ep_x,\ep_x]\times[-\ep_\theta,\ep_\theta]\times[-\ep_{\dot{x}},\ep_{\dot{x}}]\times[-\ep_{\dot{\theta}},\ep_{\dot{\theta}}]$ of allowed tolerance, i.e. reaching and staying close to the upright cart-pole position.

\subsubsection{A challenging nonlinear control problem }
Despite being argueably one of the most popular examples of a nonlinear control problem, the cart-pole system belongs to a class of nonlinear systems that is particularly hard to control: It is \textbf{nonlinear}, \textbf{non-state feedback linearizable} and \textbf{non-minimum phase}. As a consequence, the following standard nonlinear control approaches can \textbf{not} be applied:
\begin{enumerate}[nosep]
  \item State-feedback linearization \cite{slotine1991applied, khalil2002nonlinear, sastry2013nonlinear}
  \item Input-output linearization \cite{khalil2002nonlinear, sastry2013nonlinear}: Due to non-minimum phasedness, it's hard to find a "stable" output
  \item Linearization: No linear controller can swing up and balance the cart-pole.
  \item Backstepping \cite{backstepkrstic}: The dynamics do not conform to the so-called "strict-feedback" form. 
\end{enumerate}
\begin{rem}
A related control task is the "balancing cart pole" task in the RL literature (See OpenAIGym as an example \cite{openaigym}). The "balancing task" is concerned with balancing the cart-pole system, \textbf{but} allowing the cart-pole system to start from the \textbf{upward} position. This is a much easier task. As an example of a distinguishing factor: a linear LQR controller can accomplish the balancing task, while no linear controller can perform the swingup task. 
\end{rem}
\subsubsection{A challenging adaptive nonlinear control problem }\label{subsec:adaptivenonlinearchallenge}
The previous control challenges hold, even if the model of the cart-pole is perfectly known. If we consider the online learning problem of the swingup problem, these difficulties are only exacerbated: Most standard nonlinear adaptive control techniques can not be directly applied to the problem. As an example, the following popular methods can not be used: MRAC \cite{ioannou2006adaptive}, \cite{ioannou2012robust}, nonlinear adaptive back-stepping \cite{backstepkrstic}, adaptive-sliding mode control \cite{slotine1991applied}, computed torque based methods like in \cite{OrtegaSpong89}. 
On the other hand, the more recent approach of multi-model adaptive control (MMAC) \cite{anderson2001multiple} in principle covers the cart-pole swingup problem, yet it is not clear how to exactly instatiate the approach presented in \cite{anderson2001multiple}, as one requires a design of a suiting nonlinear observer and switching logic for the cart-pole which is not clear how to do.

The cartpole system can be also described to be a so-called \textit{underactuated} robotic system. This system class often inherits the same difficulties as the one mentioned above and this recent work \cite{moore2014adaptive} \cite{nguyen2015adaptive} is discussing the difficulties and some recent progress towards adaptive control for underactuated systems. Nevertheless, both methods do not offer solutions for the adaptive swingup problem of the cart-pole.

To the best of our efforts, we have not been able to find any empirical results in the adaptive control literature that show how to learn to swingup the cart-pole system with control-theoretic guarantees.

\subsubsection{A challenging reinforcement learning problem }
To characterize the difficulty of the cart-pole swing-up task as a learning problem, we tried to estimate its sampling complexity in the context of reinforcement learning and compare it to other common control tasks often used in reinforcement learning benchmark. 
As a proxy for the sample complexity, we tuned a state-of-the-art reinforcement learning algorithm to learn the cart-pole swing up task in as few as possible samples.
We setup the swingup task as an episodic RL problem with $40\,s$ episode length, $T_s = 0.02$ and used a smooth sigmoid based dense reward function from \cite{tassa2018deepmind}. As a simulation environment we used the original OpenAIGym \cite{openaigym} cart-pole environment and modified it to fit the swingup problem we consider. 

As a representative of a state-of-the-art RL method, we used the widely successful PPO \cite{schulman2017proximal} and obtained the learning curve presented in \figref{fig:ppo2}. The learning curve presented in \figref{fig:ppo2} is the best run after approximately 200 iterations of hyperparameter tuning and modifications and it shows that despite being given a dense reward function, PPO2 needed  over $3*10^6$ time-steps, to learn how to swingup and over $4.5*10^6$ time-steps, to find an optimal policy that swings up the fastest. At a sampling-time of $T_s = 0.02\,s$, the corresponding needed interaction time with the system is $16.5$ hours and $25$ hours, respectively. In comparison, the traditional cart-pole balancing task (balancing the cart-pole where each episode starts from the \textbf{top} position) only requires PPO2 around $50000$ time-steps to learn. In the original paper \cite{schulman2017proximal}, PPO has been tested on other common RL benchmark tasks \textit{HalfCheetah-v1}, \textit{Hopper-v1}, \textit{InvertedDoublePendulum-v1}, \textit{InvertedPendulum-v1}, \textit{Reacher-v1}, \textit{Swimmer-v1}, \textit{Walker2d-v1} and has been shown to learn good policies in much fewer than $10^6$ samples. In comparison to those RL tasks, our cart-pole swing-up tasks requires (after hyperparameter tuning and optimization) PPO to take significantly more samples before it can learn a good policy that achieves the swing-up task. This is showing that although the cart-pole swing up task is arguably much less complex than other tasks in openAI gym. It is nevertheless one of the hardest to learn among these control tasks.

\begin{figure*}[h]
    \includegraphics[width=\textwidth]{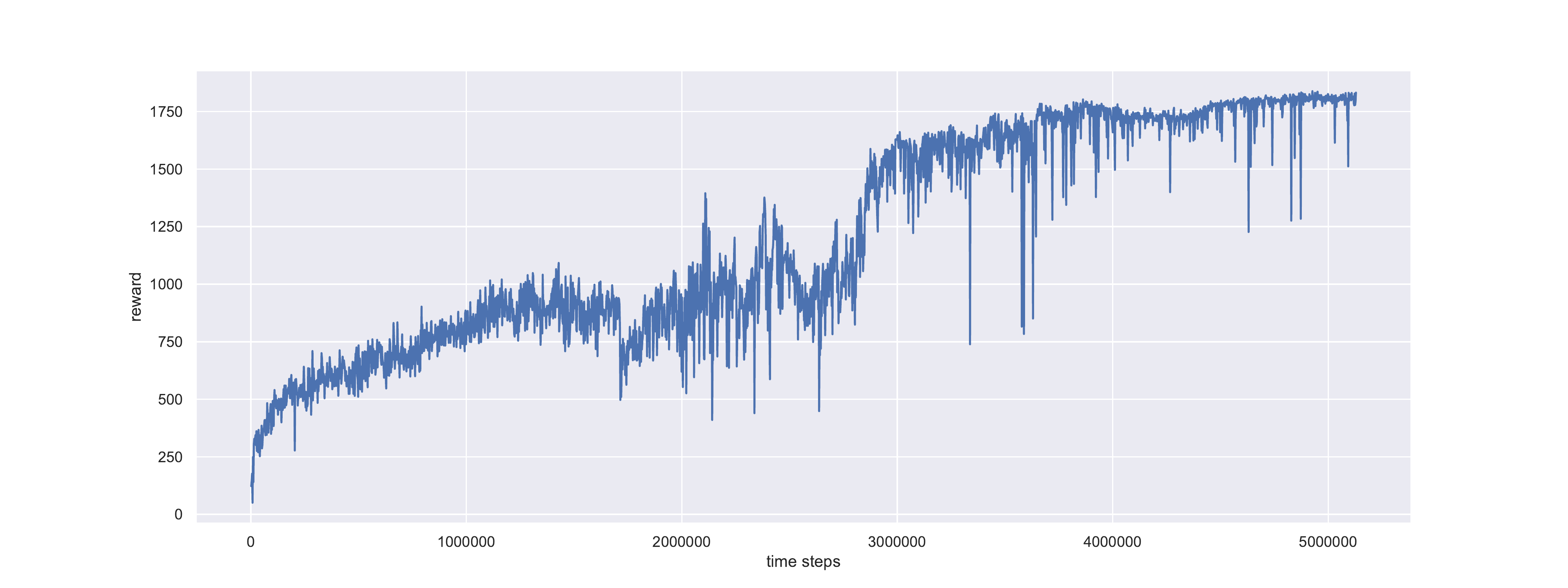}
  \vspace{-0.1in}
   \caption{ \textit{Learning curve of PPO2 for cartpole swingup:} The max episodic reward is $2000$ and a satisfactory swingup performance can be obtain starting at around a reward of $1500$.}
    \label{fig:ppo2}
   \vspace{-.1in}
  \end{figure*}


\noindent \textbf{Hyperparameters and necessary modifications.} We used the PPO2 implementation of stable baselines \cite{hill2018stable} and arrived at the hyperparameters presented in Table \ref{tab:ppo}. In addition, we had to perform additional modifications to obtain the result in \figref{fig:ppo2}:
\begin{enumerate}[nosep]
    \vspace{-0.1in}
    \item Instead of the observations $x = [x_c, \theta_p, \dot{x}_c,\dot{\theta}_p]$, we give PPO the observation vector $x' = [x_c, \cos(\theta_p), \sin(\theta_p), \dot{x}_c,\dot{\theta}_p]$ 
    \item We set $\beta_1 = 0$ of the ADAM optimizer in tensorflow
    \item We scheduled the learning rate based on the episodic reward according to Table \ref{tab:pposchedule}.
\end{enumerate}

\begin{table}[t]
    \label{tab:pposchedule}
    \caption{Learning rate schedule based on episodic reward}
    \begin{center}
    \begin{small}
    \begin{tabular}{lccccccc}
    \toprule
    episodic reward   & 700& 1300& 1400& 1500& 1600& 1800& 2001 \\
    \midrule
    learning rate   &1e-3& 1e-3& 1e-4& 1e-4& 1e-4& 1e-4& 5e-5 \\
    \bottomrule
    \end{tabular}
    \end{small}
    \end{center}
    \vskip -0.1in
\end{table}

\begin{table}[t]
    \label{tab:ppo}
    \caption{PPO hyperparameters}
    \label{par-table}
    \vskip 0.15in
    \begin{center}
    \begin{small}
    \begin{sc}
    \begin{tabular}{lc}
    \toprule
    Hyperparameter & Value  \\
    \midrule
    number of steps in epoch   & 8192  \\
    gamma & 0.999\\
    clip-range    & 0.2  \\
    num of optimization-epochs    & 10\\
    GAE-lambda    & 0.995 \\
    num of mini-batches    & 8 \\
   entropy coefficient   & 1e-6 \\
    \bottomrule
    \end{tabular}
    \end{sc}
    \end{small}
    \end{center}
    \vskip -0.1in
\end{table}

\subsection{Our validation setting: cart-pole swing-up problem subject to safety constraints}  
We test our design approach in a harder, more realistic setting of the cart-pole swingup task, to demonstrate its potential for applications with real physical systems. We enforce additional constraints on the oracle policies such as: limited noisy sensing, limited actuation range and possible interventions by backup safety controllers, etc.. In real world control applications, these difficulties have to be addressed during control design. It is therefore expected that during oracle design, in particular oracle policy design, the same issues have to be addressed. Constraining the oracle policies in this way makes online learning (especially in the single trajectory setting we consider) a harder problem, since excitation is potentially hindered. It is therefore important to verify how these constraints affect the overall performance of our learning and control meta-approach. \\

\noindent However, our experiments will show that the $\cl{A}_\pi(\sel)$ approach ensures fast learning despite imposing the following harsh constraints on the oracle policies:
\begin{enumerate}[nosep]
    \item \textbf{safety}: the oracle policies are designed to keep the cart position $x$ in an interval $[x_{min},x_{max}]$ for all time $t$ (and $x_{max}-x_{min}$ is closed to pole length). In our setting we choose $x_{max} = 0.6m$. These constraints poses a challenge for the swing up task, as the range of motion is severely limited, especially if we consider the setting $x_{max}<l$.
    \item safety controller override: The oracle policy is potentially interrupted by a safety controller (Control barrier function based) which overrides any control actions which are considered unsafe (i.e.: position to close to boundary).
     \item The oracle policies are required to use only limited acceleration $\ddot{x}\in[\ddot{x}_{min},\ddot{x}_{max}]$. In our setting we choose $x_{max} = 0.5 g$. The constraint $\ddot{x} \leq 0.5 g$ is a restrictive condition, as it can be shown in \cite{aastrom2000swinging} that this acceleration is not enough to swing-up the pole in one or two swings. 
     \item no system reset is allowed, 
     \item oracle policies can only output a force $F$ in the range of $[-200N,200N]$
\end{enumerate}
We added the above restrictions to model the following real world control challanges:
\begin{enumerate}[nosep]
  \item \textit{Safety concerns impose heavy restrictions on online exploration}: In order to guarantee safety during operation, in most cases \cite{dulac2019challenges}, strong restrictions must be imposed on the actions of the learning agent. Most commonly, safety constraints can be formulated in terms of hard and soft constraints enforced on the state $x_t$ and input $u_t$ of the system. To enforce these safety specifications online, usually one deploys a form of safety controller that is overseeing the proposed actions of the learning agent and overrides them, if they are deemed unsafe w.r.t. the safety specifications. A common approach is to implement the safety controller using control barrier functions \cite{SafetyCBFAmes2017}, \cite{GurrietAmes2018}. 
   This necessary practice causes a major challenge particularly for online learning, since exploration is potentially greatly hindered. 
  \item \textit{System resets are costly or impossible}: In many robotics application, like autonomous systems, online learning is very desirable, since resets are costly to realize or sheer impossible. Example: UAVs, autonomous cars, robots in close contact with humans. 
  \item \textit{Online data is limited and noisy} 
\end{enumerate} 

\textbf{Large system uncertainty. } We consider the initial parameter uncertainty described by the bounds in Table. \eqref{par-table}. 

Next, we describe the oracle we used to instantiate our approach $\cl{A}_\pi(\texttt{SEL}).$

\subsection{Modelbased oracle for cart-pole swing up}
Recall the nonlinear dynamics of cart pole system:
\begin{align}\label{eq:cartpole}
    (M+m)\ddot{x}-ml\ddot{\phi}\cos(\phi)+ml\dot{\phi}^2\sin(\phi)-b_x \dot{x} &= F\\
    l\ddot{\phi}-g \sin(\phi) -b_{\phi} \dot{\phi}&=\ddot{x}\cos(\phi)\nonumber
\end{align}
Let $x$ and $\dot{x}$ be the position and velocity of the cart and $\phi$, $\dot{\phi}$ the angle and angular velocity of the pole. $F$ is the force onto the cart pole and serves as our control input to the system. Furthermore, let $\bar{a}$ be the maximal cart acceleration allowed, and $\bar{d}$ be the maximum distance the cart is allowed to move from the center. \\

As a first step, we will perform so called partial feedback linearization. Let $F_d(\ddot{x},\dot{x},\phi,\dot{\phi},\dot{x})$ be the force $F$ we need to apply at time $t$ in order to achieve a desired cart acceleration of $\ddot{x}_d$. Multiply the second equation of \eqref{eq:cartpole} by $ml\cos(\phi)$ and add it to the first, to see that $F_d$ has to be chosen as:
\begin{align}
\label{eq:Fd}
    F_d(\ddot{x}_d,\dot{x},\phi,\dot{\phi},\dot{x})=(M+m\sin(\phi)^2)\ddot{x}_d - mg\cos(\phi)\sin(\phi)+ml\dot{\phi}^2\sin(\phi)-b_x \dot{x} 
\end{align}
By choosing $F=F_d(\ddot{x}_d,\dot{x},\phi,\dot{\phi},\dot{x})$, we can now treat the desired acceleration $\ddot{x}_d$ as our new control input. With respect to our new input $\ddot{x}_d$, we can simplify the original equations \eqref{eq:cartpole} to
\begin{align}\label{eq:reduce}
\ddot{x} &= \ddot{x}_d\\
l\ddot{\phi}-g \sin(\phi) -b_{\phi} \dot{\phi}&=\ddot{x}_d\cos(\phi)
\end{align}

\noindent The swingup controller consists now of three separate control laws that are later combined. 

\textbf{Around up-right position: static linear LQR controller}
If the pole has small enough kinetic energy and is close to the upright position, we simply choose $\ddot{x}_d$ to be LQR-state feedback controller based on the system \eqref{eq:reduce} linearized around the equilibrium position $x=0,\dot{x}=0,\phi = 0,\dot{\phi} = 0$. The policy then takes the form
$$\ddot{x}_{d,LQR} = - K_{LQR}(\theta)z$$

\textbf{Swing up-controller: energy-based controller.} 
The Swing-Up controller is based on an elegant energy-based approach by \cite{aastrom2000swinging}, in which we simply choose $\ddot{x}_d$ to control the total normalized energy \begin{align}\label{eq:E}
    E(\phi,\dot{\phi})= \frac{l}{2g}\dot{\phi}^2 + cos(\phi)
\end{align} of the pole. In depth derivation can be found in \cite{aastrom2000swinging} and $\ddot{x}_d$ takes the form:
\begin{align}
    \ddot{x}_{d,swing} = -\mathbf{Sat}_{\bar{a}}\left(\frac{1}{2}\gamma |\cos(\phi)|(E(\phi,\dot{\phi})-1)\mathrm{sign}(\dot{\phi}\cos(\phi))\right)
\end{align}
where $\mathbf{Sat}_{\bar{a}}$ is the saturation function which saturates at the max specified acceleration $\bar{a}$.

\textbf{Wrapping a safety controller.} 
As part of our oracle policy, we also use a control barrier function controller that prevents to trigger the safety policy. We do this simply by internally overriding our swing-up $\ddot{x}_{d,swing}$ or balancing $\ddot{x}_{d,LQR}$ terms, if we get too close to the boundary of $[-x_{max},x_{max}]$. To this end, define the $B(x,\dot{x})$ as the barrier function $$B(x,\dot{x}) = \frac{1}{2 \bar{a}} \dot{x}|\dot{x}| + x$$
and define $\ddot{\phi}_{max}:=\bar{a}g/l*\sin(30^\circ)$.

\textbf{Full controller, including safety override. }
The full controller can be described below as:
\begin{algorithm}[h]
   \caption{oracle policy $\pi[\theta]$ under potential safety policy $\pi_{\mathrm{safe}}$ override}
   \label{alg:robad_app}
\begin{algorithmic}
   \State {\bfseries Input:} $z = [x,\phi,\dot{x},\dot{\phi}]$, parameters $\theta := [M,m,l,b_x,b_\theta]$, $\epsilon_{safe}$
   \State {\bfseries Output:} $F$
   \If{$|x| < \bar{d}-\epsilon_{safe}$}
    \If{$|\dot{\phi}^2/\ddot{\phi}_{max}| <60^\circ$ \bfseries{and} $\cos(30^\circ/\ddot{\phi}_{max}+\phi\mathrm{sign}(\dot{\phi})) > \cos(30^\circ)$ \bfseries{and} $|-K_{LQR}(\theta)z|\leq \bar{a} $ }
        \State $\ddot{x}_d = -K_{LQR}(\theta)z $
    \Else
    \State $\ddot{x}_{d} = -\mathbf{Sat}_{\bar{a}}[\frac{1}{2}\gamma |\cos(\phi)|(E(\phi,\dot{\phi})-1)\mathrm{sign}(\dot{\phi}\cos(\phi))]$
    \EndIf
    \State $\ddot{x}_{d,back} = -\bar{a} \, \mathrm{sign}(\dot{x})$
    \State $\lambda = \frac{|B(x,\dot{x})|}{\bar{d}-\epsilon_{safe}} $
    \If{ $B(x,\dot{x})\geq 0$}
    \State $\ddot{x}_d \leftarrow  (1-\lambda^2)\ddot{x}_d + \lambda^2\min\{\ddot{x}_d,\ddot{x}_{d,back} \} $
    \Else 
    \State $\ddot{x}_d \leftarrow  (1-\lambda^2)\ddot{x}_d + \lambda^2\max\{\ddot{x}_d,\ddot{x}_{d,back} \} $
    \EndIf
    \State $F = F_d(\ddot{x}_d,z)$
    \Else
     \State $F = \pi_{\mathrm{safety}}(z)$
    \EndIf
\end{algorithmic}
\end{algorithm}

The controller, switches to an LQR if the system is close to the upright position and otherwise defaults to the swing up controller that brings the pendulum to the right energy level. A correction is performed to the previous control action depending on the barrier-function value $|B(x,\dot{x})|$. As $|B(x,\dot{x})|$ gets closer to the boundary $\bar{d}-\epsilon_{safe}$ the controller prioritizes to safety and overwrites the previous planned control action. If $x$ exceeds the buffer $\bar{d}-\epsilon_{safe}$, then a safe policy is being called, that bring the cart position back to the region $[-\bar{d}+\epsilon_{safe},\bar{d}-\epsilon_{safe}]$.

\subsection{Selection process $\texttt{SEL}$}
We apply the approach presented in the main paper \secref{sec:instantiation} to obtain polytopes of consistent parameters of $\st{P}_t$ for the lumped parameters $p = [m_c + m_p, m_pl, b_x, l, b_{\theta}, \tau_{d,x}, \tau_{d,\theta}]$. We use randomized LP's \cite{bubeck2020chasing}, to approximate the Steiner point of the polytope $\st{P}_t$ and select the corresponding oracle policy $\pi[\theta_t]$ as described in the meta-algorithm \ref{alg:ApiselOCO}.

\subsection{Simulation Results}

To test our algorithm in the adversarial setting we described in \secref{sec:intro}, we tested the algorithm against all $\theta^*$, which are a combinations of the test parameters given in Table. \eqref{par-table} and multiple noise random seeds, resulting in a total of $900$ experiments.

\begin{table}[t]
\caption{Initial Parameter Uncertainties and Range of Test Parameters}
\label{par-table}
\vskip 0.15in
\begin{center}
\begin{small}
\begin{sc}
\begin{tabular}{lcccr}
\toprule
Parameter & Uncertainty & Test Parameters \\
\midrule
$M$    & [0.1,5] & $\{1,2,4\}$  \\
$m$ & [0.1,1] & $\{0.1,0.2,0.4\}$ \\
$l$    & [0.05,1] & $\{0.1,0.2,0.4,0.6,1.0\}$ \\
$b_x$    & [0,20] & $\{0,10\}$\\
$b_\theta$     & [0,2] & $\{0\}$\\
\bottomrule
\end{tabular}
\end{sc}
\end{small}
\end{center}
\vskip -0.1in
\end{table}


\begin{figure*}[!t]
    \includegraphics[width=0.5\textwidth]{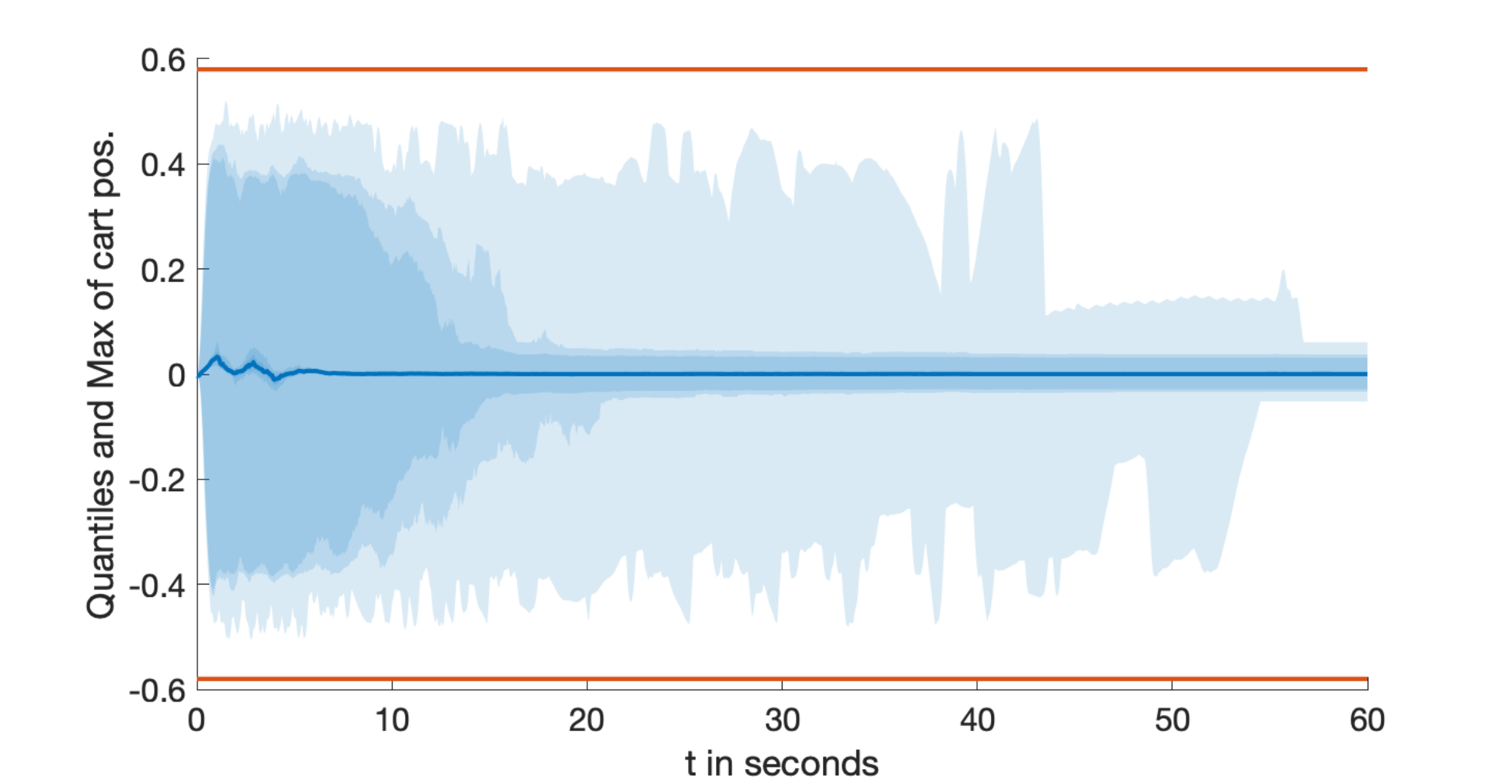}
    \includegraphics[width=0.5\textwidth]{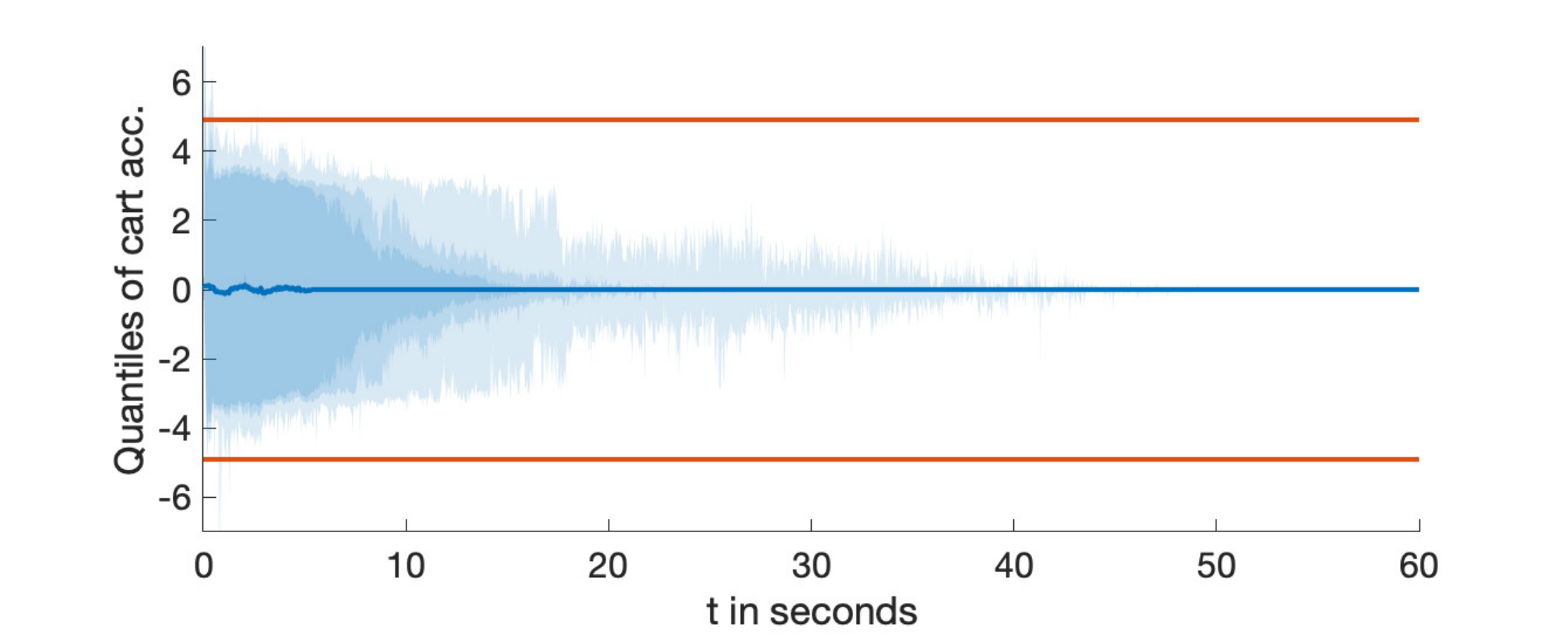}
  \vspace{-0.1in}
   \caption{ \textit{Safety and Cost Quantiles over 900 Experiments}: Left: min, max, median, top/bottom 5\% and 10\% quantiles of cart position $x_t$ over time $t$. The cart position respects the safety guarantees over all 900 experiments. Right: top/bottom 5\% and 10\% quantiles of cart acceleration $a_t$ over time $t$. The acceleration stays well within the desired acceleration bounds $0.5 g$. }
    \label{fig:ToyGraphMDPandPOMDP}
   \vspace{-.1in}
  \end{figure*}

As previously discussed, the safety policy $\pi_{\mathrm{safe}}$ overrides unsafe actions of the learning agent and thus hinders greatly exploration during online learning. As a key benefit of our approach, we show that the performance controller is practically unhindered by the safety controller: Fig. \ref{fig:onlinevsoffline} shows that in all 900 experiments, the online controller was able to learn the swing up consistently \textbf{under 60 seconds in a single episode}. In fact, in $80 \%$ of the experiments, the online algorithm learns to swing up the cart pole in \textbf{under 12 seconds.}
Comparing the online algorithm to the corresponding ideal true oracle controller $\pi[\theta^*]$, shows that the online controller $\cl{A}_\pi(\texttt{SEL})$ is only marginally slower than $\pi(\anon;\theta^*)$.
In addition to fast online learning, as shown in \figref{fig:ToyGraphMDPandPOMDP}, our approach satisfies over almost all (excluding $\approx 10$ experiments in the largest noise level setting) 900 experiments the desired safety bounds and acceleration bounds.


\section{Mistake Guarantees vs Sublinear regret}
\label{sec:regret_vs_stability}
 A common performance metric in online learning for control is phrased in terms of the regret $R(T)$. For our general problem setting, we show that sublinear regret does not imply finite mistake guarantees, however finite mistake guarantees do imply sublinear regret.
 

\subsection{Regret Definition}
Assume we are given some cost function $C:\cl{X}\times \cl{U}\mapsto \mathbb{R}^+$ and the system $x_{t+1}=f^*(t,x_t,u_t)$, $x_0=\xi_0$.
Assume that some "ideal" policy $\pi^*$ would generate the trajectory $x^*_t$, $u^*_t$, while the online algorithm $\cl{A}$ characterized by the sequence of policies, produces $x_t$, $u_t$. The optimal total cost of $J^{*}(T)$ at time $T$ is defined as $J^*(T):=\sum^{T}_{k=0} C(x^*_k,u^*_k).$
\noindent The regret $R(T)$ of $\cl{A}$ is usually refers to the sum of costs of the online algorithm up until time $T$ minus the sum of costs that the optimal policy $\pi^*$ would attain \footnote{Not the most wide-spread definition, but the most suited for adaptive control setting. See for example \cite{Zieman}}:
 \begin{align}
 R(T)&:= \sum^{T}_{k=0} C(x_k,u_k)-J^*(T).
\end{align}
Sub-linear regret is defined as follows
\begin{defn}
The regret $R(T)$ is called sublinear if $R(T) = o(T)$
or equivalently:
\begin{align}\label{eq:reg}
    && \lim_{T \rightarrow \infty} \frac{1}{T} R(T)=\lim_{T \rightarrow \infty} \frac{1}{T}\sum^{T}_{k=0} (C(x_k,u_k)-C(x^*_k,u^*_k)) &= 0, 
\end{align}
\end{defn}
\noindent The slower $R(T)$ grows with $T$, (for example $O(\log(T))$), the faster convergence we can guarantee to the above limit.

\subsection{Sub-Linear Regret does not imply bounded cost}\label{sec:asymp}
Sublinear regret is a common way to measure performance of online learning and control algorithms. Ideally, we would expect sublinear regret to subsume some more basic performance criteria such as boundedness of the online cost, i.e.: $\sup_k|C(x_k,u_k)-C(x^*_k,u^*_k)| <\infty$. However, simple derivations show that without additional assumptions, sublinear regret growth is not sufficient to show cost boundedness. The reason for that is intrinsic to the very definition of regret and simple real analysis arguments will suffice to demonstrate that.

\noindent Abbreviate $c_k : = C(x_k,u_k)$, $c^*_k : = C(x^*_k,u^*_k)$ and define the sequences
\begin{align}
    s_k &:= c_k-c^*_k\\
    m_k &:=\frac{1}{k} \sum^{k}_{j=0} (c_k-c^*_k) = \frac{1}{k} \sum^{k}_{j=1} s_k
\end{align}
Now, sublinear regret is defined as the condition $\lim\limits_{k \rightarrow \infty} m_k= 0$, while bounded cost considers the statement $\sup\limits_{k} |s_k| < \infty$.

The next counter examples, show that these statements are not related: The  first example shows that sublinear regret does \textit{not} imply finiteness of the sequence $|s_k|$; the second one shows that boundedness of $|s_k|$ does \textit{not} imply sublinear regret.

\begin{enumerate}
    \item $\lim\limits_{k \rightarrow \infty} m_k= 0 \centernot\implies  \sup\limits_{k} |s_k| < \infty$
    \begin{proof}
    Consider $s_k$, where $s_{e^{n}} = n$ and otherwise $0$. Define $\bar{n}(k):=\lfloor \log(k)\rfloor$, then 
    $$m_k \leq m_{\bar{n}(k)} = \frac{1}{e^{\bar{n}(k)}} \sum^{\bar{n}(k)}_{i=1} i = \frac{\bar{n}(k)(\bar{n}(k)+1)}{2e^{\bar{n}(k)}}.$$
    This shows $\lim_{k \rightarrow \infty} m_{\bar{n}(k)} = 0$, but $s_k$ is unbounded.
    \end{proof}
    \item For all $\varepsilon>0$: $\sup_{k} s_k - \inf_{k}s_k < \varepsilon \centernot\implies $ $m_k$ converges 
    \begin{proof}
    Define $s_k$ as the sequence
    \begin{align}
    s_k= (\underbrace{\underbrace{\underbrace{1,\delta,1,1,\delta,\delta}_{6},\underbrace{1,\dots,1}_{6},\underbrace{\delta,\dots,\delta}_{6}}_{18},\underbrace{1,\dots,1}_{18},\underbrace{\delta,\dots,\delta}_{18 },\dots,}_{2\times 3^n} \underbrace{1,\dots,1}_{2\times 3^n},\underbrace{\delta,\dots,\delta,}_{2\times 3^n }\dots  ),
    \end{align}
    with $\delta =1- \varepsilon$. From the above pattern, it becomes apparent that the corresponding $m_k$ satisfies for all $n$:
    \begin{align}
       & m_{2\times 3^{n}} = 1 - \varepsilon/2  & m_{4\times 3^{n}} = 1 - \varepsilon/4, 
    \end{align} 
    hence $m_k$ doesn't converge, yet $\sup_{k} s_k - \inf_{k}s_k < \varepsilon$.
    \end{proof}
\end{enumerate}

\begin{rem} The above arguments still hold if we change $s_k$ and $m_k$ to the definitions
\begin{align*}
    &s'_k := |C(x_k,u_k)-C(x^*_k,u^*_k)|
    &m'_k :=\frac{1}{k} \sum^{k}_{j=0} |C(x_k,u_k)-C(x^*_k,u^*_k)| = \frac{1}{k} \sum^{k}_{j=1} s'_k
\end{align*}
\end{rem}

\subsection{Sublinear regret does not imply finite mistakes, finite mistakes imply sublinear regret}
In our problem setting, the costs $C(x,u)$ are represented by $\cl{G}(x,u)$, which are $\{0,1\}$-valued cost functions. Moreover, we compare to an oracle-policy $\pi^*$ which guarantees at most $M^\pi_\rho$ mistakes i.e.: $\sum^\infty_{k=0} c^*_k \leq M^\pi_\rho$. 
Trivially, finite mistakes implies sublinear regret: If it holds that $\sum^{\infty}_{k=0} c_k < M$, then $\tfrac{1}{T}\sum^{T}_{k=0} (c_k-c^*_k) \leq \tfrac{M}{T} \rightarrow 0$. The opposite is however not true. Consider as an example the following sequence for $c_k$:
$$\bm{c}=\underbrace{\:0\:,\:1\:}_{2},\underbrace{\:0\:,\:0\:,\:0\:, \:1\:}_{4},\underbrace{ \:0\:, \:0\: , \:0\: , \:0\: , \:0\: , \:0\: , \:0\:  , \:1\:}_{8}, \dots,\underbrace{\:0\:,\dots,\:1\:}_{2^k},\dots$$
The total mistakes $\sum^{T}_{t=1}c_k = \cl{O}(\log(T))$ grow unbounded but we still have $\lim_{t\rightarrow \infty} m_k =0$.

\end{document}